\documentclass[reqno]{amsart}

\usepackage{amsmath}
\usepackage{amssymb} 
\usepackage{amsthm}
\usepackage{cite} 
\usepackage{hyperref}
\usepackage{bbm} 
\usepackage{graphicx} 

\newtheorem{theorem}{Theorem}
\newtheorem{corollary}{Corollary}
\newtheorem{lemma}{Lemma}[section]
\newtheorem{proposition}{Proposition}
\theoremstyle{definition}
\newtheorem{definition}{Definition}[section]
\numberwithin{equation}{section}

\newcommand\gl{\operatorname{gl}}
\newcommand\diag{\operatorname{diag}}
\newcommand\Real{\operatorname{Re}}
\newcommand\Imag{\operatorname{Im}}
\newcommand\adj{\operatorname{adj}}
\newcommand\Arg{\operatorname{Arg}}
\newcommand\Spec{\operatorname{Spec}}

\allowdisplaybreaks

\title[Spectral analysis for well-posed beam deflection]
 {Spectral analysis for the class of integral operators arising from
well-posed boundary value problems
of finite beam deflection on elastic foundation:
characteristic equation}

\author[S. W. Choi]{Sung Woo Choi}
\address{Sung Woo Choi \\ Department of Mathematics \\ Duksung Women's University \\ Seoul 01369, Korea}
\email{swchoi@duksung.ac.kr}


\begin{document}

\begin{abstract}
We consider the boundary value problem for the deflection of
a finite beam on an elastic foundation subject to vertical loading.
We construct a one-to-one correspondence $\Gamma$ 
from the set of equivalent well-posed two-point boundary conditions
to $\mathrm{gl}(4,\mathbb{C})$.
Using $\Gamma$,
we derive eigenconditions for the integral operator $\mathcal{K}_\mathbf{M}$
for each well-posed two-point boundary condition represented by
$\mathbf{M} \in \mathrm{gl}(4,8,\mathbb{C})$.
Special features of our eigenconditions include;
(1) they isolate the effect of the boundary condition $\mathbf{M}$ 
on $\mathrm{Spec}\,\mathcal{K}_\mathbf{M}$,
(2) they connect $\mathrm{Spec}\,\mathcal{K}_\mathbf{M}$
to $\mathrm{Spec}\,\mathcal{K}_{l,\alpha,k}$
whose structure has been well understood.
Using our eigenconditions,
we show that, for each nonzero real $\lambda \not \in \mathrm{Spec}\,\mathcal{K}_{l,\alpha,k}$,
there exists a real well-posed boundary condition $\mathbf{M}$ such that 
$\lambda \in \mathrm{Spec}\,\mathcal{K}_\mathbf{M}$.
This in particular shows that the integral operators $\mathcal{K}_\mathbf{M}$
arising from well-posed boundary conditions,
may not be positive nor contractive in general,
as opposed to $\mathcal{K}_{l,\alpha,k}$.
\end{abstract}

\maketitle

\section{Introduction}

We consider the boundary value problem for
the vertical deflection of a linear-shaped beam of finite length $2l$ resting 
horizontally
on an elastic foundation,
while the beam is subject to a vertical loading.
Due to its wide range of applications,
this problem has been one of the main topics in mechanical
engineering for decades~\cite{
Beaufait1980,
Choi2020nonlinear,
ChoiJang,
Galewski2011,
Hetenyi1946,
Kuoetal1994, 
Maetal2009,
Mirandaetal1966,
Timoshenko1953,
Ting1982}.
By the classical Euler beam theory~\cite{Timoshenko1953}, 
the upward vertical beam deflection $u(x)$
satisfies the following linear fourth-order
ordinary differential equation.
\begin{equation}
\label{equation_linearnonhomogeneous-original}
E I \cdot u^{(4)}(x) + k \cdot u(x) = w(x),
\qquad
x \in [-l,l].
\end{equation}
Here, $k$
is the spring constant density of the elastic foundation,
and
$w(x)$ is the downward load density
applied vertically on the beam.
The constants $E$ and $I$ are the Young's modulus
and the mass moment of inertia respectively, 
so that $EI$ is the flexural rigidity of the beam.
Denoting
$
\alpha
 =
\sqrt[4]{k/E I}
 >
0
$,
we transform \eqref{equation_linearnonhomogeneous-original}
into the following equivalent form,
which we call $\mathrm{DE}(w)$.
\begin{equation}
\label{equation_linearnonhomogeneous}
\mathrm{DE}(w):
u^{(4)} + \alpha^4 u
 =
\frac{\alpha^4}{k} \cdot w.
\end{equation}
Throughout this paper,
we will assume that $l$, $\alpha$, $k$ are
fixed positive constants.
The homogeneous version
of \eqref{equation_linearnonhomogeneous} is
\begin{equation}
\label{equation_linearhomogeneous}
\mathrm{DE}(0):
u^{(4)} + \alpha^4 u
 =
0.
\end{equation}

Let $\gl(m,n,\mathbb{C})$ (respectively, $\gl(m,n,\mathbb{R})$)
be the set of $m \times n$ matrices with complex (respectively, real) entries.
When $m = n$, we denote
$
\gl(n,\mathbb{C})
 =
\gl(n,n,\mathbb{C})
$
and
$
\gl(n,\mathbb{R})
 =
\gl(n,n,\mathbb{R})
$.
Define the following linear operator
$\mathcal{B} : C^3[-l,l] \to \gl(8,1,\mathbb{C})$ by
\begin{equation}
\label{equation_calB}
\mathcal{B}[u]
 =
\begin{pmatrix}
u(-l) &
u^\prime(-l) &
u^{\prime\prime}(-l) &
u^{(3)}(-l) &
u(l) &
u^\prime(l) &
u^{\prime\prime}(l) &
u^{(3)}(l)
\end{pmatrix}^T,
\end{equation}
where $C^n[-l,l]$ is the space of $n$ times differentiable complex-valued functions on the interval $[-l,l]$. 
Then any {\em two-point boundary condition} can formally be given
with a $4 \times 8$ matrix $\mathbf{M} \in \gl(4,8,\mathbb{C})$ 
and a $4 \times 1$ matrix
$\mathbf{b} \in \gl(4,1,\mathbb{C})$ as follows.
\begin{equation}
\label{equation_BCb}
\mathbf{M}
\cdot
\mathcal{B}[u]
 =
\mathbf{b}.
\end{equation}
For example,
the boundary condition
$
u(-l) = u_-
$,
$
u^\prime(-l) = u_-^\prime
$,
$
u(l) = u_+
$,
$
u^\prime(l) = u_+^\prime
$
corresponds to the case when
\[
\mathbf{M}
 =
\left(
\begin{array}{cccc|cccc}
1 & 0 & 0 & 0 & 0 & 0 & 0 & 0 \\
0 & 1 & 0 & 0 & 0 & 0 & 0 & 0 \\
0 & 0 & 0 & 0 & 1 & 0 & 0 & 0 \\
0 & 0 & 0 & 0 & 0 & 1 & 0 & 0
\end{array}
\right),
\quad
\mathbf{b}
 =
\begin{pmatrix}
u_- \\
u_-^\prime \\
u_+ \\
u_+^\prime
\end{pmatrix}.
\]
The {\em homogeneous boundary condition} associated to \eqref{equation_BCb},
which we denote by $\mathrm{BC}(\mathbf{M})$, is
\begin{equation}
\label{equation_BC0}
\mathrm{BC}(\mathbf{M}):
\mathbf{M} \cdot \mathcal{B}[u] = \mathbf{0},
\end{equation}
where
$
\mathbf{0}
 =
\begin{pmatrix}
0 & 0 & 0 & 0
\end{pmatrix}^T$.
The boundary value problem consisting of the nonhomogeneous equation
$\mathrm{DE}(w)$
and
the boundary condition
\eqref{equation_BCb}
is {\em well-posed},
if it has a unique solution.
In fact, it is easy to see that 
this boundary value problem is well-posed
for {\em any} fixed $w$ and $\mathbf{b}$,
if and only if
the boundary value problem consisting of the homogeneous equation
$\mathrm{DE}(0)$
and the homogeneous boundary condition
$\mathrm{BC}(\mathbf{M})$
is well-posed,
in which case we will just call $\mathbf{M} \in \gl(4,8,\mathbb{C})$ {\em well-posed}.
We denote the set of all well-posed matrices in $\gl(4,8,\mathbb{C})$
by $\mathrm{wp}(4,8,\mathbb{C})$.

It is well-known from the classical Green's function theory~\cite{Stakgold}
that, for each well-posed $\mathbf{M} \in \mathrm{wp}(4,8,\mathbb{C})$,
there exists a unique function $G_\mathbf{M}(x,\xi)$ defined on $[-l,l] \times [-l,l]$,
called the {\em Green's function} corresponding to $\mathbf{M}$,
such that the unique solution of the boundary value problem consisting of
$\mathrm{DE}(w)$ and $\mathrm{BC}(\mathbf{M})$ is given by
\[
\mathcal{K}_\mathbf{M}[w]
 =
\int_{-l}^l G_\mathbf{M}(x,\xi) w(\xi) \, d\xi
\]
for every continuous function $w$ on $[-l,l]$.
The integral operator $\mathcal{K}_\mathbf{M}$
becomes a compact linear operator
on the Hilbert space $L^2[-l,l]$ of complex-valued square-integrable functions on $[-l,l]$.
Analyzing the structure of
the {\em spectrum} $\Spec\mathcal{K}_\mathbf{M}$,
or the set of eigenvalues,
of the operator $\mathcal{K}_\mathbf{M}$,
is of paramount importance for understanding the boundary value problem
represented by given well-posed $\mathbf{M} \in \mathrm{wp}(4,8,\mathbb{C})$.

We call $\mathbf{M}, \mathbf{N} \in \mathrm{wp}(4,8,\mathbb{C})$
{\em equivalent}, and denote $\mathbf{M} \approx \mathbf{N}$,
when $\mathcal{K}_\mathbf{M} = \mathcal{K}_\mathbf{N}$,
or equivalently,
when $G_\mathbf{M} = G_\mathbf{N}$.
For given $\mathbf{M} \in \mathrm{wp}(4,8,\mathbb{C})$,
denote by
$\left[ \mathbf{M} \right]$
the equivalence class with respect to $\approx$ containing $\mathbf{M}$.
The set of all these equivalence classes,
which is the quotient set $\mathrm{wp}(4,8,\mathbb{C})/\!\!\approx$
of $\mathrm{wp}(4,8,\mathbb{C})$ by the relation $\approx$,
is denoted simply by $\mathrm{wp}(\mathbb{C})$.

In \cite{ChoiBKMS2015,ChoiII},
Choi analyzed
an integral operator 
$
\mathcal{K}_{l,\alpha,k}
 =
\mathcal{K}_l
$
on $L^2[-l,l]$
defined by
\begin{equation}
\label{equation_Klalphak}
\mathcal{K}_{l,\alpha,k}[w](x)
 =
\int_{-l}^l
G(x,\xi)
w(\xi)
\, d\xi,
\end{equation}
where
\begin{equation}
\label{equation_Glalphak}
G(x,\xi)
 =
\frac{\alpha}{2k}
\exp
\left(
-
\frac{\alpha}{\sqrt{2}}
\left| x - \xi \right|
\right)
\sin
\left(
\frac{\alpha}{\sqrt{2}}
\left| x - \xi \right|
+
\frac{\pi}{4}
\right)
\end{equation}
is the Green's function of the boundary value problem
consisting of $\mathrm{DE}(0)$ and 
the boundary condition 
$
\lim_{x \to \pm \infty}
u(x)
 =
0
$
for an {\em infinitely long} beam.
It turns out that 
$
\mathcal{K}_{l,\alpha,k}
 =
\mathcal{K}_\mathbf{Q}$
in our terminology,
where
\begin{equation}
\label{equation_Q}
\mathbf{Q}
 =
\left(
\begin{array}{cccc|cccc}
0 & \alpha^2 & -\sqrt{2} \alpha & 1 & 0 & 0 & 0 & 0 \\
\sqrt{2} \alpha^3 & -\alpha^2 & 0 & 1 & 0 & 0 & 0 & 0 \\
0 & 0 & 0 & 0 & 0 & \alpha^2 & \sqrt{2} \alpha & 1 \\
0 & 0 & 0 & 0 & -\sqrt{2} \alpha^3 & -\alpha^2 & 0 & 1
\end{array}
\right)
\end{equation}
in $\mathrm{wp}(4,8,\mathbb{C})$.
What is special about this particular operator $\mathcal{K}_\mathbf{Q}$ is that
its spectrum is exceptionally well-understood.
In Proposition~\ref{proposition_Q} below, 
$h$ is an explicitly defined strictly increasing function from $[0,\infty)$ to itself
such that $h(0) = 0$ and $\lim_{t \to \infty}{h(t)/t} = L$,
where $L = 2 l \alpha$ is the dimensionless constant called 
the {\em intrinsic length} of the beam.
For two nonnegative functions $f$, $g$ defined either on $[0,\infty)$ or on $\mathbb{N}$,
denote $f(t) \sim g(t)$,
if there exists $T > 0$ such that 
$m \leq f(t)/g(t) \leq M$ for every $t > T$
for some constants $0 < m \leq M < \infty$.
Thus $h^{-1}(t) \sim t/L$ with this notation.

\begin{proposition}[\cite{ChoiII}]
\label{proposition_Q}
The spectrum $\Spec\mathcal{K}_\mathbf{Q}$ of the operator 
$
\mathcal{K}_\mathbf{Q}
 =
\mathcal{K}_{l,\alpha,k}
$
is of the form
\[
\left\{\left. 
\frac{\mu_n}{k} \,\right|\, n = 1,2,3,\cdots 
\right\}
\cup
\left\{\left. \frac{\nu_n}{k} \,\right|\, n = 1,2,3,\cdots \right\}
\subset
\left( 0, \frac{1}{k} \right),
\]
where
$\mu_n$ and $\nu_n$ for $n = 1,2,3,\ldots$ depend only on 
the intrinsic length $L$ of the beam.
$\mu_n \sim \nu_n \sim n^{-4}$,
and
\begin{gather*}
\frac{1}
{
1
+
\left\{
  h^{-1}\left( 2\pi n + \frac{\pi}{2} \right)
\right\}^4
}
 <
\nu_n
 <
\frac{1}
{
1
+
\left\{
  h^{-1}\left( 2\pi n \right)
\right\}^4
} \\
 <
\mu_n
 <
\frac{1}
{
1
+
\left\{
h^{-1}\left( 2\pi n - \frac{\pi}{2} \right)
\right\}^4
},
\quad
n = 1,2,3,\ldots, \\
\frac{1}
{1 + \left\{ h^{-1}\left( 2\pi n - \frac{\pi}{2} \right) \right\}^4}
-
\mu_n
 \sim
\nu_n
-
\frac{1}
{1 + \left\{ h^{-1}\left( 2\pi n + \frac{\pi}{2} \right) \right\}^4}
 \sim
n^{-5} e^{-2\pi n}, \\
\frac{1}
{
1
+
\frac{1}
{L^4}
\left( 2\pi (n-1) - \frac{\pi}{2} \right)^4
}
-
\mu_n
 \sim
\frac{1}
{
1
+
\frac{1}
{L^4}
\left( 2\pi (n-1) + \frac{\pi}{2} \right)^4
}
-
\nu_n
 \sim
n^{-6}.
\end{gather*}
\end{proposition}

In fact, numerical values of $\mu_n$ and $\nu_n$
can be computed
with arbitrary precision for any given $L > 0$.
See \cite{ChoiII} for more details.

In this paper,
we will construct the Green's function $G_\mathbf{M}$ explicitly
for {\em every} $\mathbf{M} \in \mathrm{wp}(4,8,\mathbb{C})$.
As a result,
we construct an explicit map $\mathrm{wp}(4,8,\mathbb{C}) \to \gl(4,\mathbb{C})$,
$\mathbf{M} \mapsto \mathbf{G}_\mathbf{M}$,
in such a way that
$
\mathbf{G}_\mathbf{M}
 =
\mathbf{G}_\mathbf{N}
$,
if and only if
$\mathbf{M} \approx \mathbf{N}$.
This induces a map
$
\Gamma:
\mathrm{wp}(\mathbb{C}) \to \gl(4,\mathbb{C})
$,
where $\Gamma\left( \left[ \mathbf{M} \right] \right) = \mathbf{G}_\mathbf{M}$
for $\mathbf{M} \in \mathrm{wp}(4,8,\mathbb{C})$.
Especially, our construction of the map $\Gamma$ has the following features.

\begin{itemize}
\item[$(\Gamma1)$]
$\Gamma$ is a one-to-one correspondence from
$\mathrm{wp}(\mathbb{C})$ to $\gl(4,\mathbb{C})$.

\item[$(\Gamma2)$]
$
\Gamma\left( \left[ \mathbf{Q} \right] \right)
 =
\mathbf{O}
$.

\item[$(\Gamma3)$]
$\Gamma$ is {\em constructive},
in that
$\Gamma\left( \left[ \mathbf{M} \right] \right)$ can be computed explicitly
for any given $\mathbf{M} \in \mathrm{wp}(4,8,\mathbb{C})$,
and conversely,
a representative of $\Gamma^{-1}\left( \mathbf{G} \right)$ in $\mathrm{wp}(4,8,\mathbb{C})$
can be computed explicitly
for any given $\mathbf{G} \in \gl(4,\mathbb{C})$.
\end{itemize}

By $(\Gamma1)$,
$\Gamma$ can be regarded as a
{\em faithful representation}
of $\mathrm{wp}(\mathbb{C})$
by the algebra $\gl(4,\mathbb{C})$.
$(\Gamma2)$ says that
$\Gamma$ is constructed to incorporate the special boundary condition $\mathbf{Q}$.
This will enable us in Theorem~\ref{theorem_eigencondition-X}
and Corollaries~\ref{corollary_eigencondition-det}, \ref{corollary_eigencondition-Y} below
to obtain an eigencondition and
characteristic equations for the operator $\mathcal{K}_\mathbf{M}$,
which connect $\Spec\mathcal{K}_\mathbf{M}$
for general $\mathbf{M} \in \mathrm{wp}(4,8,\mathbb{C})$ to the well-understood
$\Spec\mathcal{K}_\mathbf{Q}$ in Proposition~\ref{proposition_Q}.
$(\Gamma3)$ means that
our eigencondition and characteristic equations
for $\mathcal{K}_\mathbf{M}$ are constructed explicitly
for each given $\mathbf{M} \in \mathrm{wp}(4,8,\mathbb{C})$.
Conversely, whenever you find a class of matrices in $\gl(4,\mathbb{C})$
with which you can say something about 
the corresponding eigencondition in Theorem~\ref{theorem_eigencondition-X},
you can translate them back to the corresponding boundary conditions explicitly.
In fact, this is exactly what we do in Theorem~\ref{theorem_existence} below.

Among well-posed boundary conditions in $\mathrm{wp}(4,8,\mathbb{C})$,
ones with {\em real} entries are of particular interest.
We denote by $\mathrm{wp}(4,8,\mathbb{R})$,
the set of well-posed matrices in $\mathrm{wp}(4,8,\mathbb{C})$
with real entries.
The set of all equivalence classes $\left[ \mathbf{M} \right]$ in $\mathrm{wp}(\mathbb{C})$
such that $\mathbf{M} \approx \mathbf{N}$ for some $\mathbf{N} \in \mathrm{wp}(4,8,\mathbb{R})$,
is denoted by $\mathrm{wp}(\mathbb{R})$.
To characterize $\mathrm{wp}(\mathbb{R})$,
we introduce an $\mathbb{R}$-algebra $\overline{\pi}(4)$
contained in $\gl(4,\mathbb{C})$.
With $\overline{\pi}(4)$,
we have a faithful representation of
$\mathrm{wp}(\mathbb{R})$,
which is another feature of $\Gamma$.

\begin{itemize}
\item[$(\Gamma4)$]
$
\Gamma\left( \mathrm{wp}(\mathbb{R}) \right)
 =
\overline{\pi}(4)$.
\end{itemize}

The usefulness of
$\overline{\pi}(4)$ is not just limited to characterizing $\mathrm{wp}(\mathbb{R})$.
The $\mathbb{R}$-algebra $\overline{\pi}(4)$ 
is designed to measure an important symmetry of $4 \times 4$ matrices,
which is utilized in proving Theorem~\ref{theorem_existence}.

Using our representation $\Gamma$,
we prove Theorem~\ref{theorem_eigencondition-X} below.
Here, $\mathbf{X}_\lambda(x) \in \gl(4,\mathbb{C})$ 
and 
$\mathbf{y}_\lambda(x) \in \gl(4,1,\mathbb{C})$
will be defined explicitly 
for every $\lambda \in \mathbb{C} \setminus \{ 0 \}$ and $x \in \mathbb{R}$ in Section~\ref{section_X}.
Note that the second statement follows immediately from the first one and $(\Gamma2)$ above.

\begin{theorem}
\label{theorem_eigencondition-X}
Let $\mathbf{M} \in \mathrm{wp}(4,8,\mathbb{C})$,
$0 \neq u \in L^2[-l,l]$, 
and
$\lambda \in \mathbb{C}$.
Then $\mathcal{K}_\mathbf{M}[u] = \lambda \cdot u$,
if and only if
$\lambda \neq 0$
and
there exists 
$
\mathbf{0}
 \neq
\mathbf{c}
\in \gl(4,1,\mathbb{C})
$
such that
$
u
 =
\mathbf{y}_\lambda^T
\mathbf{c}
$
and
$
\left[
\mathbf{G}_\mathbf{M}
\left\{
\mathbf{X}_\lambda(l)
-
\mathbf{X}_\lambda(-l)
\right\}
+
\mathbf{X}_\lambda(l)
\right]
\mathbf{c}
 =
\mathbf{0}
$.
$\mathcal{K}_\mathbf{Q}[u] = \lambda \cdot u$,
if and only if
$\lambda \neq 0$
and
there exists 
$
\mathbf{0}
 \neq
\mathbf{c}
\in \gl(4,1,\mathbb{C})
$
such that
$
u
 =
\mathbf{y}_\lambda^T
\mathbf{c}
$
and
$
\mathbf{X}_\lambda(l)
\cdot
\mathbf{c}
 =
\mathbf{0}
$.
\end{theorem}

Thus, if we focus on the spectrum $\Spec\mathcal{K}_\mathbf{M}$,
we have the following characteristic equation.

\begin{corollary}
\label{corollary_eigencondition-det}
Let $\mathbf{M} \in \mathrm{wp}(4,8,\mathbb{C})$ and $\lambda \in \mathbb{C}$.
Then $\lambda \in \Spec\mathcal{K}_\mathbf{M}$,
if and only if
$\lambda \neq 0$
and
$
\det
\left[
\mathbf{G}_\mathbf{M}
\left\{
\mathbf{X}_\lambda(l)
-
\mathbf{X}_\lambda(-l)
\right\}
+
\mathbf{X}_\lambda(l)
\right]
 =
0
$.
$\lambda \in \Spec\mathcal{K}_\mathbf{Q}$,
if and only if
$\lambda \neq 0$
and
$
\det
\mathbf{X}_\lambda(l)
 =
0
$.
\end{corollary}

Theorem~\ref{theorem_eigencondition-X} and Corollary~\ref{corollary_eigencondition-det}
reveal an interesting connection between $\Spec\mathcal{K}_\mathbf{M}$
for general $\mathbf{M} \in \mathrm{wp}(4,8,\mathbb{C})$
and the well-analyzed $\Spec\mathcal{K}_\mathbf{Q}$.
The forms of the eigencondition and the characteristic equation for $\mathcal{K}_\mathbf{M}$ in them
isolate the effect $\mathbf{G}_\mathbf{M}$ 
of the boundary condition $\mathbf{M} \in \mathrm{wp}(4,8,\mathbb{C})$,
from the rest that is expressed essentially by
the matrix $\mathbf{X}_\lambda$ which is closely related to
$\Spec\mathcal{K}_\mathbf{Q}$.

By Corollary~\ref{corollary_eigencondition-det}, $\mathbf{X}_\lambda(l)$ is invertible
for every $0 \neq \lambda \not\in \Spec\mathcal{K}_\mathbf{Q}$.
Thus
we can define
$
\mathbf{Y}_\lambda(l)
 =
\mathbf{X}_\lambda(-l)
\mathbf{X}_\lambda(l)^{-1}
-
\mathbf{I}
\in \gl(4,\mathbb{C})
$
for every $0 \neq \lambda \not\in \Spec\mathcal{K}_\mathbf{Q}$,
where $\mathbf{I}$ is the $4 \times 4$ identity matrix.

\begin{corollary}
\label{corollary_eigencondition-Y}
Let $\mathbf{M} \in \mathrm{wp}(4,8,\mathbb{C})$.
Suppose $\lambda \in \mathbb{C} \setminus \Spec\mathcal{K}_\mathbf{Q}$.
Then $\lambda \in \Spec\mathcal{K}_\mathbf{M}$,
if and only if
$\lambda \neq 0$
and
$
\det
\left\{
\mathbf{G}_\mathbf{M}
\mathbf{Y}_\lambda(l)
-
\mathbf{I}
\right\}
 =
0
$.
\end{corollary}

Let
$\mathbf{M} \in \mathrm{wp}(4,8,\mathbb{C})$.
We call the dimensionless quantity
$
k
\cdot
\left\|
\mathcal{K}_\mathbf{M}
\right\|_2
$
the {\em intrinsic $L^2$-norm}
of $\mathcal{K}_\mathbf{M}$,
where
$
\left\|
\mathcal{K}_\mathbf{M}
\right\|_2
$
is the usual {\em $L^2$-norm}
of
$\mathcal{K}_\mathbf{M}$,
which is 
equal to the {\em spectral radius} 
$
\max
\left\{
|\lambda|
:
\lambda \in \Spec\mathcal{K}_\mathbf{M}
\right\}
$
of
$\mathcal{K}_\mathbf{M}$.
For each $\lambda \in \Spec\mathcal{K}_\mathbf{M}$,
we call the dimensionless quantity $k \cdot \lambda$ an {\em intrinsic eigenvalue}.
By Proposition~\ref{proposition_Q},
the operator $\mathcal{K}_\mathbf{Q}$
is {\em positive}
in that all of its intrinsic eigenvalues are positive,
and is {\em contractive}
in that its intrinsic $L^2$-norm,
which equals to its largest intrinsic eigenvalue $\mu_1$,
is less than $1$.
Since these properties of $\mathcal{K}_\mathbf{Q}$
are important in analyzing nonlinear non-uniform problem corresponding to 
$\mathrm{DE}(w)$ in \eqref{equation_linearnonhomogeneous}~\cite{Choi2020nonlinear,ChoiJang},
one immediate question
is whether or not they are also shared by other general $\mathcal{K}_\mathbf{M}$.
Using Corollary~\ref{corollary_eigencondition-Y},
we prove the following negative answer.

\begin{theorem}
\label{theorem_existence}
For each 
$
0
 \neq
\lambda 
\in \mathbb{R} \setminus \Spec\mathcal{K}_\mathbf{Q}$,
there exists $\mathbf{M} \in \mathrm{wp}(4,8,\mathbb{R})$
such that $\lambda \in \Spec\mathcal{K}_\mathbf{M}$.
\end{theorem}

Thus, one cannot expect $\mathcal{K}_\mathbf{M}$
to be positive nor contractive even for {\em real} $\mathbf{M} \in \mathrm{wp}(4,8,\mathbb{R})$.
This result,
which shows the diversity of general well-posed boundary conditions,
and might have been tricky to obtain otherwise,
demonstrates the usefulness of our presentation of the eigencondition and the characteristic equations
for operators $\mathcal{K}_\mathbf{M}$ 
with general $\mathbf{M} \in \mathrm{wp}(4,8,\mathbb{C})$.

The rest of the paper is organized as follows.
In Section~\ref{section_preliminaries},
we present basic mathematical terminologies we use,
and introduce some specific matrices useful to our problems.
In Section~\ref{section_Green}, the Green's function $G_\mathbf{M}$
is explicitly constructed,
and an initial form of eigencondition for the operator $\mathcal{K}_\mathbf{M}$ is presented
for each well-posed $\mathbf{M} \in \mathrm{wp}(4,8,\mathbb{C})$.
Using the results in Section~\ref{section_Green},
the two intermediate representations $\Gamma^-$, $\Gamma^+$ of $\mathrm{wp}(\mathbb{C})$
are constructed and analyzed in
Section~\ref{section_Gammapm}.
In Section~\ref{section_real}, 
the $\mathbb{R}$-algebra $\overline{\pi}(n)$ is introduced,
and $\overline{\pi}(4)$ is used to characterize the real boundary conditions $\mathrm{wp}(\mathbb{R})$.
In Section~\ref{section_Q},
explicit computations on the boundary condition $\mathbf{Q}$
in \eqref{equation_Q} are performed,
resulting in explicit forms of 
$
\Gamma^-\left( \left[ \mathbf{Q} \right] \right)
 =
\mathbf{G}_\mathbf{Q}^-
$
and
$
\Gamma^+\left( \left[ \mathbf{Q} \right] \right)
 =
\mathbf{G}_\mathbf{Q}^+
$.
In Section~\ref{section_Gamma},
the representation $\Gamma$ is constructed,
and is shown to have the features $(\Gamma1)$, $(\Gamma2)$, $(\Gamma3)$, and $(\Gamma4)$ above.
In Section~\ref{section_X},
the matrices $\mathbf{X}_\lambda(x)$ and $\mathbf{y}_\lambda(x)$
are defined explicitly,
and
Theorem~\ref{theorem_eigencondition-X} is proved.
In Section~\ref{section_symmetry},
some of the symmetries of $\mathbf{X}_\lambda(x)$, $\mathbf{Y}_\lambda(x)$ are explored,
and in particular, we show that $\mathbf{Y}_\lambda(l) \in \overline{\pi}(4)$
for every $0 \neq \lambda \in \mathbb{R} \setminus \Spec\mathcal{K}_\mathbf{Q}$
and $l > 0$.
Using the results in Section~\ref{section_symmetry},
we prove Theorem~\ref{theorem_existence}
in Section~\ref{section_existence}.
Finally, brief comments on future directions are given in Section~\ref{section_discussion}.

\section{Preliminaries}
\label{section_preliminaries}

\subsection{Terminologies}

We denote
$\mathbbm{i} = \sqrt{-1}$.
When
the $(i,j)$th entry of $\mathbf{A} \in \gl(m,n,\mathbb{C})$
is $a_{i,j}$, $1 \leq i \leq m$, $1 \leq j \leq n$,
we write
$
\mathbf{A}
 =
\left(
a_{i,j}
\right)_{1 \leq i \leq m, \, 1 \leq j \leq n}
$.
In case $m = n$,
we also write
$
\mathbf{A}
 =
\left(
a_{i,j}
\right)_{1 \leq i, j \leq n}
$.
For $\mathbf{A} \in \gl(m,n,\mathbb{C})$,
we denote the $(i,j)$th entry of $\mathbf{A}$
by $\mathbf{A}_{i,j}$.
The complex conjugate,
the transpose,
and
the conjugate transpose of $\mathbf{A} \in \gl(m,n,\mathbb{C})$
are denoted
respectively by
$
\overline{\mathbf{A}}
$,
$
\mathbf{A}^T
$,
$
\mathbf{A}^*
$.
For $\mathbf{A} \in \gl(n,\mathbb{C})$,
$
\adj{\mathbf{A}}
$
is the classical adjoint of $\mathbf{A}$,
so that,
if $\mathbf{A}$ is invertible,
then
$
\mathbf{A}^{-1}
 =
\adj{\mathbf{A}}
/\det\mathbf{A}.
$

For $n \in \mathbb{N}$,
let
$GL(n,\mathbb{C})$
(respectively, $GL(n,\mathbb{R})$)
be 
the set of invertible matrices in $\gl(n,\mathbb{C})$
(respectively, in $\gl(n,\mathbb{R})$).
$\mathbf{A} \in GL(n,\mathbb{C})$
is {\em orthogonal},
if $\mathbf{A}^{-1} = \mathbf{A}^T$,
and
is {\em unitary},
if $\mathbf{A}^{-1} = \mathbf{A}^*$.
For $n \in \mathbb{N}$,
let
$O(n)$
and
$U(n)$
be
the set of orthogonal matrices
and
the set of unitary matrices in $GL(n,\mathbb{C})$
respectively.
Regardless of their sizes,
we denote by $\mathbf{I}$ and $\mathbf{O}$,
the identity matrix and the zero matrix
respectively.
In case of possible confusion with size,
we denote 
$\mathbf{I} = \mathbf{I}_n \in \gl(n,\mathbb{C})$,
$\mathbf{O} = \mathbf{O}_{mn} \in \gl(m,n,\mathbb{C})$,
$\mathbf{O} = \mathbf{O}_n \in \gl(n,\mathbb{C})$.
In particular, we denote the zero column vector by
$\mathbf{0} = \mathbf{0}_n = \mathbf{O}_{n1}
\in \gl(n,1,\mathbb{C})$.
The diagonal matrix with entries
$
c_1, c_2, \cdots, c_n
$
is denoted by
$\diag\left( c_1, c_2, \cdots, c_n \right)$.

\subsection{Frequently used matrices}

Here, we introduce some special matrices
which will be used extensively in this paper.
They are useful for dealing with various symmetries in our problem,
and readers are recommended to be acquainted with their properties.

\begin{definition}
\label{definition_omega}
Denote
$
\omega_j
 =
e^{\mathbbm{i} \frac{\pi}{4} (2j - 1)}
$
for $j \in \mathbb{Z}$,
$
\mathbf{\Omega}
 =
\diag\left( \omega_1, \omega_2, \omega_3, \omega_4 \right)
$,
and
$
\mathbf{W}_0
 =
\left(
\omega_j^{i-1}
\right)_{1 \leq i, j \leq 4}
$.
\end{definition}

$\omega_1$, $\omega_2$, $\omega_3$, $\omega_4$
are the primitive $4$th roots of $-1$,
and
$\omega_{j+4} = \omega_j$, $j \in \mathbb{Z}$,
hence
\begin{gather}
\label{equation_omega}
\omega_j^4
 = 
-1,
\qquad
\overline{\omega_j}
 =
\omega_j^{-1},
\qquad
\mathbbm{i}
\omega_j
 =
\omega_{j+1},
\qquad
j \in \mathbb{Z}, \\
\label{equation_omegaR}
\omega_4
 =
\overline{\omega_1},
\qquad
\omega_3
 =
\overline{\omega_2}, \\
\label{equation_omegaL2}
\omega_3
 =
-
\omega_1,
\qquad
\omega_2
 =
-
\omega_4, \\
\label{equation_Omega}
\mathbf{\Omega}^4
 =
-\mathbf{I},
\qquad
\overline{\mathbf{\Omega}}
 =
\mathbf{\Omega}^{-1}.
\end{gather}

\begin{definition}
\label{definition_epsilon}
Let
$
\epsilon_1 = \epsilon_4 = 1
$,
$
\epsilon_2 = \epsilon_3 = -1
$,
and
$
\epsilon_{j+4} = \epsilon_j
$,
$j \in \mathbb{Z}$.
Denote
$
\mathcal{E}
 =
\diag
\left( \epsilon_1, \epsilon_2, \epsilon_3, \epsilon_4 \right)
 =
\diag(1,-1,-1,1)
$.
\end{definition}
Note that
\begin{equation}
\label{equation_omegaReIm}
\Real{\omega_j}
 = 
\frac{\epsilon_j}{\sqrt{2}},
\qquad
\Imag{\omega_j}
 = 
\frac{\epsilon_{j-1}}{\sqrt{2}},
\qquad
j \in \mathbb{Z}.
\end{equation}

\begin{definition}
\label{definition_W(x)}
Denote
$
y_j(x)
 =
e^{\omega_j \alpha x}
$,
$j = 1,2,3,4$,
and
\[
\mathbf{y}(x)
 =
\begin{pmatrix}
y_1(x) & y_2(x) & y_3(x) & y_4(x)
\end{pmatrix}^T
 =
\begin{pmatrix}
e^{\omega_1 \alpha x} &
e^{\omega_2 \alpha x} &
e^{\omega_3 \alpha x} &
e^{\omega_4 \alpha x}
\end{pmatrix}^T.
\]
Denote the Wronskian matrix corresponding to $y_1(x)$, $y_2(x)$, $y_3(x)$, $y_4(x)$ by
\[
\mathbf{W}(x)
 =
\begin{pmatrix}
y_1(x) & y_2(x) & y_3(x) & y_4(x) \\
y_1^\prime(x) & y_2^\prime(x) & y_3^\prime(x) & y_4^\prime(x) \\
y_1''(x) & y_2''(x) & y_3''(x) & y_4''(x) \\
y_1'''(x) & y_2'''(x) & y_3'''(x) & y_4'''(x)
\end{pmatrix}
 =
\begin{pmatrix}
\mathbf{y}(x)^T \\
\mathbf{y}^\prime(x)^T \\
\mathbf{y}^{\prime\prime}(x)^T \\
\mathbf{y}^{\prime\prime\prime}(x)^T
\end{pmatrix}.
\]
\end{definition}

Note that
\begin{equation}
\label{equation_y'x}
\mathbf{y}^\prime(x)
 =
\frac{d}{dx}
\begin{pmatrix}
e^{\omega_1 \alpha x} \\
e^{\omega_2 \alpha x} \\
e^{\omega_3 \alpha x} \\
e^{\omega_4 \alpha x}
\end{pmatrix}
 =
\begin{pmatrix}
\omega_1 \alpha \cdot e^{\omega_1 \alpha x} \\
\omega_2 \alpha \cdot e^{\omega_2 \alpha x} \\
\omega_3 \alpha \cdot e^{\omega_3 \alpha x} \\
\omega_4 \alpha \cdot e^{\omega_4 \alpha x}
\end{pmatrix}
 =
\alpha
\mathbf{\Omega}
\cdot
\mathbf{y}(x).
\end{equation}
By Definitions~\ref{definition_omega} and \ref{definition_W(x)},
\begin{align}
\mathbf{W}(x)
 &=
\left(
\left( \omega_j \alpha \right)^{i-1}
e^{\omega_j \alpha x}
\right)_{1 \leq i, j \leq 4}
\nonumber \\
 &=
\diag
\left( 1, \alpha, \alpha^2, \alpha^3 \right)
\cdot
\left(
\omega_j^{i-1}
\right)_{1 \leq i, j \leq 4}
\cdot
\diag
\left(
e^{\omega_1 \alpha x},
e^{\omega_2 \alpha x},
e^{\omega_3 \alpha x},
e^{\omega_4 \alpha x}
\right)
\nonumber \\
 &=
\diag
\left( 1, \alpha, \alpha^2, \alpha^3 \right)
\cdot
\mathbf{W}_0
e^{\mathbf{\Omega} \alpha x}.
\label{equation_Wxdecomposed}
\end{align}
By \eqref{equation_omega},
\begin{equation}
\label{equation_W0*}
\mathbf{W}_0^*
 =
\left(
\overline{\omega_i^{j-1}}
\right)_{1 \leq i,j \leq 4}
 =
\left(
\overline{\omega_i}^{j-1}
\right)_{1 \leq i,j \leq 4}
 =
\left(
\omega_i^{1-j}
\right)_{1 \leq i,j \leq 4}.
\end{equation}

\begin{lemma}
\label{lemma_W0}
$
\mathbf{W}_0^{-1}
 =
\frac{1}{4}
\mathbf{W}_0^*
$.
\end{lemma}

\begin{proof}
By \eqref{equation_W0*},
\begin{align*}
\mathbf{W}_0^*
\mathbf{W}_0
 &=
\left(
\omega_i^{1-j}
\right)_{1 \leq i,j \leq 4}
\cdot
\left(
\omega_j^{i-1}
\right)_{1 \leq i,j \leq 4}
 =
\left(
\sum_{r=1}^4
\omega_i^{1-r}
\cdot
\omega_j^{r-1}
\right)_{1 \leq i,j \leq 4} \\
 &=
\left(
\sum_{r=1}^4
\left(
\frac{\omega_j}{\omega_i}
\right)^{r-1}
\right)_{1 \leq i,j \leq 4},
\end{align*}
hence
\[
\left(
\mathbf{W}_0^*
\mathbf{W}_0
\right)_{i,j}
 =
\left\{
\begin{array}{ll}
4,
 &
\text{if }
i = j,
 \\
\frac
{
1
-
\left(
\frac{\omega_j}{\omega_i}
\right)^4
}
{
1
-
\left(
\frac{\omega_j}{\omega_i}
\right)
},
 &
\text{if }
i \neq j.
\end{array}
\right.
\]
If $i \neq j$,
then
$
1
-
\left(
\omega_j/\omega_i
\right)^4
 =
1
-
\omega_j^4/\omega_i^4
 =
1
-
(-1)/(-1)
 =
0
$
by \eqref{equation_omega}.
Thus
$
\mathbf{W}_0^*
\mathbf{W}_0
 =
4
\mathbf{I}$,
from which the result follows.
\end{proof}

The inverse $\mathbf{W}(x)^{-1}$ of $\mathbf{W}(x)$
is well-defined for every $x \in \mathbb{R}$,
and
by \eqref{equation_Wxdecomposed} and Lemma~\ref{lemma_W0},
\begin{align}
\mathbf{W}(x)^{-1}
 &=
e^{-\mathbf{\Omega} \alpha x}
\mathbf{W}_0^{-1}
\cdot
\diag
\left( 1, \alpha, \alpha^2, \alpha^3 \right)^{-1}
\nonumber \\
 &=
\frac{1}{4}
e^{-\mathbf{\Omega} \alpha x}
\mathbf{W}_0^*
\cdot
\diag
\left( 1, \alpha, \alpha^2, \alpha^3 \right)^{-1}.
\label{equation_Wx-1decomposed}
\end{align}

\begin{definition}
\label{definition_R}
Regardless of size, we denote
\[
\mathbf{R}
 =
\begin{pmatrix}
0 & 0 & 0 & 1 \\
0 & 0 & \rotatebox{90}{$\ddots$} & 0 \\
0 & 1 & 0 & 0 \\
1 & 0 & 0 & 0
\end{pmatrix}
\qquad
\in O(n).
\]
In case of possible confusion,
we denote 
$\mathbf{R} = \mathbf{R}_n \in O(n)$.
\end{definition}

Note that
$
\mathbf{R}^T
 =
\mathbf{R}^*
 =
\mathbf{R}
 =
\mathbf{R}^{-1}
$.
When multiplied to the left (respectively, to the right) of a matrix,
$\mathbf{R}$ {\em reverses} the order of the rows (respectively, the columns)
of that matrix.
Hence by \eqref{equation_omegaR},
\begin{equation}
\label{equation_R}
\mathbf{R}
\mathbf{\Omega}
 =
\overline{\mathbf{\Omega}}
\mathbf{R},
\quad
\mathbf{W}_0
\mathbf{R}
 =
\overline{
\mathbf{W}_0
},
\quad
\mathbf{R}
\cdot
\mathbf{y}(x)
 =
\overline{
\mathbf{y}(x)
},
\quad
\mathbf{W}(x)
\mathbf{R}
 =
\overline{
\mathbf{W}(x)
}.
\end{equation}

\begin{definition}
\label{definition_L}
Denote
\[
\mathbf{L}
 =
\begin{pmatrix}
0 & 1 & 0 & 0 \\
0 & 0 & 1 & 0 \\
0 & 0 & 0 & 1 \\
1 & 0 & 0 & 0
\end{pmatrix}
\qquad
\in O(4).
\]
\end{definition}

Note that,
when multiplied to the left (respectively, to the right) of a matrix,
$\mathbf{L}$ {\em lifts up} the rows cyclically 
by one row
(respectively, moves the columns to the right cyclically
by one column).
Thus by \eqref{equation_omega},
\begin{align}
\mathbf{L}
\mathbf{\Omega}
\mathbf{L}^{-1}
 &=
\begin{pmatrix}
0 & \omega_2 & 0 & 0 \\
0 & 0 & \omega_3 & 0 \\
0 & 0 & 0 & \omega_4 \\
\omega_1 & 0 & 0 & 0
\end{pmatrix}
\mathbf{L}^{-1}
 =
\diag
\left( \omega_2, \omega_3, \omega_4, \omega_1 \right)
 =
\mathbbm{i}
\mathbf{\Omega},
\label{equation_LOmegaL-1} \\
\mathbf{W}_0
\mathbf{L}^{-1}
 &=
\left(
\omega_{j+1}^{i-1}
\right)_{1 \leq i,j \leq 4}
 =
\left(
\left(
\mathbbm{i}
\omega_j
\right)^{i-1}
\right)_{1 \leq i,j \leq 4}
 =
\diag
\left( 1, \mathbbm{i}, \mathbbm{i}^2, \mathbbm{i}^3 \right)
\cdot
\mathbf{W}_0.
\label{equation_W0L-1}
\end{align}
In particular, 
\[
\mathbf{L}^2
 =
\begin{pmatrix}
0 & 0 & 1 & 0 \\
0 & 0 & 0 & 1 \\
1 & 0 & 0 & 0 \\
0 & 1 & 0 & 0
\end{pmatrix}
\qquad
\in
O(4)
\]
is also frequently used.
Note that
$
\left(
\mathbf{L}^2
\right)^T
 =
\left(
\mathbf{L}^2
\right)^*
 =
\mathbf{L}^2
 =
\left(
\mathbf{L}^2
\right)^{-1}
$.
By \eqref{equation_omegaL2},
\begin{equation}
\label{equation_L2}
\mathbf{L}^2
\mathbf{\Omega}
 =
-
\mathbf{\Omega}
\mathbf{L}^2,
\qquad
\mathbf{L}^2
\cdot
\mathbf{y}(x)
 =
\mathbf{y}(-x).
\end{equation}

\begin{definition}
\label{definition_calBpm}
Define
$\mathcal{B}^-, \mathcal{B}^+ : C^3[-l,l] \to \gl(4,1,\mathbb{C})$
by
\begin{align*}
\mathcal{B}^-[u]
 &=
\begin{pmatrix}
u(-l) & u^\prime(-l) & u^{\prime\prime}(-l) & u^{\prime\prime\prime}(-l)
\end{pmatrix}^T, \\
\mathcal{B}^+[u]
 &=
\begin{pmatrix}
u(l) & u^\prime(l) & u^{\prime\prime}(l) & u^{\prime\prime\prime}(l)
\end{pmatrix}^T.
\end{align*}
\end{definition}

Let $u \in C^3[-l,l]$.
Note that
\begin{equation}
\label{equation_calBpm}
\mathcal{B}[u]
 =
\begin{pmatrix}
\mathcal{B}^-[u] \\
\mathcal{B}^+[u]
\end{pmatrix},
\end{equation}
where $\mathcal{B}$ is defined by \eqref{equation_calB}.
Let
$\mathbf{M}^-, \mathbf{M}^+ \in \gl(4,\mathbb{C})$,
so that
$
\mathbf{M}
 = 
\begin{pmatrix}
\mathbf{M}^-  & \mathbf{M}^+ 
\end{pmatrix}
\in \gl(4,8,\mathbb{C})
$.
Then by \eqref{equation_calBpm},
\begin{equation}
\label{equation_MBpm}
\mathbf{M} \cdot \mathcal{B}[u]
 =
\begin{pmatrix}
\mathbf{M}^- & \mathbf{M}^+
\end{pmatrix}
\begin{pmatrix}
\mathcal{B}^-[u] \\ \mathcal{B}^+[u]
\end{pmatrix}
 =
\mathbf{M}^-
\cdot
\mathcal{B}^-[u]
+
\mathbf{M}^+
\cdot
\mathcal{B}^+[u].
\end{equation}

By \eqref{equation_omega},
the functions $y_1, y_2, y_3, y_4$ in Definition~\ref{definition_W(x)} form
a fundamental set of solutions of
the linear homogeneous equation $\mathrm{DE}(0)$:
$u^{(4)} + \alpha^4 u = 0$ in \eqref{equation_linearhomogeneous}.
Thus $u \in L^2[-l,l]$ is a solution
of $\mathrm{DE}(0)$,
if and only if
$
u(x)
 = 
\sum_{j=1}^4 c_j \cdot y_j(x)
 =
\mathbf{y}(x)^T
\mathbf{c}
$
for some 
$
\mathbf{c}
 =
\begin{pmatrix}
c_1 & c_2 & c_3 & c_4
\end{pmatrix}^T
\in \gl(4,1,\mathbb{C})
$,
in which case we have
\begin{align}
\mathcal{B}^\pm[u]
 &=
\mathcal{B}^\pm
\left[
\sum_{j=1}^4 c_j \cdot y_j
\right]
 =
\sum_{j=1}^4
c_j
\cdot
\mathcal{B}^\pm\left[ y_j \right]
 =
\sum_{j=1}^4
\begin{pmatrix}
y_j(\pm l) \\ y_j^\prime(\pm l) \\ y_j^{\prime\prime}(\pm l) \\ y_j^{\prime\prime\prime}(\pm l)
\end{pmatrix}
\cdot
c_j
\nonumber \\
 &=
\begin{pmatrix}
\mathbf{y}(\pm l)^T \\
\mathbf{y}^\prime( \pm l)^T \\
\mathbf{y}^{\prime\prime}(\pm l)^T \\
\mathbf{y}^{\prime\prime\prime}(\pm l)^T
\end{pmatrix}
\mathbf{c}
 =
\mathbf{W}(\pm l)
\mathbf{c}.
\label{equation_WBpm}
\end{align}

\section{Green's functions for well-posed boundary conditions}
\label{section_Green}

\begin{definition}
\label{definition_wellposed}
$\mathbf{M} \in \gl(4,8,\mathbb{C})$ is called {\em well-posed},
if the boundary value problem 
consisting of
$\mathrm{DE}(0)$:
$
u^{(4)} + \alpha^4 u
 =
0
$
and
$\mathrm{BC}(\mathbf{M})$:
$\mathbf{M} \cdot \mathcal{B}[u] = \mathbf{0}$ in \eqref{equation_BC0},
has the unique trivial solution $u = 0$
in $L^2[-l,l]$.
The set of well-posed matrices in $\gl(4,8,\mathbb{C})$
is denoted by 
$\mathrm{wp}(4,8,\mathbb{C})$.
\end{definition}

\begin{definition}
\label{definition_Mtilde}
For
$
\mathbf{M}
\in \gl(4,8,\mathbb{C})
$,
we denote
$
\widetilde{\mathbf{M}}^-
 =
\mathbf{M}^- \mathbf{W}(-l)
$,
$
\widetilde{\mathbf{M}}^+
 =
\mathbf{M}^+ \mathbf{W}(l)
$,
and
$
\widetilde{\mathbf{M}}
 =
\widetilde{\mathbf{M}}^- + \widetilde{\mathbf{M}}^+
\in \gl(4,\mathbb{C})
$,
where
$\mathbf{M}^-, \mathbf{M}^+ \in \gl(4,\mathbb{C})$
are the $4 \times 4$ minors of 
$\mathbf{M}$ such that
$
\mathbf{M}
 = 
\begin{pmatrix}
\mathbf{M}^-  & \mathbf{M}^+ 
\end{pmatrix}
$.
\end{definition}

\begin{lemma}
\label{lemma_wellposed}
Let
$
\mathbf{M}
\in \gl(4,8,\mathbb{C})
$.
Then
$
\mathbf{M}
\in \mathrm{wp}(4,8,\mathbb{C})
$,
if and only if
$
\det
\widetilde{\mathbf{M}}
 \neq
0
$.
\end{lemma}

\begin{proof}
Let
$\mathbf{M}^-, \mathbf{M}^+ \in \gl(4,\mathbb{C})$
be the $4 \times 4$ minors of 
$\mathbf{M}$ such that
$
\mathbf{M}
 = 
\begin{pmatrix}
\mathbf{M}^-  & \mathbf{M}^+ 
\end{pmatrix}
$.
Suppose $u \in L^2[-l,l]$ is a solution of 
$\mathrm{DE}(0)$,
so that
$
u(x)
 =
\mathbf{y}(x)^T
\mathbf{c}
$
for some 
$
\mathbf{c}
\in \gl(4,1,\mathbb{C})
$.
Then by \eqref{equation_MBpm}, \eqref{equation_WBpm}, and Definition~\ref{definition_Mtilde},
the boundary condition $\mathrm{BC}(\mathbf{M})$ becomes
\begin{equation}
\label{equation_mvas;odfijv}
\mathbf{0}
 =
\mathbf{M} \cdot \mathcal{B}[u]
 =
\mathbf{M}^-
\cdot
\mathbf{W}(-l)
\mathbf{c}
+
\mathbf{M}^+
\cdot
\mathbf{W}(l)
\mathbf{c}
 =
\left(
\widetilde{\mathbf{M}}^-
+
\widetilde{\mathbf{M}}^+
\right)
\mathbf{c}
 =
\widetilde{\mathbf{M}}
\mathbf{c}.
\end{equation}
By Definition~\ref{definition_wellposed},
$
\mathbf{M}
\in \mathrm{wp}(4,8,\mathbb{C})
$,
if and only if
$\mathbf{c} = \mathbf{0}$
is the only solution
in $\gl(4,1,\mathbb{C})$
satisfying \eqref{equation_mvas;odfijv},
which is equivalent to
$\det\widetilde{\mathbf{M}} = 0$.
\end{proof}

Since
$
\widetilde{\mathbf{M}}
$
is invertible
for every 
$
\mathbf{M}
\in \mathrm{wp}(4,8,\mathbb{C})
$,
the following is well-defined.

\begin{definition}
\label{definition_GMpm}
For
$
\mathbf{M}
\in \mathrm{wp}(4,8,\mathbb{C})
$,
we denote
$
\mathbf{G}_\mathbf{M}^-
 =
\widetilde{\mathbf{M}}^{-1}
\widetilde{\mathbf{M}}^-
\mathbf{\Omega}
\mathbf{L}^2
$
and
$
\mathbf{G}_\mathbf{M}^+
 =
\widetilde{\mathbf{M}}^{-1}
\widetilde{\mathbf{M}}^+
\mathbf{\Omega}
\mathbf{L}^2
$.
\end{definition}

By Definition~\ref{definition_Mtilde},
$
\mathbf{G}_\mathbf{M}^-
+
\mathbf{G}_\mathbf{M}^+
 =
\widetilde{\mathbf{M}}^{-1}
\widetilde{\mathbf{M}}^-
\mathbf{\Omega}
\mathbf{L}^2
+
\widetilde{\mathbf{M}}^{-1}
\widetilde{\mathbf{M}}^+
\mathbf{\Omega}
\mathbf{L}^2
 =
\widetilde{\mathbf{M}}^{-1}
\widetilde{\mathbf{M}}
\mathbf{\Omega}
\mathbf{L}^2
$,
hence
we have
\begin{equation}
\label{equation_GM-+GM+}
\mathbf{G}_\mathbf{M}^-
+
\mathbf{G}_\mathbf{M}^+
 =
\mathbf{\Omega}
\mathbf{L}^2,
\qquad
\mathbf{M}
\in \mathrm{wp}(4,8,\mathbb{C}).
\end{equation}

\begin{definition}
\label{definition_GMKM}
Let $\mathbf{M} \in \mathrm{wp}(4,8,\mathbb{C})$.
Define the function
$G_\mathbf{M}: [-l,l] \times [-l,l] \to \mathbb{C}$,
called the {\em Green's function}
corresponding to $\mathbf{M}$,
by
\[
G_\mathbf{M}(x,\xi)
 =
\frac{\alpha}{4k}
\cdot
\left\{
\begin{array}{ll}
\mathbf{y}(x)^T
\cdot
\mathbf{G}_\mathbf{M}^+
\cdot
\mathbf{y}(\xi),
 &
\text{if }
x \leq \xi, \\
-
\mathbf{y}(x)^T
\cdot
\mathbf{G}_\mathbf{M}^-
\cdot
\mathbf{y}(\xi),
 &
\text{if }
\xi \leq x.
\end{array}
\right.
\]
Define the operator
$\mathcal{K}_\mathbf{M}: L^2[-l,l] \to L^2[-l,l]$
by
\[
\mathcal{K}_\mathbf{M}[w](x)
 =
\int_{-l}^l G_\mathbf{M}(x,\xi) w(\xi) \, d\xi
,
\qquad
w \in L^2[-l,l],
\
x \in [-l,l].
\]
\end{definition}

Note that $G_\mathbf{M}$ is bounded on $[-l,l] \times [-l,l]$.
It is well-known~\cite{Stakgold} that
an integral operator of the form
$w \mapsto \int_{-l}^l g(x,\xi) w(\xi) \, dx$
with bounded kernel $g(x,\xi)$,
is a compact linear operator on $L^2[-l,l]$.
Thus $\mathcal{K}_\mathbf{M}$ is a well-defined compact linear operator on $L^2[-l,l]$.
Note from Definitions~\ref{definition_GMpm} and \ref{definition_GMKM} 
that the function $G_\mathbf{M}$
is defined constructively in terms of
given $\mathbf{M} \in \mathrm{wp}(4,8,\mathbb{C})$. 
Lemma~\ref{lemma_KM} below,
whose proof is in Appendix~\ref{appendix_Green},
shows that
$G_\mathbf{M}$ is the usual Green's function
for the boundary value problem consisting of $\mathrm{DE}(0)$ and $\mathrm{BC}(\mathbf{M})$.

\begin{lemma}
\label{lemma_KM}
Let $\mathbf{M} \in \mathrm{wp}(4,8,\mathbb{C})$
and $w \in L^2[-l,l]$.
Then
$\mathcal{K}_\mathbf{M}[w]$ is the unique solution of the boundary value problem 
consisting of
$\mathrm{DE}(w)$ and $\mathrm{BC}(\mathbf{M})$.
\end{lemma}

By Lemma~\ref{lemma_KM},
we have
\begin{align}
\label{equation_KMwDE}
\mathcal{K}_\mathbf{M}[u]^{(4)} 
+ 
\alpha^4 
\cdot 
\mathcal{K}_\mathbf{M}[u]
 &=
\frac{\alpha^4}{k} 
\cdot 
u,
 &
\mathbf{M} \in \mathrm{wp}(4,8,\mathbb{C}),
\
u \in L^2[-l,l], \\
\label{equation_KMwBC}
\mathbf{M} 
\cdot 
\mathcal{B}
\left[ 
\mathcal{K}_\mathbf{M}[u]
\right]
 &=
\mathbf{0},
 &
\mathbf{M} \in \mathrm{wp}(4,8,\mathbb{C}),
\
u \in L^2[-l,l].
\end{align}

Note that \eqref{equation_KMwDE} in particular implies that
the linear operator $\mathcal{K}_\mathbf{M}$ is one-to-one,
or {\em injective}, for every $\mathbf{M} \in \mathrm{wp}(4,8,\mathbb{C})$.

\begin{definition}
\label{definition_EDE}
For $0 \neq \lambda \in \mathbb{C}$,
we denote by $\mathrm{EDE}(\lambda)$,
the homogeneous equation
\[
\mathrm{EDE}(\lambda):
u^{(4)} + \left( 1 - \frac{1}{\lambda k} \right) \alpha^4 u
 =
0.
\]
\end{definition}

Note that
$
1
-
1/(\lambda k)
 \neq
1
$
for any $\lambda \in \mathbb{C}$.
In fact, the homogeneous equation $\mathrm{DE}(0)$:
$u^{(4)} + \alpha^4 u = 0$ can be regarded as the limiting case $\mathrm{EDE}(\infty)$.
In terms of
given $\mathbf{M} \in \mathrm{wp}(4,8,\mathbb{C})$,
$\mathcal{K}_\mathbf{M}$ has the following eigencondition.

\begin{lemma}
\label{lemma_eigencondition}
Let
$
\mathbf{M}
\in \mathrm{wp}(4,8,\mathbb{C})$,
$0 \neq u \in L^2[-l,l]$,
and
$\lambda \in \mathbb{C}$.
Then $\mathcal{K}_\mathbf{M}[u] = \lambda \cdot u$,
if and only if
$\lambda \neq 0$
and
$u$ satisfies $\mathrm{EDE}(\lambda)$
and $\mathrm{BC}(\mathbf{M})$.
\end{lemma}

\begin{proof}
Let $u \neq 0 \in L^2[-l,l]$,
$\lambda \in \mathbb{C}$.
Suppose $\mathcal{K}_\mathbf{M}[u] = \lambda \cdot u$.
If $\lambda = 0$,
then $\mathcal{K}_\mathbf{M}[u] = 0 \cdot u = 0$,
hence $u = 0$ by \eqref{equation_KMwDE},
contradicting the assumption $u \neq 0$.
Thus $\lambda \neq 0$.
By \eqref{equation_KMwDE},
we have
\begin{align*}
\lambda \cdot u^{(4)}
 &=
(\lambda \cdot u)^{(4)}
 =
\mathcal{K}_\mathbf{M}[u]^{(4)}
 =
-\alpha^4
\cdot
\mathcal{K}_\mathbf{M}[u]
+
\frac{\alpha^4}{k}
\cdot
u \\
 &=
-\alpha^4
\cdot
(\lambda \cdot u)
+
\frac{\alpha^4}{k}
\cdot
u
 =
-
\left(
\lambda - \frac{1}{k}
\right)
\alpha^4
\cdot
u,
\end{align*}
which shows that $u$ satisfies $\mathrm{EDE}(\lambda)$.
By \eqref{equation_KMwBC},
we have
$
\mathbf{M}
\cdot
\mathcal{B}\left[ u \right]
 =
(1/\lambda)
\cdot
\mathbf{M}
\cdot
\mathcal{B}\left[ \lambda \cdot u \right]
 =
(1/\lambda)
\cdot
\mathbf{M}
\cdot
\mathcal{B}
\left[
\mathcal{K}_\mathbf{M}[u]
\right]
 =
\mathbf{0}
$,
hence $u$ satisfies $\mathrm{BC}(\mathbf{M})$.

Conversely,
suppose $\lambda \neq 0$
and 
$u$ satisfies $\mathrm{EDE}(\lambda)$
and $\mathrm{BC}(\mathbf{M})$.
Let 
$
\hat{u}
 =
\mathcal{K}_\mathbf{M}[u] - \lambda \cdot u
$.
Then by \eqref{equation_KMwDE},
we have
\begin{align*}
\lefteqn{
\hat{u}^{(4)}
+
\alpha^4
\hat{u}
 =
\left(
\mathcal{K}_\mathbf{M}[u] - \lambda \cdot u
\right)^{(4)}
+
\alpha^4
\left(
\mathcal{K}_\mathbf{M}[u] - \lambda \cdot u
\right)
} \\
 &=
\left(
\mathcal{K}_\mathbf{M}[u]^{(4)}
+
\alpha^4
\cdot
\mathcal{K}_\mathbf{M}[u]
\right)
-
\lambda
\left(
u^{(4)}
+
\alpha^4
u
\right)
 =
\frac{\alpha^4}{k}
\cdot
u
-
\lambda
\left(
u^{(4)}
+
\alpha^4
u
\right) \\
 &=
-\lambda
\left\{
u^{(4)}
+
\left(
1
-
\frac{1}{\lambda k}
\right)
\alpha^4
u
\right\},
\end{align*}
hence
$
\hat{u}^{(4)}
+
\alpha^4
\hat{u}
 =
0
$,
since $u$ satisfies $\mathrm{EDE}(\lambda)$.
By \eqref{equation_KMwBC},
we have
$
\mathbf{M}
\cdot
\mathcal{B}
\left[
\hat{u}
\right]
 =
\mathbf{M}
\cdot
\mathcal{B}
\left[
\mathcal{K}_\mathbf{M}[u] - \lambda \cdot u
\right]
 =
\mathbf{M}
\cdot
\mathcal{B}
\left[
\mathcal{K}_\mathbf{M}[u]
\right]
-
\lambda
\mathbf{M}
\cdot
\mathcal{B}\left[ u \right]
 =
-
\lambda
\mathbf{M}
\cdot
\mathcal{B}\left[ u \right]
$,
hence
$
\mathbf{M}
\cdot
\mathcal{B}\left[ \hat{u} \right]
 =
\mathbf{0}
$,
since $u$ satisfies $\mathrm{BC}(\mathbf{M})$.
It follows that
$
\hat{u}
 =
\mathcal{K}_\mathbf{M}[u] - \lambda \cdot u
$
is the unique solution of the boundary value problem
consisting of 
$\mathrm{DE}(0)$ and $\mathrm{BC}(\mathbf{M})$,
which is $0$ by Definition~\ref{definition_wellposed}.
Thus we have $\mathcal{K}_\mathbf{M}[u] = \lambda \cdot u$,
and the proof is complete.
\end{proof}

\section{Representation of well-posed boundary conditions}

\subsection{The representations $\protect\Gamma^-$ and $\protect\Gamma^+$}
\label{section_Gammapm}

\begin{definition}
\label{definition_equivalence}
$\mathbf{M}, \mathbf{N} \in \mathrm{wp}(4,8,\mathbb{C})$
are called {\em equivalent},
and denote $\mathbf{M} \approx \mathbf{N}$,
if $\mathcal{K}_\mathbf{M} = \mathcal{K}_\mathbf{N}$.
The quotient set
$
\left.
\mathrm{wp}(4,8,\mathbb{C})
\right/\!\!\approx
$
of $\mathrm{wp}(4,8,\mathbb{C})$
with respect to the equivalence relation $\approx$,
is denoted by
$\mathrm{wp}(\mathbb{C})$.
For $\mathbf{M} \in \mathrm{wp}(4,8,\mathbb{C})$,
we denote by $\left[ \mathbf{M} \right] \in \mathrm{wp}(\mathbb{C})$
the equivalence class with respect to $\approx$ containing
$\mathbf{M}$.
\end{definition}

Note from Definitions~\ref{definition_GMKM} and \ref{definition_equivalence} that
$\mathbf{M} \approx \mathbf{N}$,
if and only if $G_\mathbf{M} = G_{\mathbf{N}}$.

\begin{lemma}
\label{lemma_equivalence}
For
$\mathbf{M}, \mathbf{N} \in \mathrm{wp}(4,8,\mathbb{C})$,
the following {\rm (a)}, {\rm (b)}, {\rm (c)}, {\rm (d)} are equivalent.
{\rm (a)}
$\mathbf{M} \approx \mathbf{N}$.
{\rm (b)}
$\mathbf{G}_\mathbf{M}^+ = \mathbf{G}_\mathbf{N}^+$.
{\rm (c)}
$\mathbf{G}_\mathbf{M}^- = \mathbf{G}_\mathbf{N}^-$.
{\rm (d)}
$
\mathbf{N}
 =
\mathbf{P}
\mathbf{M}$
for some $\mathbf{P} \in GL(4,\mathbb{C})$.
\end{lemma}

\begin{proof}
The equivalence of (b) and (c) follows immediately,
since
$
\mathbf{G}_\mathbf{M}^-
+
\mathbf{G}_\mathbf{M}^+
 =
\mathbf{\Omega}
\mathbf{L}^2
 =
\mathbf{G}_\mathbf{N}^-
+
\mathbf{G}_\mathbf{N}^+
$
by \eqref{equation_GM-+GM+}.
Since the entries $y_1, y_2, y_3, y_4$
of $\mathbf{y}$ in Definition~\ref{definition_W(x)}
are linearly independent,
it follows from Definition~\ref{definition_GMKM}
that 
$G_\mathbf{M} = G_\mathbf{N}$,
if and only if 
$\mathbf{G}_\mathbf{M}^- = \mathbf{G}_\mathbf{N}^-$
and
$\mathbf{G}_\mathbf{M}^+ = \mathbf{G}_\mathbf{N}^+$.
Thus (a), (b), (c) are equivalent,
and hence it is sufficient to show the equivalence of (b) and (d).

Let
$\mathbf{M}^-, \mathbf{M}^+ \in \gl(4,\mathbb{C})$
and
$\mathbf{N}^-, \mathbf{N}^+ \in \gl(4,\mathbb{C})$
be the $4 \times 4$ minors of 
$\mathbf{M}$
and
$\mathbf{N}$
respectively
such that
$
\mathbf{M}
 = 
\begin{pmatrix}
\mathbf{M}^-  & \mathbf{M}^+ 
\end{pmatrix}
$,
$
\mathbf{N}
 = 
\begin{pmatrix}
\mathbf{N}^-  & \mathbf{N}^+ 
\end{pmatrix}
$.
Suppose (d).
Then
$
\begin{pmatrix}
\mathbf{N}^-  & \mathbf{N}^+ 
\end{pmatrix}
 =
\mathbf{N}
 =
\mathbf{P}
\mathbf{M}
 =
\mathbf{P}
\begin{pmatrix}
\mathbf{M}^-  & \mathbf{M}^+ 
\end{pmatrix}
 =
\begin{pmatrix}
\mathbf{P}
\mathbf{M}^-  &
\mathbf{P}
\mathbf{M}^+ 
\end{pmatrix}
$
for some
$
\mathbf{P}
\in GL(4,\mathbb{C})
$.
So we have
$
\mathbf{N}^-
 =
\mathbf{P}
\mathbf{M}^-
$
and
$
\mathbf{N}^+
 =
\mathbf{P}
\mathbf{M}^+
$,
and hence by Definition~\ref{definition_Mtilde},
$
\widetilde{\mathbf{N}}^\pm
 =
\mathbf{N}^\pm
\mathbf{W}(\pm l)
 =
\mathbf{P}
\mathbf{M}^\pm
\mathbf{W}(\pm l)
 =
\mathbf{P}
\widetilde{\mathbf{M}}^\pm
$
and
$
\widetilde{\mathbf{N}}
 =
\widetilde{\mathbf{N}}^-
+
\widetilde{\mathbf{N}}^+
 =
\mathbf{P}
\widetilde{\mathbf{M}}^-
+
\mathbf{P}
\widetilde{\mathbf{M}}^+
 =
\mathbf{P}
\widetilde{\mathbf{M}}
$.
Thus by Definition~\ref{definition_GMpm},
\begin{align*}
\mathbf{G}_\mathbf{N}^+
 &=
\widetilde{\mathbf{N}}^{-1}
\widetilde{\mathbf{N}}^+
\mathbf{\Omega}
\mathbf{L}^2
 =
\left(
\mathbf{P}
\widetilde{\mathbf{M}}
\right)^{-1}
\left(
\mathbf{P}
\widetilde{\mathbf{M}}^+
\right)
\mathbf{\Omega}
\mathbf{L}^2
 =
\widetilde{\mathbf{M}}^{-1}
\mathbf{P}^{-1}
\mathbf{P}
\widetilde{\mathbf{M}}^+
\mathbf{\Omega}
\mathbf{L}^2 \\
 &=
\widetilde{\mathbf{M}}^{-1}
\widetilde{\mathbf{M}}^+
\mathbf{\Omega}
\mathbf{L}^2
 =
\mathbf{G}_\mathbf{M}^+,
\end{align*}
hence we have (b).

Conversely, suppose (b),
so that $\mathbf{G}_\mathbf{M}^+ = \mathbf{G}_\mathbf{N}^+$.
Since (b) and (c) are equivalent,
we also have
$\mathbf{G}_\mathbf{M}^- = \mathbf{G}_\mathbf{N}^-$.
By Definition~\ref{definition_GMpm},
we have
$
\widetilde{\mathbf{M}}^{-1}
\widetilde{\mathbf{M}}^-
 =
\widetilde{\mathbf{N}}^{-1}
\widetilde{\mathbf{N}}^-
$
and
$
\widetilde{\mathbf{M}}^{-1}
\widetilde{\mathbf{M}}^+
 =
\widetilde{\mathbf{N}}^{-1}
\widetilde{\mathbf{N}}^+
$,
since
$
\mathbf{\Omega}
\mathbf{L}^2
$
is invertible.
So we have
$
\widetilde{\mathbf{N}}^\pm
 =
\widetilde{\mathbf{N}}
\widetilde{\mathbf{M}}^{-1}
\cdot
\widetilde{\mathbf{M}}^\pm
$,
hence
$
\mathbf{N}^\pm
\mathbf{W}(\pm l)
 =
\widetilde{\mathbf{N}}
\widetilde{\mathbf{M}}^{-1}
\cdot
\mathbf{M}^\pm
\mathbf{W}(\pm l)
$
by Definition~\ref{definition_Mtilde}.
Since $\mathbf{W}(\pm l)$ are invertible,
we have
$
\mathbf{N}^\pm
 =
\widetilde{\mathbf{N}}
\widetilde{\mathbf{M}}^{-1}
\cdot
\mathbf{M}^\pm
$,
hence
\[
\mathbf{N}
 =
\begin{pmatrix}
\mathbf{N}^-  & \mathbf{N}^+ 
\end{pmatrix}
 =
\begin{pmatrix}
\widetilde{\mathbf{N}}
\widetilde{\mathbf{M}}^{-1}
\cdot
\mathbf{M}^-  & 
\widetilde{\mathbf{N}}
\widetilde{\mathbf{M}}^{-1}
\cdot
\mathbf{M}^+ 
\end{pmatrix}
 =
\widetilde{\mathbf{N}}
\widetilde{\mathbf{M}}^{-1}
\cdot
\mathbf{M}.
\]
Thus we have (d),
since
$
\widetilde{\mathbf{N}}
\widetilde{\mathbf{M}}^{-1}
\in GL(4,\mathbb{C})
$
by Lemma~\ref{lemma_wellposed},
and the proof is complete.
\end{proof}

\begin{definition}
\label{definition_Gammapm}
Define
$\Gamma^-, \Gamma^+ : \mathrm{wp}(\mathbb{C}) \to \gl(4,\mathbb{C})$
by
$
\Gamma^-\left( \left[ \mathbf{M} \right] \right)
 =
\mathbf{G}_\mathbf{M}^-
$
and
$
\Gamma^+\left( \left[ \mathbf{M} \right] \right)
 =
\mathbf{G}_\mathbf{M}^+
$
for
$\mathbf{M} \in \mathrm{wp}(4,8,\mathbb{C})$.
\end{definition}

By Lemma~\ref{lemma_equivalence},
$\Gamma^-$, $\Gamma^+$ 
are well-defined and one-to-one.
Lemma~\ref{lemma_GMpmtoM} below shows that
$\Gamma^-$, $\Gamma^+$
are also onto,
and hence are one-to-one correspondences
from $\mathrm{wp}(\mathbb{C})$ to $\gl(4,\mathbb{C})$.

\begin{lemma}
\label{lemma_GMpmtoM}
Suppose $\mathbf{G}^-, \mathbf{G}^+ \in \gl(4,\mathbb{C})$
satisfy $\mathbf{G}^- + \mathbf{G}^+ = \mathbf{\Omega} \mathbf{L}^2$.
Then there exists
$
\mathbf{M}
\in \mathrm{wp}(4,8,\mathbb{C})$
such that
$
\mathbf{G}_\mathbf{M}^-
 =
\mathbf{G}^-$,
$
\mathbf{G}_\mathbf{M}^+
 =
\mathbf{G}^+$.
In particular,
$\mathbf{M}$ can be taken by
$
\mathbf{M}
 =
\begin{pmatrix}
\mathbf{M}^- & \mathbf{M}^+
\end{pmatrix}
$,
where
$
\mathbf{M}^-
 =
\mathbf{G}^-
\left(
\mathbf{\Omega}
\mathbf{L}^2
\right)^{-1}
\mathbf{W}(-l)^{-1}
$,
$
\mathbf{M}_+
 =
\mathbf{G}^+
\left(
\mathbf{\Omega}
\mathbf{L}^2
\right)^{-1}
\mathbf{W}(l)^{-1}
$.
\end{lemma}

\begin{proof}
Let
$
\mathbf{M}^-
 =
\mathbf{G}^-
\left(
\mathbf{\Omega}
\mathbf{L}^2
\right)^{-1}
\mathbf{W}(-l)^{-1}
$,
$
\mathbf{M}_+
 =
\mathbf{G}^+
\left(
\mathbf{\Omega}
\mathbf{L}^2
\right)^{-1}
\mathbf{W}(l)^{-1}
$,
and let
$
\mathbf{M}
 =
\begin{pmatrix}
\mathbf{M}^- & \mathbf{M}^+
\end{pmatrix}
\in \gl(4,8,\mathbb{C})
$.
Then by Definition~\ref{definition_Mtilde},
we have
\begin{equation}
\label{equation_sdgbhso}
\widetilde{\mathbf{M}}^\pm
 =
\mathbf{M}^\pm
\mathbf{W}(\pm l)
 =
\mathbf{G}^\pm
\left(
\mathbf{\Omega}
\mathbf{L}^2
\right)^{-1}
\mathbf{W}(\pm l)^{-1}
\cdot
\mathbf{W}(\pm l)
 =
\mathbf{G}^\pm
\left(
\mathbf{\Omega}
\mathbf{L}^2
\right)^{-1},
\end{equation}
hence
\begin{align}
\widetilde{\mathbf{M}}
 &=
\widetilde{\mathbf{M}}_- + \widetilde{\mathbf{M}}_+
 =
\mathbf{G}^-
\left(
\mathbf{\Omega}
\mathbf{L}^2
\right)^{-1}
+
\mathbf{G}^+
\left(
\mathbf{\Omega}
\mathbf{L}^2
\right)^{-1}
 =
\left(
\mathbf{G}^- + \mathbf{G}^+
\right)
\left(
\mathbf{\Omega}
\mathbf{L}^2
\right)^{-1}
\nonumber \\
 &=
\mathbf{\Omega}
\mathbf{L}^2
\cdot
\left(
\mathbf{\Omega}
\mathbf{L}^2
\right)^{-1}
 =
\mathbf{I}.
\label{equation_aho;fjv}
\end{align}
Thus
$\mathbf{M} \in \mathrm{wp}(4,8,\mathbb{C})$ 
by Lemma~\ref{lemma_wellposed},
since
$
\widetilde{\mathbf{M}}
 =
\mathbf{I}
$
is invertible.
By Definition~\ref{definition_GMpm}
and
\eqref{equation_sdgbhso}, \eqref{equation_aho;fjv},
we have
$
\mathbf{G}_\mathbf{M}^\pm
 =
\widetilde{\mathbf{M}}^{-1}
\widetilde{\mathbf{M}}^\pm
\mathbf{\Omega}
\mathbf{L}^2
 =
\mathbf{I}^{-1}
\cdot
\mathbf{G}^\pm
\left(
\mathbf{\Omega}
\mathbf{L}^2
\right)^{-1}
\cdot
\mathbf{\Omega}
\mathbf{L}^2
 =
\mathbf{G}^\pm
$,
hence the proof is complete.
\end{proof}

Note from Definitions~\ref{definition_Mtilde}, \ref{definition_GMpm}, and \ref{definition_Gammapm}
that the maps $\Gamma^-$ and $\Gamma^+$ are constructive,
in that 
$\Gamma^-\left( [\mathbf{M}] \right) = \mathbf{G}_\mathbf{M}^-$, 
$\Gamma^+\left( [\mathbf{M}] \right) = \mathbf{G}_\mathbf{M}^+$
can be computed explicitly in terms of given $\mathbf{M} \in \mathrm{wp}(4,8,\mathbb{C})$.
In fact, 
the inverses $\left( \Gamma^- \right)^{-1}$, $\left( \Gamma^+ \right)^{-1}$
are also constructive.
Lemma~\ref{lemma_GMpmtoM} implies that
\begin{align}
\lefteqn{
\left( \Gamma^- \right)^{-1}(\mathbf{G})
}
\nonumber \\
 &=
\left[
\left(
\begin{array}{c|c}
\mathbf{G}
\left(
\mathbf{\Omega}
\mathbf{L}^2
\right)^{-1}
\mathbf{W}(-l)^{-1}
 &
\left( 
\mathbf{\Omega}
\mathbf{L}^2
-
\mathbf{G}
\right)
\left(
\mathbf{\Omega}
\mathbf{L}^2
\right)^{-1}
\mathbf{W}(l)^{-1}
\end{array}
\right)
\right],
\label{equation_GM-toM} \\
\lefteqn{
\left( \Gamma^+ \right)^{-1}(\mathbf{G})
}
\nonumber \\
 &=
\left[
\left(
\begin{array}{c|c}
\left( 
\mathbf{\Omega}
\mathbf{L}^2
-
\mathbf{G}
\right)
\left(
\mathbf{\Omega}
\mathbf{L}^2
\right)^{-1}
\mathbf{W}(-l)^{-1}
 &
\mathbf{G}
\left(
\mathbf{\Omega}
\mathbf{L}^2
\right)^{-1}
\mathbf{W}(l)^{-1}
\end{array}
\right)
\right]
\label{equation_GM+toM}
\end{align}
for every 
$\mathbf{G} \in \gl(4,\mathbb{C})$.

\subsection{Real boundary conditions and the algebra $\protect\overline{\pi}(4)$}
\label{section_real}

Of particular interest among boundary conditions in $\mathrm{wp}(4,8,\mathbb{C})$
are those with {\em real} entries.
We characterize
this important class of real boundary conditions
in terms of the maps $\Gamma^-$ and $\Gamma^+$
in Definition~\ref{definition_Gammapm}.

\begin{definition}
\label{definition_wpR}
Denote
$
\mathrm{wp}(4,8,\mathbb{R})
 =
\mathrm{wp}(4,8,\mathbb{C})
\cap
\gl(4,8,\mathbb{R})
$
and
\[
\mathrm{wp}(\mathbb{R})
 =
\left\{
\left[ \mathbf{M} \right]
\in \mathrm{wp}(\mathbb{C})
\,|\,
\mathbf{M}
\in \mathrm{wp}(4,8,\mathbb{R})
\right\}
 \subset
\mathrm{wp}(\mathbb{C}).
\]
\end{definition}

Let $\mathbf{M} \in \mathrm{wp}(4,8,\mathbb{R})$.
By Lemma~\ref{lemma_KM},
it is clear that
$
\mathcal{K}_\mathbf{M}[w](x)
$ is real-valued for every real-valued $w \in L^2[-l,l]$.
Thus it follows that
$G_\mathbf{M}(x,\xi)$ is real-valued for every $\mathbf{M} \in \mathrm{wp}(4,8,\mathbb{R})$.

\begin{lemma}
\label{lemma_MtoGMpmReal}
$
\mathbf{R}
\overline{\mathbf{G}_\mathbf{M}^-}
\mathbf{R}
 =
\mathbf{G}_\mathbf{M}^-
$
and
$
\mathbf{R}
\overline{\mathbf{G}_\mathbf{M}^+}
\mathbf{R}
 =
\mathbf{G}_\mathbf{M}^+
$
for every
$\mathbf{M} \in \mathrm{wp}(4,8,\mathbb{R})$.
\end{lemma}

\begin{proof}
Let $\mathbf{M} \in \mathrm{wp}(4,8,\mathbb{R})$.
Since $G_\mathbf{M}(x,\xi)$ is real-valued,
we have
$\overline{G_\mathbf{M}(x,\xi)} = G_\mathbf{M}(x,\xi)$ for $(x,\xi) \in [-l,l] \times [-l,l]$.
By Definition~\ref{definition_GMKM} and \eqref{equation_R},
we have
\begin{align*}
\overline{G_\mathbf{M}(x,\xi)}
 &=
\frac{\alpha}{4k}
\cdot
\left\{
\begin{array}{ll}
\overline{
\mathbf{y}(x)
}^T
\cdot
\overline{\mathbf{G}_\mathbf{M}^+}
\cdot
\overline{\mathbf{y}(\xi)},
 &
\text{if }
x \leq \xi, \\
-
\overline{
\mathbf{y}(x)
}^T
\cdot
\overline{\mathbf{G}_\mathbf{M}^-}
\cdot
\overline{\mathbf{y}(\xi)},
 &
\text{if }
\xi \leq x
\end{array}
\right.
\nonumber \\
 &=
\frac{\alpha}{4k}
\cdot
\left\{
\begin{array}{ll}
\left\{
\mathbf{R}
\cdot
\mathbf{y}(x)
\right\}^T
\cdot
\overline{\mathbf{G}_\mathbf{M}^+}
\cdot
\left\{
\mathbf{R}
\cdot
\mathbf{y}(\xi)
\right\},
 &
\text{if }
x \leq \xi, \\
-
\left\{
\mathbf{R}
\cdot
\mathbf{y}(x)
\right\}^T
\cdot
\overline{\mathbf{G}_\mathbf{M}^-}
\cdot
\left\{
\mathbf{R}
\cdot
\mathbf{y}(\xi)
\right\},
 &
\text{if }
\xi \leq x
\end{array}
\right.
\nonumber \\
 &=
\frac{\alpha}{4k}
\cdot
\left\{
\begin{array}{ll}
\mathbf{y}(x)^T
\cdot
\mathbf{R}
\overline{\mathbf{G}_\mathbf{M}^+}
\mathbf{R}
\cdot
\mathbf{y}(\xi),
 &
\text{if }
x \leq \xi, \\
-
\mathbf{y}(x)^T
\cdot
\mathbf{R}
\overline{\mathbf{G}_\mathbf{M}^-}
\mathbf{R}
\cdot
\mathbf{y}(\xi),
 &
\text{if }
\xi \leq x,
\end{array}
\right.
\end{align*}
hence 
\[
0
 =
\frac{4k}{\alpha}
\left\{
\overline{G_\mathbf{M}(x,\xi)}
-
G_\mathbf{M}(x,\xi)
\right\}
 =
\left\{
\begin{array}{ll}
\mathbf{y}(x)^T
\cdot
\left(
\mathbf{R}
\overline{\mathbf{G}_\mathbf{M}^+}
\mathbf{R}
-
\mathbf{G}_\mathbf{M}^+
\right)
\cdot
\mathbf{y}(\xi),
 &
\text{if }
x \leq \xi, \\
-
\mathbf{y}(x)^T
\cdot
\left(
\mathbf{R}
\overline{\mathbf{G}_\mathbf{M}^-}
\mathbf{R}
-
\mathbf{G}_\mathbf{M}^-
\right)
\cdot
\mathbf{y}(\xi),
 &
\text{if }
\xi \leq x, \\
\end{array}
\right.
\]
which is equivalent to
$
\mathbf{R}
\overline{\mathbf{G}_\mathbf{M}^+}
\mathbf{R}
-
\mathbf{G}_\mathbf{M}^+
 =
\mathbf{R}
\overline{\mathbf{G}_\mathbf{M}^-}
\mathbf{R}
-
\mathbf{G}_\mathbf{M}^-
 =
\mathbf{O}$,
since the entries $y_1, y_2, y_3, y_4$
of $\mathbf{y}$ in Definition~\ref{definition_W(x)}
are linearly independent.
\end{proof}

Lemma~\ref{lemma_MtoGMpmReal} leads us to the following definition.

\begin{definition}
\label{definition_pibar}
For $n \in \mathbb{N}$,
we denote 
$
\overline{\pi}(n)
 =
\left\{
\mathbf{A} \in \gl(n,\mathbb{C})
\,|\,
\mathbf{R}
\overline{\mathbf{A}}
\mathbf{R}
 =
\mathbf{A}
\right\}$.
\end{definition}

Note that
$\overline{\pi}(n)$ is the set of matrices in $\gl(n,\mathbb{C})$
invariant under the transformation
$
\mathbf{A}
\mapsto
\mathbf{R}
\overline{\mathbf{A}}
\mathbf{R}$,
which is the complex conjugation with the $180^\circ$ rotation of matrix entries.
Lemma~\ref{lemma_pibar} below,
whose proof is immediate from Definition~\ref{definition_pibar},
shows in particular that
$\overline{\pi}(n)$
forms an $\mathbb{R}$-algebra.

\begin{lemma}
\label{lemma_pibar}
For $n \in \mathbb{N}$, we have the following.
\begin{itemize}
\item[{\rm (a)}]
If $\mathbf{A}, \mathbf{B} \in \overline{\pi}(n)$,
then
$a \mathbf{A} + b \mathbf{B} \in \overline{\pi}(n)$
for every $a, b \in \mathbb{R}$.

\item[{\rm (b)}]
If $\mathbf{A}, \mathbf{B} \in \overline{\pi}(n)$,
then
$\mathbf{A} \mathbf{B} \in \overline{\pi}(n)$.

\item[{\rm (c)}]
If $\mathbf{A} \in \overline{\pi}(n)$ is invertible,
then
$\mathbf{A}^{-1} \in \overline{\pi}(n)$

\item[{\rm (d)}]
If $\mathbf{A} \in \overline{\pi}(n)$,
then
$\mathbf{A}^T \in \overline{\pi}(n)$

\item[{\rm (e)}]
$
\mathbf{O}_n, \mathbf{I}_n, \mathbf{R}_n
\in \overline{\pi}(n)$
and
$\mathbf{\Omega}, \mathbf{L}^2, \mathcal{E} \in \overline{\pi}(4)$.
\end{itemize}
\end{lemma}

In particular, we have
\begin{align*}
\overline{\pi}(4)
 &=
\left\{
\left.
\begin{pmatrix}
a_{11} & a_{12} & a_{13} & a_{14} \\
a_{21} & a_{22} & a_{23} & a_{24} \\
\overline{a_{24}} & \overline{a_{23}} & \overline{a_{22}} & \overline{a_{21}} \\
\overline{a_{14}} & \overline{a_{13}} & \overline{a_{12}} & \overline{a_{11}}
\end{pmatrix}
\,\right|\,
a_{ij} \in \mathbb{C},
\,
1=1,2,
\,
j = 1,2,3,4
\right\},
\end{align*}
which shows that the dimension of $\overline{\pi}(4)$
as an $\mathbb{R}$-algebra is $16$.
In fact, it will be shown in Section~\ref{section_existence}
that  
$\overline{\pi}(2n)$ is isomorphic to $\gl(2n,\mathbb{R})$ for $n \in \mathbb{N}$.

Lemma~\ref{lemma_MtoGMpmReal} shows that
the images of $\mathrm{wp}(\mathbb{R})$ under
$\Gamma^-$ and $\Gamma^+$ in Definition~\ref{definition_Gammapm} are subsets of
$\overline{\pi}(4)$.
Lemma~\ref{lemma_GMpmtoMReal} below shows that,
in fact,
$
\Gamma^-\left( \mathrm{wp}(\mathbb{R}) \right)
 =
\Gamma^+\left( \mathrm{wp}(\mathbb{R}) \right)
 =
\overline{\pi}(4)
$,
by constructing representatives
{\em in} $\mathrm{wp}(4,8,\mathbb{R})$
of the inverses $\left( \Gamma^- \right)^{-1}\left( \mathbf{G} \right)$,
$\left( \Gamma^+ \right)^{-1}\left( \mathbf{G} \right)$
of $\mathbf{G} \in \overline{\pi}(4)$.
Denote
\begin{equation}
\label{equation_U}
\mathbf{U}
 =
\frac{1}{\sqrt{2}}
\begin{pmatrix}
1 & 0 & 0 & 1 \\
0 & 1 & 1 & 0 \\
0 & \mathbbm{i} & -\mathbbm{i} & 0 \\
\mathbbm{i} & 0 & 0 & -\mathbbm{i}
\end{pmatrix}
 =
\begin{pmatrix}
\mathbf{I} & \mathbf{R} \\
\mathbbm{i} \mathbf{R} & -\mathbbm{i} \mathbf{I}
\end{pmatrix}
\qquad
\in U(4).
\end{equation}
Note that
\begin{equation}
\label{equation_Ubar=UR}
\overline{\mathbf{U}}
 = 
\mathbf{U}
\mathbf{R}.
\end{equation}

\begin{lemma}
\label{lemma_GMpmtoMReal}
Suppose $\mathbf{G}^-, \mathbf{G}^+ \in \overline{\pi}(4)$
satisfy $\mathbf{G}^- + \mathbf{G}^+ = \mathbf{\Omega} \mathbf{L}^2$.
Then there exists
$
\mathbf{M}
\in \mathrm{wp}(4,8,\mathbb{R})$
such that
$
\mathbf{G}_\mathbf{M}^-
 =
\mathbf{G}^-$,
$
\mathbf{G}_\mathbf{M}^+
 =
\mathbf{G}^+$.
In particular,
$\mathbf{M}$ can be taken by
$
\mathbf{M}
 =
\begin{pmatrix}
\mathbf{M}^- & \mathbf{M}^+
\end{pmatrix}
$,
where
$
\mathbf{M}^-
 =
\mathbf{U}
\mathbf{G}^-
\left(
\mathbf{\Omega}
\mathbf{L}^2
\right)^{-1}
\mathbf{W}(-l)^{-1}
$,
$
\mathbf{M}_+
 =
\mathbf{U}
\mathbf{G}^+
\left(
\mathbf{\Omega}
\mathbf{L}^2
\right)^{-1}
\mathbf{W}(l)^{-1}
$.
\end{lemma}

\begin{proof}
Let
$
\mathbf{M}^-
 =
\mathbf{U}
\mathbf{G}^-
\left(
\mathbf{\Omega}
\mathbf{L}^2
\right)^{-1}
\mathbf{W}(-l)^{-1}
$,
$
\mathbf{M}_+
 =
\mathbf{U}
\mathbf{G}^+
\left(
\mathbf{\Omega}
\mathbf{L}^2
\right)^{-1}
\mathbf{W}(l)^{-1}
$,
and
$
\mathbf{M}
 =
\begin{pmatrix}
\mathbf{M}^- & \mathbf{M}^+
\end{pmatrix}
$.
Let
\[
\hat{\mathbf{M}}
 =
\left(
\begin{array}{c|c}
\mathbf{G}^-
\left(
\mathbf{\Omega}
\mathbf{L}^2
\right)^{-1}
\mathbf{W}(-l)^{-1} &
\mathbf{G}^+
\left(
\mathbf{\Omega}
\mathbf{L}^2
\right)^{-1}
\mathbf{W}(l)^{-1}
\end{array}
\right).
\]
Then
$
\mathbf{M}
 =
\mathbf{U}
\hat{\mathbf{M}}
$,
hence
$\mathbf{M} \approx \hat{\mathbf{M}}$ by Lemma~\ref{lemma_equivalence},
since $\mathbf{U}$ is invertible.
So by Definition~\ref{definition_Gammapm} and Lemma~\ref{lemma_GMpmtoM},
we have
$
\mathbf{G}_\mathbf{M}^\pm
 =
\Gamma^\pm
\left(
\left[ \mathbf{M} \right]
\right)
 =
\Gamma^\pm
\left(
\left[ \hat{\mathbf{M}} \right]
\right)
 =
\mathbf{G}_{\hat{\mathbf{M}}}^\pm
 =
\mathbf{G}^\pm
$,
since $\mathbf{G}^-, \mathbf{G}^+ \in \gl(4,\mathbb{C})$.
Thus it is sufficient to show that
$\mathbf{M}^-, \mathbf{M}^+ \in \gl(4,\mathbb{R})$.
By Definition~\ref{definition_pibar},
\begin{equation}
\label{equation_jbsgbjs}
\overline{
\mathbf{G}^\pm
\left(
\mathbf{\Omega}
\mathbf{L}^2
\right)^{-1}
}
 =
\mathbf{R}
\cdot
\mathbf{G}^\pm
\left(
\mathbf{\Omega}
\mathbf{L}^2
\right)^{-1}
\cdot
\mathbf{R},
\end{equation}
since
$
\mathbf{G}^-
\left(
\mathbf{\Omega}
\mathbf{L}^2
\right)^{-1},
\mathbf{G}^+
\left(
\mathbf{\Omega}
\mathbf{L}^2
\right)^{-1}
\in \overline{\pi}(4)$
by Lemma~\ref{lemma_pibar}.
Since
$\overline{\mathbf{W}(x)}
 =
\mathbf{W}(x)
\mathbf{R}$
by \eqref{equation_R},
we have
\begin{equation}
\label{equation_bgbjfv}
\overline{\mathbf{W}(\pm l)^{-1}}
 =
\left\{
\overline{\mathbf{W}(\pm l)}
\right\}^{-1}
 =
\left\{
\mathbf{W}(\pm l)
\mathbf{R}
\right\}^{-1}
 =
\mathbf{R}
\cdot
\mathbf{W}(\pm l)^{-1}.
\end{equation}
Thus by \eqref{equation_Ubar=UR}, \eqref{equation_jbsgbjs}, \eqref{equation_bgbjfv},
we have
\begin{align*}
\overline{\mathbf{M}^\pm}
 &=
\overline{\mathbf{U}}
\cdot
\overline{
\mathbf{G}^\pm
\left(
\mathbf{\Omega}
\mathbf{L}^2
\right)^{-1}
}
\cdot
\overline{
\mathbf{W}(\pm l)^{-1}
} \\
 &=
\mathbf{U}
\mathbf{R}
\cdot
\left\{
\mathbf{R}
\cdot
\mathbf{G}^\pm
\left(
\mathbf{\Omega}
\mathbf{L}^2
\right)^{-1}
\cdot
\mathbf{R}
\right\}
\cdot
\left\{
\mathbf{R}
\cdot
\mathbf{W}(\pm l)^{-1}
\right\} \\
 &=
\mathbf{U}
\mathbf{G}^\pm
\left(
\mathbf{\Omega}
\mathbf{L}^2
\right)^{-1}
\mathbf{W}(\pm l)^{-1}
 =
\mathbf{M}^\pm.
\end{align*}
This shows
$\mathbf{M}^-, \mathbf{M}^+ \in \gl(4,\mathbb{R})$,
and the proof is complete.
\end{proof}

\section{The boundary condition $\protect\mathbf{Q}$ 
and the operator $\protect\mathcal{K}_\mathbf{Q} = \mathcal{K}_{l,\alpha,k}$}
\label{section_Q}

Let
\[
\mathbf{Q}^-
 =
\begin{pmatrix}
0 & \alpha^2 & -\sqrt{2} \alpha & 1 \\
\sqrt{2} \alpha^3 & -\alpha^2 & 0 & 1 \\
0 & 0 & 0 & 0 & \\
0 & 0 & 0 & 0 &
\end{pmatrix},
\
\mathbf{Q}^+
 =
\begin{pmatrix}
0 & 0 & 0 & 0 \\
0 & 0 & 0 & 0 \\
0 & \alpha^2 & \sqrt{2} \alpha & 1 \\
-\sqrt{2} \alpha^3 & -\alpha^2 & 0 & 1
\end{pmatrix},
\]
so that 
$
\begin{pmatrix}
\mathbf{Q}^- & \mathbf{Q}^+
\end{pmatrix}
=
\mathbf{Q} 
$
in \eqref{equation_Q}.
In this section, we apply Definitions~\ref{definition_Mtilde} and \ref{definition_GMpm} to $\mathbf{Q}$
to obtain
explicit forms of 
$
\mathbf{G}_\mathbf{Q}^-
 =
\Gamma^-\left( \left[ \mathbf{Q} \right] \right)
$
and 
$
\mathbf{G}_\mathbf{Q}^+
 =
\Gamma^+\left( \left[ \mathbf{Q} \right] \right)
$.
In addition to being needed to construct the map $\Gamma$ in Section~\ref{section_Gamma},
this will also serve as a concrete example of computing $\Gamma^-$ and $\Gamma^+$.
We also show in Lemma~\ref{lemma_GQ=Glalphak} below that
$
\mathcal{K}_\mathbf{Q}
 =
\mathcal{K}_{l,\alpha,k}
$,
where $\mathcal{K}_{l,\alpha,k}$
is the integral operator defined in \eqref{equation_Klalphak}.

By Definition~\ref{definition_omega}, we have
\begin{align}
\lefteqn{
\mathbf{Q}^-
\cdot
\diag
\left( 1, \alpha, \alpha^2, \alpha^3 \right)
\cdot
\mathbf{W}_0
}
\nonumber \\
 &=
\alpha^3
\begin{pmatrix}
0 & 1 & -\sqrt{2} & 1 \\
\sqrt{2} & -1 & 0 & 1 \\
0 & 0 & 0 & 0 & \\
0 & 0 & 0 & 0 &
\end{pmatrix}
\left(
\omega_j^{i-1}
\right)_{1 \leq i,j \leq 4}
 =
\alpha^3
\begin{pmatrix}
a_1^- & a_2^- & a_3^- & a_4^- \\
b_1^- & b_2^- & b_3^- & b_4^- \\
0 & 0 & 0 & 0 \\
0 & 0 & 0 & 0
\end{pmatrix},
\label{equation_sbdfgbl;jk-} \\
\lefteqn{
\mathbf{Q}^+
\cdot
\diag
\left( 1, \alpha, \alpha^2, \alpha^3 \right)
\cdot
\mathbf{W}_0
}
\nonumber \\
 &=
\alpha^3
\begin{pmatrix}
0 & 0 & 0 & 0 \\
0 & 0 & 0 & 0 \\
0 & 1 & \sqrt{2} & 1 \\
-\sqrt{2} & -1 & 0 & 1
\end{pmatrix}
\left(
\omega_j^{i-1}
\right)_{1 \leq i,j \leq 4}
 =
\alpha^3
\begin{pmatrix}
0 & 0 & 0 & 0 \\
0 & 0 & 0 & 0 \\
a_1^+ & a_2^+ & a_3^+ & a_4^+ \\
b_1^+ & b_2^+ & b_3^+ & b_4^+
\end{pmatrix},
\label{equation_sbdfgbl;jk+}
\end{align}
where we put
$
a_j^\pm
 =
\omega_j \pm \sqrt{2} \omega_j^2 + \omega_j^3
$,
$
b_j^\pm
 =
\mp \sqrt{2} - \omega_j + \omega_j^3
$,
$j = 1,2,3,4$.
Note that
$
\omega_j^2
 = 
\left\{
e^{\mathbbm{i} \frac{\pi}{4} (2j - 1)}
\right\}^2
 =
\mathbbm{i}^{2j - 1}
 =
(-1)^{j+1}
\mathbbm{i}
$
by Definition~\ref{definition_omega},
and
$
\omega_j + \omega_j^3
 =
\omega_j - \overline{\omega_j}
 =
2
\mathbbm{i}
\Imag{\omega_j}
 =
\sqrt{2}
\epsilon_{j-1}
\mathbbm{i}
$, 
$
-\omega_j + \omega_j^3
 =
-\omega_j - \overline{\omega_j}
 =
-
2
\Real{\omega_j}
 =
-
\sqrt{2}
\epsilon_j
$
by \eqref{equation_omega}, \eqref{equation_omegaReIm}.
Hence, for $j = 1,2,3,4$,
we have
\begin{align}
a_j^\pm
 &=
\omega_j \pm \sqrt{2} \omega_j^2 + \omega_j^3
 =
\sqrt{2}
\epsilon_{j-1}
\mathbbm{i}
\pm
\sqrt{2}
\cdot
(-1)^{j+1}
\mathbbm{i}
\nonumber \\
 &=
2 \sqrt{2}
\cdot
\frac
{
\epsilon_{j-1}
\pm
(-1)^{j+1}
}
{2}
\mathbbm{i},
\label{equation_a;sdova;sdvjh} \\
b_j^\pm
 &=
\mp \sqrt{2} - \omega_j + \omega_j^3
 =
-
\sqrt{2}
\epsilon_j
\mp
\sqrt{2}
 =
2 \sqrt{2}
\cdot
\frac
{
-
\epsilon_j
\mp
1}
{2}.
\label{equation_svfvjhasdfoi}
\end{align}
By \eqref{equation_sbdfgbl;jk-}, \eqref{equation_sbdfgbl;jk+},
\eqref{equation_a;sdova;sdvjh}, \eqref{equation_svfvjhasdfoi},
we have
\begin{align*}
\lefteqn{
\mathbf{Q}^-
\cdot
\diag
\left( 1, \alpha, \alpha^2, \alpha^3 \right)
\cdot
\mathbf{W}_0
} \\
 &=
2 \sqrt{2}
\alpha^3
\begin{pmatrix}
\frac{1 - 1}{2} \mathbbm{i}
 & 
\frac{1 + 1}{2} \mathbbm{i}
 & 
\frac{-1 - 1}{2} \mathbbm{i}
 & 
\frac{-1 + 1}{2} \mathbbm{i}
 \\
\frac{-1 + 1}{2}
 & 
\frac{1 + 1}{2}
 & 
\frac{1 + 1}{2}
 & 
\frac{-1 + 1}{2}
 \\
0 & 0 & 0 & 0 \\ 
0 & 0 & 0 & 0
\end{pmatrix}
 =
2 \sqrt{2}
\alpha^3
\begin{pmatrix}
0 & \mathbbm{i} & -\mathbbm{i} & 0 \\
0 & 1 & 1 & 0 \\
0 & 0 & 0 & 0 \\ 
0 & 0 & 0 & 0
\end{pmatrix}, \\
\lefteqn{
\mathbf{Q}^+
\cdot
\diag
\left( 1, \alpha, \alpha^2, \alpha^3 \right)
\cdot
\mathbf{W}_0
} \\
 &=
2 \sqrt{2}
\alpha^3
\begin{pmatrix}
0 & 0 & 0 & 0 \\ 
0 & 0 & 0 & 0 \\
\frac{1 + 1}{2} \mathbbm{i}
 & 
\frac{1 - 1}{2} \mathbbm{i}
 & 
\frac{-1 + 1}{2} \mathbbm{i}
 & 
\frac{-1 - 1}{2} \mathbbm{i}
 \\
\frac{-1 - 1}{2}
 & 
\frac{1 - 1}{2}
 & 
\frac{1 - 1}{2}
 & 
\frac{-1 - 1}{2}
\end{pmatrix}
 =
2 \sqrt{2}
\alpha^3
\begin{pmatrix}
0 & 0 & 0 & 0 \\ 
0 & 0 & 0 & 0 \\
\mathbbm{i} & 0 & 0 & -\mathbbm{i} \\
-1 & 0 & 0 & -1
\end{pmatrix},
\end{align*}
hence by Definition~\ref{definition_Mtilde} and \eqref{equation_Wxdecomposed},
\begin{align}
\lefteqn{
\widetilde{\mathbf{Q}}^-
 =
\mathbf{Q}^-
\mathbf{W}(-l)
 =
\mathbf{Q}^-
\cdot
\diag
\left( 1, \alpha, \alpha^2, \alpha^3 \right)
\cdot
\mathbf{W}_0
e^{-\mathbf{\Omega} \alpha l}
}
\nonumber \\
 &=
2 \sqrt{2}
\alpha^3
\begin{pmatrix}
0 & \mathbbm{i} & -\mathbbm{i} & 0 \\
0 & 1 & 1 & 0 \\
0 & 0 & 0 & 0 \\ 
0 & 0 & 0 & 0
\end{pmatrix}
e^{-\mathbf{\Omega} \alpha l}
 =
2 \sqrt{2}
\alpha^3
\begin{pmatrix}
0 & \mathbbm{i} e^{-\omega_2 \alpha l} & -\mathbbm{i} e^{-\omega_3 \alpha l} & 0 \\
0 & e^{-\omega_2 \alpha l} & e^{-\omega_3 \alpha l} & 0 \\
0 & 0 & 0 & 0 \\ 
0 & 0 & 0 & 0
\end{pmatrix},
\label{equation_o;sdifgbj-} \\
\lefteqn{
\widetilde{\mathbf{Q}}^+
 =
\mathbf{Q}^+
\mathbf{W}(l)
 =
\mathbf{Q}^+
\cdot
\diag
\left( 1, \alpha, \alpha^2, \alpha^3 \right)
\cdot
\mathbf{W}_0
e^{\mathbf{\Omega} \alpha l}
}
\nonumber \\
 &=
2 \sqrt{2}
\alpha^3
\begin{pmatrix}
0 & 0 & 0 & 0 \\ 
0 & 0 & 0 & 0 \\
\mathbbm{i} & 0 & 0 & -\mathbbm{i} \\
-1 & 0 & 0 & -1
\end{pmatrix}
e^{\mathbf{\Omega} \alpha l}
 =
2 \sqrt{2}
\alpha^3
\begin{pmatrix}
0 & 0 & 0 & 0 \\ 
0 & 0 & 0 & 0 \\
\mathbbm{i} e^{\omega_1 \alpha l} & 0 & 0 & -\mathbbm{i} e^{\omega_4 \alpha l} \\
-e^{\omega_1 \alpha l} & 0 & 0 & -e^{\omega_4 \alpha l}
\end{pmatrix}, 
\label{equation_o;sdifgbj+}
\end{align}
\begin{equation}
\label{equation_o;sdifgbj}
\widetilde{\mathbf{Q}}
 =
\widetilde{\mathbf{Q}}^-
+
\widetilde{\mathbf{Q}}^+
 =
2 \sqrt{2}
\alpha^3
\begin{pmatrix}
0 & \mathbbm{i} e^{-\omega_2 \alpha l} & -\mathbbm{i} e^{-\omega_3 \alpha l} & 0 \\
0 & e^{-\omega_2 \alpha l} & e^{-\omega_3 \alpha l} & 0 \\
\mathbbm{i} e^{\omega_1 \alpha l} & 0 & 0 & -\mathbbm{i} e^{\omega_4 \alpha l} \\
-e^{\omega_1 \alpha l} & 0 & 0 & -e^{\omega_4 \alpha l}
\end{pmatrix}.
\end{equation}
Let
\[
\hat{\mathbf{U}}
 =
\frac{1}{\sqrt{2}}
\begin{pmatrix}
0 & \mathbbm{i} & -\mathbbm{i} & 0 \\
0 & 1 & 1 & 0 \\
\mathbbm{i} & 0 & 0 & -\mathbbm{i} \\
-1 & 0 & 0 & -1
\end{pmatrix}
\quad
\in U(4).
\]
Then by \eqref{equation_o;sdifgbj} and Definition~\ref{definition_epsilon},
\begin{align*}
\widetilde{\mathbf{Q}}
 &=
4
\alpha^3
\cdot
\frac{1}{\sqrt{2}}
\begin{pmatrix}
0 & \mathbbm{i} & -\mathbbm{i} & 0 \\
0 & 1 & 1 & 0 \\
\mathbbm{i} & 0 & 0 & -\mathbbm{i} \\
-1 & 0 & 0 & -1
\end{pmatrix}
\cdot
\diag
\left(
e^{\omega_1 \alpha l},
e^{-\omega_2 \alpha l},
e^{-\omega_3 \alpha l},
e^{\omega_4 \alpha l}
\right) \\
 &=
4
\alpha^3
\hat{\mathbf{U}}
e^{\mathcal{E} \mathbf{\Omega} \alpha l},
\end{align*}
hence we have
\begin{equation}
\label{equation_Qtilde-1}
\widetilde{\mathbf{Q}}^{-1}
 =
\frac{1}{4 \alpha^3}
e^{-\mathcal{E} \mathbf{\Omega} \alpha l}
\hat{\mathbf{U}}^{-1}
 =
\frac{1}{4 \alpha^3}
e^{-\mathcal{E} \mathbf{\Omega} \alpha l}
\hat{\mathbf{U}}^*,
\end{equation}
since $\hat{\mathbf{U}}$ is unitary.
Note that this in particular shows that
$\mathbf{Q}$ is well-posed
by Lemma~\ref{lemma_wellposed}.

Let
\[
\hat{\mathbf{U}}^-
 =
\frac{1}{\sqrt{2}}
\begin{pmatrix}
0 & \mathbbm{i} & -\mathbbm{i} & 0 \\
0 & 1 & 1 & 0 \\
0 & 0 & 0 & 0 \\ 
0 & 0 & 0 & 0
\end{pmatrix},
\qquad
\hat{\mathbf{U}}^+
 =
\frac{1}{\sqrt{2}}
\begin{pmatrix}
0 & 0 & 0 & 0 \\ 
0 & 0 & 0 & 0 \\
\mathbbm{i} & 0 & 0 & -\mathbbm{i} \\
-1 & 0 & 0 & -1
\end{pmatrix}.
\]
Then by \eqref{equation_o;sdifgbj-}, \eqref{equation_o;sdifgbj+},
\begin{align}
\widetilde{\mathbf{Q}}^-
 &=
4
\alpha^3
\cdot
\frac{1}{\sqrt{2}}
\begin{pmatrix}
0 & \mathbbm{i} & -\mathbbm{i} & 0 \\
0 & 1 & 1 & 0 \\
0 & 0 & 0 & 0 \\
0 & 0 & 0 & 0
\end{pmatrix}
\cdot
\diag
\left(
e^{\omega_1 \alpha l},
e^{-\omega_2 \alpha l},
e^{-\omega_3 \alpha l},
e^{\omega_4 \alpha l}
\right)
\nonumber \\
 &=
4
\alpha^3
\hat{\mathbf{U}}^-
e^{\mathcal{E} \mathbf{\Omega} \alpha l},
\label{equation_Qtilde-} \\
\widetilde{\mathbf{Q}}^+
 &=
4
\alpha^3
\cdot
\frac{1}{\sqrt{2}}
\begin{pmatrix}
0 & 0 & 0 & 0 \\ 
0 & 0 & 0 & 0 \\
\mathbbm{i} & 0 & 0 & -\mathbbm{i} \\
-1 & 0 & 0 & -1
\end{pmatrix}
\cdot
\diag
\left(
e^{\omega_1 \alpha l},
e^{-\omega_2 \alpha l},
e^{-\omega_3 \alpha l},
e^{\omega_4 \alpha l}
\right)
\nonumber \\
 &=
4
\alpha^3
\hat{\mathbf{U}}^+
e^{\mathcal{E} \mathbf{\Omega} \alpha l}.
\label{equation_Qtilde+}
\end{align}
By \eqref{equation_Qtilde-1}, \eqref{equation_Qtilde-}, \eqref{equation_Qtilde+},
we have
\begin{equation}
\label{equation_;asoifjvbasdf}
\widetilde{\mathbf{Q}}^{-1}
\widetilde{\mathbf{Q}}^\pm
 =
\frac{1}{4 \alpha^3}
e^{-\mathcal{E} \mathbf{\Omega} \alpha l}
\hat{\mathbf{U}}^*
\cdot
4 \alpha^3
\hat{\mathbf{U}}^\pm
e^{\mathcal{E} \mathbf{\Omega} \alpha l}
 =
e^{-\mathcal{E} \mathbf{\Omega} \alpha l}
\hat{\mathbf{U}}^*
\hat{\mathbf{U}}^\pm
e^{\mathcal{E} \mathbf{\Omega} \alpha l}.
\end{equation}
Note that 
\begin{align*}
\hat{\mathbf{U}}^*
\hat{\mathbf{U}}^-
 &=
\frac{1}{\sqrt{2}}
\begin{pmatrix}
0 & 0 & -\mathbbm{i} & -1 \\
-\mathbbm{i} & 1 & 0 & 0 \\
\mathbbm{i} & 1 & 0 & 0 \\
0 & 0 & \mathbbm{i} & -1
\end{pmatrix}
\cdot
\frac{1}{\sqrt{2}}
\begin{pmatrix}
0 & \mathbbm{i} & -\mathbbm{i} & 0 \\
0 & 1 & 1 & 0 \\
0 & 0 & 0 & 0 \\ 
0 & 0 & 0 & 0
\end{pmatrix}
 =
\diag(0,1,1,0), \\
\hat{\mathbf{U}}^*
\hat{\mathbf{U}}^+
 &=
\frac{1}{\sqrt{2}}
\begin{pmatrix}
0 & 0 & -\mathbbm{i} & -1 \\
-\mathbbm{i} & 1 & 0 & 0 \\
\mathbbm{i} & 1 & 0 & 0 \\
0 & 0 & \mathbbm{i} & -1
\end{pmatrix}
\cdot
\frac{1}{\sqrt{2}}
\begin{pmatrix}
0 & 0 & 0 & 0 \\ 
0 & 0 & 0 & 0 \\
\mathbbm{i} & 0 & 0 & -\mathbbm{i} \\
-1 & 0 & 0 & -1
\end{pmatrix}
 =
\diag(1,0,0,1),
\end{align*}
hence by \eqref{equation_;asoifjvbasdf},
\begin{align}
\widetilde{\mathbf{Q}}^{-1}
\widetilde{\mathbf{Q}}^-
 &=
e^{-\mathcal{E} \mathbf{\Omega} \alpha l}
\cdot
\diag(0,1,1,0)
\cdot
e^{\mathcal{E} \mathbf{\Omega} \alpha l}
 =
\diag(0,1,1,0),
\label{equation_Q-1Q-} \\
\widetilde{\mathbf{Q}}^{-1}
\widetilde{\mathbf{Q}}^+
 &=
e^{-\mathcal{E} \mathbf{\Omega} \alpha l}
\cdot
\diag(1,0,0,1)
\cdot
e^{\mathcal{E} \mathbf{\Omega} \alpha l}
 =
\diag(1,0,0,1).
\label{equation_Q-1Q+}
\end{align}
Thus by Definition~\ref{definition_GMpm},
we finally have
\begin{align}
\mathbf{G}_\mathbf{Q}^-
 &=
\widetilde{\mathbf{Q}}^{-1}
\widetilde{\mathbf{Q}}^-
\mathbf{\Omega}
\mathbf{L}^2
\nonumber \\
 &=
\diag(0,1,1,0)
\cdot
\begin{pmatrix}
0 & 0 & \omega_1 & 0 \\
0 & 0 & 0 & \omega_2 \\
\omega_3 & 0 & 0 & 0 \\
0 & \omega_4 & 0 & 0
\end{pmatrix}
 =
\begin{pmatrix}
0 & 0 & 0 & 0 \\
0 & 0 & 0 & \omega_2 \\
\omega_3 & 0 & 0 & 0 \\
0 & 0 & 0 & 0
\end{pmatrix},
\label{equation_GQ-} \\
\mathbf{G}_\mathbf{Q}^+
 &=
\widetilde{\mathbf{Q}}^{-1}
\widetilde{\mathbf{Q}}^+
\mathbf{\Omega}
\mathbf{L}^2
\nonumber \\
 &=
\diag(1,0,0,1)
\cdot
\begin{pmatrix}
0 & 0 & \omega_1 & 0 \\
0 & 0 & 0 & \omega_2 \\
\omega_3 & 0 & 0 & 0 \\
0 & \omega_4 & 0 & 0
\end{pmatrix}
 =
\begin{pmatrix}
0 & 0 & \omega_1 & 0 \\
0 & 0 & 0 & 0 \\
0 & 0 & 0 & 0 \\
0 & \omega_4 & 0 & 0
\end{pmatrix}.
\label{equation_GQ+}
\end{align}
Note that 
$\mathbf{Q} \in \mathrm{wp}(4,8,\mathbb{R})$
and
$\mathbf{G}_\mathbf{Q}^-, \mathbf{G}_\mathbf{Q}^+ \in \overline{\pi}(4)$,
satisfying Lemma~\ref{lemma_MtoGMpmReal}.

\begin{lemma}
\label{lemma_GQ=Glalphak}
$
\mathcal{K}_\mathbf{Q}
 =
\mathcal{K}_{l,\alpha,k}$.
\end{lemma}

\begin{proof}
By \eqref{equation_Klalphak} and Definition~\ref{definition_GMKM},
it is sufficient to show that $G_\mathbf{Q} = G$,
where $G$ is defined by \eqref{equation_Glalphak}.
By Definition~\ref{definition_W(x)}, 
and \eqref{equation_omegaL2}, \eqref{equation_GQ-}, \eqref{equation_GQ+},
we have
\begin{align*}
\mathbf{y}(x)^T
\cdot
\mathbf{G}_\mathbf{Q}^-
\cdot
\mathbf{y}(\xi)
 &=
\omega_3
e^{\omega_3 \alpha x}
e^{\omega_1 \alpha \xi}
+
\omega_2
e^{\omega_2 \alpha x}
e^{\omega_4 \alpha \xi} \\
 &=
-
\omega_1
e^{-\omega_1 \alpha (x - \xi)}
-
\omega_4
e^{-\omega_4 \alpha (x - \xi)}, \\
\mathbf{y}(x)^T
\cdot
\mathbf{G}_\mathbf{Q}^+
\cdot
\mathbf{y}(\xi)
 &=
\omega_1
e^{\omega_1 \alpha x}
e^{\omega_3 \alpha \xi}
+
\omega_4
e^{\omega_4 \alpha x}
e^{\omega_2 \alpha \xi} \\
 &=
\omega_1
e^{-\omega_1 \alpha (\xi - x)}
+
\omega_4
e^{-\omega_4 \alpha (\xi - x)},
\end{align*}
hence by Definitions~\ref{definition_omega}, \ref{definition_GMKM}, and \eqref{equation_omegaR},
\begin{align*}
\lefteqn{
G_\mathbf{Q}(x,\xi)
 =
\frac{\alpha}{4k}
\cdot
\left\{
\begin{array}{ll}
\omega_1
e^{-\omega_1 \alpha (\xi - x)}
+
\omega_4
e^{-\omega_4 \alpha (\xi - x)}, &
x \leq \xi \\
\omega_1
e^{-\omega_1 \alpha (x - \xi)}
+
\omega_4
e^{-\omega_4 \alpha (x - \xi)}, &
\xi \leq x
\end{array}
\right.
} \\
 &=
\frac{\alpha}{4k}
\left(
\omega_1
e^{-\omega_1 \alpha |x - \xi|}
+
\omega_4
e^{-\omega_4 \alpha |x - \xi|}
\right)
 =
\frac{\alpha}{4k}
\cdot
2
\Real
\left(
\omega_1
e^{-\omega_1 \alpha |x - \xi|}
\right) \\
 &=
\frac{\alpha}{2k}
\Real
\left(
e^{\mathbbm{i} \frac{\pi}{4}}
e^{-\left( \frac{1}{\sqrt{2}} + \mathbbm{i} \frac{1}{\sqrt{2}} \right) \alpha |x - \xi|}
\right)
 =
\frac{\alpha}{2k}
\Real
e^{
-
\frac{\alpha}{\sqrt{2}} |x - \xi|
+
\mathbbm{i}
\left(
\frac{\pi}{4}
-
\frac{\alpha}{\sqrt{2}} |x - \xi|
\right)
} \\
 &=
\frac{\alpha}{2k}
e^{
-
\frac{\alpha}{\sqrt{2}} |x - \xi|
}
\cos
\left(
\frac{\pi}{4}
-
\frac{\alpha}{\sqrt{2}} |x - \xi|
\right)
 =
\frac{\alpha}{2k}
e^{
-
\frac{\alpha}{\sqrt{2}} |x - \xi|
}
\sin
\left(
\frac{\alpha}{\sqrt{2}} |x - \xi|
+
\frac{\pi}{4}
\right),
\end{align*}
which is identical to $G(x,\xi)$ in \eqref{equation_Glalphak}.
Thus we have the proof.
\end{proof}

\section{The representation $\protect\Gamma$ and proof of Theorem~\ref{theorem_eigencondition-X}}

\subsection{The representation $\protect\Gamma$}
\label{section_Gamma}

\begin{definition}
\label{definition_GM}
For $\mathbf{M} \in \mathrm{wp}(4,8,\mathbb{C})$,
denote
$
\mathbf{G}_\mathbf{M}
 =
\left(
\mathbf{G}_\mathbf{M}^+
-
\mathbf{G}_\mathbf{Q}^+
\right)
\left(
\mathbf{\Omega}
\mathbf{L}^2
\right)^{-1}
\mathcal{E}
$.
Define
$\Gamma : \mathrm{wp}(\mathbb{C}) \to \gl(4,\mathbb{C})$
by
$
\Gamma\left( \left[ \mathbf{M} \right] \right)
 =
\mathbf{G}_\mathbf{M}
$
for $\mathbf{M} \in \mathrm{wp}(4,8,\mathbb{C})$.
\end{definition}

Readers should be cautious to distinguish the {\em $4 \times 4$ matrix} $\mathbf{G}_\mathbf{M}$
in Definition~\ref{definition_GM}
from the {\em Green's function} $G_\mathbf{M}$ in Definition~\ref{definition_GMKM}.
Note that the map $\Gamma$ is well-defined,
since
\begin{equation}
\label{equation_GammaGamma+}
\Gamma\left( \left[ \mathbf{M} \right] \right)
 =
\left(
\mathbf{G}_\mathbf{M}^+
-
\mathbf{G}_\mathbf{Q}^+
\right)
\left(
\mathbf{\Omega}
\mathbf{L}^2
\right)^{-1}
\mathcal{E}
 =
\left\{
\Gamma^+\left( \left[ \mathbf{M} \right] \right)
-
\mathbf{G}_\mathbf{Q}^+
\right\}
\left(
\mathbf{\Omega}
\mathbf{L}^2
\right)^{-1}
\mathcal{E}
\end{equation}
by Definition~\ref{definition_Gammapm}.
By \eqref{equation_GM-+GM+},
\[
\mathbf{G}_\mathbf{M}
 =
\left\{
\left(
\mathbf{\Omega}
\mathbf{L}^2
-
\mathbf{G}_\mathbf{M}^-
\right)
-
\left(
\mathbf{\Omega}
\mathbf{L}^2
-
\mathbf{G}_\mathbf{Q}^-
\right)
\right\}
\left(
\mathbf{\Omega}
\mathbf{L}^2
\right)^{-1}
\mathcal{E}
 =
-
\left(
\mathbf{G}_\mathbf{M}^-
-
\mathbf{G}_\mathbf{Q}^-
\right)
\left(
\mathbf{\Omega}
\mathbf{L}^2
\right)^{-1}
\mathcal{E},
\]
which could have been used for an alternative definition of $\mathbf{G}_\mathbf{M}$.
Note also that
\begin{equation}
\label{equation_GQ=O}
\mathbf{G}_\mathbf{Q}
 =
\mathbf{O}.
\end{equation}

\begin{lemma}
\label{lemma_GM}
$\Gamma : \mathrm{wp}(\mathbb{C}) \to \gl(4,\mathbb{C})$
is a one-to-one correspondence,
and
$
\Gamma\left( \mathrm{wp}(\mathbb{R}) \right)
 =
\overline{\pi}(4)
$.
\end{lemma}

\begin{proof}
Since $\Gamma^+$ 
is a one-to-one correspondence
and 
$
\left(
\mathbf{\Omega}
\mathbf{L}^2
\right)^{-1}
\mathcal{E}
$ is invertible,
it follows from \eqref{equation_GammaGamma+}
that $\Gamma$ also is a one-to-one correspondence.
Since 
$
\Gamma^+\left( \mathrm{wp}(4,\mathbb{R}) \right)
$
$
 =
\overline{\pi}(4)
$
and
$
\mathbf{G}_\mathbf{Q}^+,
\left(
\mathbf{\Omega}
\mathbf{L}^2
\right)^{-1}
\mathcal{E}
\in \overline{\pi}(4)
$,
we have
$
\Gamma\left( \mathrm{wp}(4,\mathbb{R}) \right)
 =
\overline{\pi}(4)
$
by \eqref{equation_GammaGamma+}
and Lemma~\ref{lemma_pibar}.
\end{proof}

Thus we finally have our representation 
$\mathbf{M} \mapsto \mathbf{G}_\mathbf{M}$ ,
from the set of well-posed boundary conditions
$\mathrm{wp}(4,8,\mathbb{C})$ to the algebra $\gl(4,\mathbb{C})$,
and from
the set of well-posed {\em real} boundary conditions $\mathrm{wp}(4,8,\mathbb{R})$
to the $\mathbb{R}$-algebra $\overline{\pi}(4)$.
Note that this representation is {\em constructive} in both directions,
in that $\mathbf{G}_\mathbf{M} \in \gl(4,\mathbb{C})$ is expressed explicitly
in terms of given $\mathbf{M} \in \mathrm{wp}(4,8,\mathbb{C})$,
and conversely,
$\mathbf{M} \in \mathrm{wp}(4,8,\mathbb{C})$ 
such that $\mathbf{G}_\mathbf{M} = \mathbf{G}$
can be chosen explicitly
in terms of given $\mathbf{G} \in \gl(4,\mathbb{C})$.
Especially,
given 
$
\mathbf{M}
 =
\begin{pmatrix}
\mathbf{M}^- & \mathbf{M}^+
\end{pmatrix}
\in \mathrm{wp}(4,8,\mathbb{C})
$,
$
\mathbf{M}^-,
\mathbf{M}^+
\in \gl(4,\mathbb{C})
$,
we have
\begin{align}
\mathbf{G}_\mathbf{M}
 &=
\left(
\mathbf{G}_\mathbf{M}^+
-
\mathbf{G}_\mathbf{Q}^+
\right)
\left(
\mathbf{\Omega}
\mathbf{L}^2
\right)^{-1}
\mathcal{E}
\nonumber \\
 &=
\left(
\widetilde{\mathbf{M}}^{-1}
\widetilde{\mathbf{M}}^+
\mathbf{\Omega}
\mathbf{L}^2
-
\widetilde{\mathbf{Q}}^{-1}
\widetilde{\mathbf{Q}}^+
\mathbf{\Omega}
\mathbf{L}^2
\right)
\left(
\mathbf{\Omega}
\mathbf{L}^2
\right)^{-1}
\mathcal{E}
 =
\left(
\widetilde{\mathbf{M}}^{-1}
\widetilde{\mathbf{M}}^+
-
\widetilde{\mathbf{Q}}^{-1}
\widetilde{\mathbf{Q}}^+
\right)
\mathcal{E}
\nonumber \\
 &=
\left\{
\mathbf{M}^-
\mathbf{W}(-l)
+
\mathbf{M}^+
\mathbf{W}(l)
\right\}^{-1}
\mathbf{M}^+
\mathbf{W}(l)
\mathcal{E}
-
\diag(1,0,0,1)
\label{equation_MtoGM}
\end{align}
by combining Definitions~\ref{definition_Mtilde}, \ref{definition_GMpm}, \ref{definition_GM},
and \eqref{equation_Q-1Q+}.
Conversely, 
suppose $\mathbf{G} \in \gl(4,\mathbb{C})$ is given.
By \eqref{equation_GammaGamma+},
we have
$
\Gamma^{-1}(\mathbf{G})
 =
\left( \Gamma^+ \right)^{-1}
\left(
\mathbf{G}
\mathcal{E}
\mathbf{\Omega}
\mathbf{L}^2
+
\mathbf{G}_\mathbf{Q}^+
\right)
$,
hence by \eqref{equation_GM+toM},  \eqref{equation_Q-1Q+},
and Definition~\ref{definition_GMpm},
$
\mathbf{M}
 =
\begin{pmatrix}
\mathbf{M}^- & \mathbf{M}^+
\end{pmatrix}
\in \mathrm{wp}(4,8,\mathbb{C})
$
is a representative
of
$
\Gamma^{-1}(\mathbf{G})
\in \mathrm{wp}(\mathbb{C})$,
where
\begin{align*}
\mathbf{M}^-
 &=
\left\{
\mathbf{\Omega}
\mathbf{L}^2
-
\left(
\mathbf{G}
\mathcal{E}
\mathbf{\Omega}
\mathbf{L}^2
+
\mathbf{G}_\mathbf{Q}^+
\right)
\right\}
\left(
\mathbf{\Omega}
\mathbf{L}^2
\right)^{-1}
\mathbf{W}(-l)^{-1} \\
 &=
\left(
\mathbf{\Omega}
\mathbf{L}^2
-
\mathbf{G}
\mathcal{E}
\mathbf{\Omega}
\mathbf{L}^2
-
\widetilde{\mathbf{Q}}^{-1}
\widetilde{\mathbf{Q}}^+
\mathbf{\Omega}
\mathbf{L}^2
\right)
\left(
\mathbf{\Omega}
\mathbf{L}^2
\right)^{-1}
\mathbf{W}(-l)^{-1} \\
 &=
\left\{
\diag(0,1,1,0)
-
\mathbf{G}
\mathcal{E}
\right\}
\mathbf{W}(-l)^{-1}, \\
\mathbf{M}^+
 &=
\left(
\mathbf{G}
\mathcal{E}
\mathbf{\Omega}
\mathbf{L}^2
+
\mathbf{G}_\mathbf{Q}^+
\right)
\left(
\mathbf{\Omega}
\mathbf{L}^2
\right)^{-1}
\mathbf{W}(l)^{-1} \\
 &=
\left(
\mathbf{G}
\mathcal{E}
\mathbf{\Omega}
\mathbf{L}^2
+
\widetilde{\mathbf{Q}}^{-1}
\widetilde{\mathbf{Q}}^+
\mathbf{\Omega}
\mathbf{L}^2
\right)
\left(
\mathbf{\Omega}
\mathbf{L}^2
\right)^{-1}
\mathbf{W}(l)^{-1} \\
 &=
\left\{
\diag(1,0,0,1)
+
\mathbf{G}
\mathcal{E}
\right\}
\mathbf{W}(l)^{-1}.
\end{align*}
Thus,
given $\mathbf{G} \in \gl(4,\mathbb{C})$,
we have
\begin{align}
\begin{split}
\Gamma^{-1}(\mathbf{G})
 &=
\left[
\left(
\left\{
\diag(0,1,1,0)
-
\mathbf{G}
\mathcal{E}
\right\}
\mathbf{W}(-l)^{-1}
\right|
\right. \\
 &\qquad\qquad\qquad\qquad
\left.
\left|
\left\{
\diag(1,0,0,1)
+
\mathbf{G}
\mathcal{E}
\right\}
\mathbf{W}(l)^{-1}
\right)
\right]
\end{split}
\label{equation_GMto[M]}
\end{align}
in $\mathrm{wp}(\mathbb{C})$.
When $\mathbf{G} \in \overline{\pi}(4)$,
a representative of $\Gamma^{-1}\left( \mathbf{G} \right)$
{\em in} $\mathrm{wp}(4,8,\mathbb{R})$ is
\[
\mathbf{U}
\cdot
\left(
\begin{array}{c|c}
\left\{
\diag(0,1,1,0)
-
\mathbf{G}
\mathcal{E}
\right\}
\mathbf{W}(-l)^{-1}
 &
\left\{
\diag(1,0,0,1)
+
\mathbf{G}
\mathcal{E}
\right\}
\mathbf{W}(l)^{-1}
\end{array}
\right)
\]
by Lemma~\ref{lemma_GMpmtoMReal},
where $\mathbf{U}$ is defined by \eqref{equation_U}.

The boundary condition $\mathrm{BC}(\mathbf{M})$
in Lemma~\ref{lemma_eigencondition} 
is transformed into the following equivalent condition
which is expressed now in terms of $\mathbf{G}_\mathbf{M}$.

\begin{lemma}
\label{lemma_BCGM}
For $\mathbf{M} \in \mathrm{wp}(4,8,\mathbb{C})$,
the boundary condition 
$\mathrm{BC}(\mathbf{M})$:
$\mathbf{M} \cdot \mathcal{B}[u] = \mathbf{0}$
is equivalent to 
\begin{align*}
\mathbf{0}
 &=
\mathbf{G}_\mathbf{M}
\mathcal{E}
\left\{
\mathbf{W}(l)^{-1}
\mathcal{B}^+[u]
-
\mathbf{W}(-l)^{-1}
\mathcal{B}^-[u]
\right\}
\nonumber \\
 &\quad
+
\diag(0,1,1,0)
\cdot
\mathbf{W}(-l)^{-1}
\mathcal{B}^-[u]
+
\diag(1,0,0,1)
\cdot
\mathbf{W}(l)^{-1}
\mathcal{B}^+[u].
\end{align*}
\end{lemma}

\begin{proof}
By \eqref{equation_GMto[M]},
$
\left[ \hat{\mathbf{M}} \right]
 =
\Gamma^{-1}\left( \mathbf{G}_\mathbf{M} \right)
\in \mathrm{wp}(\mathbb{C})
$,
where
\begin{align}
\begin{split}
\hat{\mathbf{M}}
 &=
\left(
\left\{
\diag(0,1,1,0)
-
\mathbf{G}_\mathbf{M}
\mathcal{E}
\right\}
\mathbf{W}(-l)^{-1}
\right| \\
 &\qquad\qquad\qquad\qquad
\left|
\left\{
\diag(1,0,0,1)
+
\mathbf{G}_\mathbf{M}
\mathcal{E}
\right\}
\mathbf{W}(l)^{-1}
\right).
\end{split}
\label{equation_n;afhvldfjv}
\end{align}
So by Definition~\ref{definition_GM},
$
\left[ \mathbf{M} \right]
 =
\Gamma^{-1}\left( \mathbf{G}_\mathbf{M} \right)
 =
\left[ \hat{\mathbf{M}} \right]
$,
hence by Lemma~\ref{lemma_equivalence},
there exists $\mathbf{P} \in GL(4,\mathbb{C})$
such that
$
\mathbf{M}
 =
\mathbf{P}
\hat{\mathbf{M}}
$.
Thus the condition $\mathrm{BC}(\mathbf{M})$:  
$
\mathbf{M}
\cdot
\mathcal{B}[u]
 =
\mathbf{0}
$
is equivalent to
$
\hat{\mathbf{M}}
\cdot
\mathcal{B}[u]
 =
\mathbf{0}
$,
since $\mathbf{P}$ is invertible.
By \eqref{equation_calBpm}, \eqref{equation_n;afhvldfjv},
\begin{align*}
\lefteqn{
\hat{\mathbf{M}}
\cdot
\mathcal{B}[u]
 =
\hat{\mathbf{M}}
\cdot
\begin{pmatrix}
\mathcal{B}^-[u] \\
\mathcal{B}^+[u]
\end{pmatrix}
} \\
 &=
\left\{
\diag(0,1,1,0)
-
\mathbf{G}_\mathbf{M}
\mathcal{E}
\right\}
\mathbf{W}(-l)^{-1}
\cdot
\mathcal{B}^-[u] \\
 &\quad
+
\left\{
\diag(1,0,0,1)
+
\mathbf{G}_\mathbf{M}
\mathcal{E}
\right\}
\mathbf{W}(l)^{-1}
\cdot
\mathcal{B}^+[u] \\
 &=
\mathbf{G}_\mathbf{M}
\mathcal{E}
\left\{
\mathbf{W}(l)^{-1}
\mathcal{B}^+[u]
-
\mathbf{W}(-l)^{-1}
\mathcal{B}^-[u]
\right\} \\
 &\quad
+
\diag(0,1,1,0)
\cdot
\mathbf{W}(-l)^{-1}
\mathcal{B}^-[u]
+
\diag(1,0,0,1)
\cdot
\mathbf{W}(l)^{-1}
\mathcal{B}^+[u],
\end{align*}
hence the lemma follows.
\end{proof}

\subsection{Proof of Theorem~\ref{theorem_eigencondition-X}}
\label{section_X}

Note that the solution space of the linear homogeneous equation 
$\mathrm{EDE}(\lambda)$ in Definition~\ref{definition_EDE}
depends on the value $\lambda \neq 0$.
In particular,
depending on whether $1 - 1/(\lambda k) = 0$ or not,
or equivalently, whether $\lambda = 1/k$ or not,
$\mathrm{EDE}(\lambda)$ becomes as follows.

\begin{itemize}
\item[(I)]
When $\lambda = 1/k$:
$\mathrm{EDE}(\lambda)$ becomes $u^{(4)} = 0$.

\item[(II)]
When $\lambda \neq 1/k$:
$\mathrm{EDE}(\lambda)$ becomes 
$
u^{(4)} + \left( \kappa \alpha \right)^4 u
 =
0
$,
where $\kappa = \chi(\lambda) \neq 0$
is defined in Definition~\ref{definition_kappa} below.
\end{itemize}

\begin{definition}
\label{definition_kappa}
For $\lambda \in \mathbb{C} \setminus \left\{ 0, 1/k \right\}$,
define
$\chi(\lambda)$ to be the unique complex number satisfying
$
\chi(\lambda)^4
 = 
1 - 1/(\lambda k)
$
and
$
0 \leq \Arg{\chi(\lambda)} < \pi/2
$.
\end{definition}

Note that $\chi(\lambda) \neq 0$ and $\chi(\lambda)^4 \neq 1$
for $\lambda \in \mathbb{C} \setminus \left\{ 0, 1/k \right\}$.
In fact,
$\chi$ is a one-to-one correspondence
from $\mathbb{C} \setminus \left\{ 0, 1/k \right\}$
to 
$
\left\{
\kappa \in \mathbb{C}
\,|\,
0
 \leq
\Arg{\kappa}
 <
\frac{\pi}{2}
\right\}
\setminus
\{ 0, 1 \}
$,
and its inverse is given by
$
\chi^{-1}(\kappa)
 =
1/
\left\{
k
\left( 1 - \kappa^4 \right)
\right\}
$.

\begin{definition}
\label{definition_Wlambda(x)}
For $0 \neq \lambda \in \mathbb{C}$,
denote
\[
y_{\lambda,j}(x)
 =
\left\{
\begin{array}{ll}
\frac{x^{j-1}}{(j-1)!},
 &
\text{if }
\lambda = \frac{1}{k},
 \\
e^{\omega_j \kappa \alpha x},
 &
\text{if }
\lambda \neq \frac{1}{k},
\end{array}
\right.
\qquad
j = 1,2,3,4,
\]
where $\kappa = \chi(\lambda)$.
Denote
$\mathbf{y}_\lambda(x)
 =
\begin{pmatrix}
y_{\lambda,1}(x) &
y_{\lambda,2}(x) &
y_{\lambda,3}(x) &
y_{\lambda,4}(x)
\end{pmatrix}^T
$
and
\[
\mathbf{W}_\lambda(x)
 =
\left( y_{\lambda,j}^{(i-1)}(x) \right)_{1 \leq i,j \leq 4}
 =
\begin{pmatrix}
\mathbf{y}_\lambda(x)^T \\
\mathbf{y}_\lambda^\prime(x)^T \\
\mathbf{y}_\lambda^{\prime\prime}(x)^T \\
\mathbf{y}_\lambda^{\prime\prime\prime}(x)^T
\end{pmatrix}.
\]
\end{definition}

Note that the functions
$y_{\lambda,1}, y_{\lambda,2}, y_{\lambda,3}, y_{\lambda,4}$
form a fundamental set of solutions of $\mathrm{EDE}(\lambda)$
for every $0 \neq \lambda \in \mathbf{C}$,
and

\begin{equation}
\label{equation_Wlambdax}
\mathbf{W}_\lambda(x)
 =
\left\{
\begin{array}{ll}
\left(
H(j-i)
\frac{x^{j-i}}{(j - i)!}
\right)_{1 \leq i,j \leq 4}
 =
\begin{pmatrix}
1 & x & \frac{x^2}{2} & \frac{x^3}{6} \\
0 & 1 & x & \frac{x^2}{2} \\
0 & 0 & 1 & x \\
0 & 0 & 0 & 1
\end{pmatrix},
 &
\lambda = \frac{1}{k},
 \\
\left(
\left( \omega_j \kappa \alpha \right)^{i-1}
e^{\omega_j \kappa \alpha x}
\right)_{1 \leq i,j \leq 4},
 &
\lambda \neq \frac{1}{k},
\end{array}
\right.
\end{equation}
where
$\kappa = \chi(\lambda)$
and
$
H(t)
 =
\left\{
\begin{array}{ll}
1, & \text{if } t \geq 0, \\
0, & \text{if } t < 0.
\end{array}
\right.
$

Let $\mathbf{M} \in \mathrm{wp}(4,8,\mathbb{C})$.
By Lemma~\ref{lemma_eigencondition},
$\lambda \in \mathbb{C}$ is an eigenvalue of $\mathcal{K}_\mathbf{M}$
and $u \neq 0$ is a corresponding eigenfunction,
if and only if
$\lambda \neq 0$ and
$u$ is a nontrivial solution of $\mathrm{EDE}(\lambda)$
satisfying the boundary condition $\mathrm{BC}(\mathbf{M})$.
Note that
$u$ is a nontrivial solution of $\mathrm{EDE}(\lambda)$,
if and only if
there exists 
$
\mathbf{0}
 \neq
\mathbf{c}
 = 
\begin{pmatrix}
c_1 & c_2 & c_3 & c_4
\end{pmatrix}^T
 \in
\gl(4,1,\mathbb{C})
$ 
such that
$
u
 = 
\sum_{j=1}^4
c_j  y_{\lambda,j}
 =
\mathbf{y}_{\lambda}^T
\mathbf{c}
$.
If 
$
u
 = 
\sum_{j=1}^4
c_j  y_{\lambda,j}
$,
then by Definitions~\ref{definition_calBpm} and \ref{definition_Wlambda(x)},
\[
\mathcal{B}^\pm[u]
 =
\sum_{j=1}^4
c_j
\cdot
\mathcal{B}^\pm\left[ y_{\lambda,j} \right]
 =
\sum_{j=1}^4
\begin{pmatrix}
y_{\lambda,j}(\pm l) \\
y_{\lambda,j}^\prime(\pm l) \\
y_{\lambda,j}^{\prime\prime}(\pm l) \\
y_{\lambda,j}^{\prime\prime\prime}(\pm l)
\end{pmatrix}
\cdot
c_j
 =
\begin{pmatrix}
\mathbf{y}_\lambda(\pm l)^T \\
\mathbf{y}_\lambda^\prime((\pm l)^T \\
\mathbf{y}_\lambda^{\prime\prime}((\pm l)^T \\
\mathbf{y}_\lambda^{\prime\prime\prime}((\pm l)^T
\end{pmatrix}
\mathbf{c}
 =
\mathbf{W}_\lambda(\pm l)
\mathbf{c}.
\]
It follows from Lemma~\ref{lemma_BCGM}
that
the condition $\mathrm{BC}(\mathbf{M})$
is equivalent to
\begin{align}
\mathbf{0}
 &=
\mathbf{G}_\mathbf{M}
\mathcal{E}
\left\{
\mathbf{W}(l)^{-1}
\mathbf{W}_\lambda(l)
\mathbf{c}
-
\mathbf{W}(-l)^{-1}
\mathbf{W}_\lambda(-l)
\mathbf{c}
\right\}
\nonumber \\
 &\quad
+
\diag(0,1,1,0)
\cdot
\mathbf{W}(-l)^{-1}
\mathbf{W}_\lambda(-l)
\mathbf{c}
+
\diag(1,0,0,1)
\cdot
\mathbf{W}(l)^{-1}
\mathbf{W}_\lambda(l)
\mathbf{c}
\nonumber \\
\begin{split}
 &=
\left[
\mathbf{G}_\mathbf{M}
\mathcal{E}
\left\{
\mathbf{W}(l)^{-1}
\mathbf{W}_\lambda(l)
-
\mathbf{W}(-l)^{-1}
\mathbf{W}_\lambda(-l)
\right\}
\right. \\
 &\quad
\left.
+
\diag(0,1,1,0)
\cdot
\mathbf{W}(-l)^{-1}
\mathbf{W}_\lambda(-l)
+
\diag(1,0,0,1)
\cdot
\mathbf{W}(l)^{-1}
\mathbf{W}_\lambda(l)
\right]
\mathbf{c}.
\end{split}
\label{equation_afsdhbv;afovjh}
\end{align}
Thus Lemma~\ref{lemma_eigencondition} can be rephrased as the following.

\begin{lemma}
\label{lemma_eigencondition-rephrased}
Let
$
\mathbf{M}
\in \mathrm{wp}(4,8,\mathbb{C})$,
$0 \neq u \in L^2[-l,l]$,
and
$\lambda \in \mathbb{C}$.
Then $\mathcal{K}_\mathbf{M}[u] = \lambda \cdot u$,
if and only if
$\lambda \neq 0$
and
$
u
 = 
\mathbf{y}_{\lambda}^T
\mathbf{c}
$
for some
$
\mathbf{0}
 \neq
\mathbf{c}
 \in
\gl(4,1,\mathbb{C})$ 
which satisfies
\eqref{equation_afsdhbv;afovjh}.
\end{lemma}

The following matrix 
$\mathbf{X}_\lambda(x)$ will have a key role in our discussions.

\begin{definition}
\label{definition_Xlambda(x)}
For $0 \neq  \lambda \in \mathbb{C}$
and $x \in \mathbb{R}$,
we denote
\[
\mathbf{X}_\lambda(x)
 =
\diag(0,1,1,0)
\cdot
\mathbf{W}(-x)^{-1}
\mathbf{W}_\lambda(-x)
+
\diag(1,0,0,1)
\cdot
\mathbf{W}(x)^{-1}
\mathbf{W}_\lambda(x).
\]
\end{definition}

Note from Definitions~\ref{definition_W(x)} and \ref{definition_Wlambda(x)} 
that,
for each $0 \neq \lambda \in \mathbb{C}$ and $x \in \mathbb{R}$,
$\mathbf{X}_\lambda(x)$ is a concrete $4 \times 4$ matrix
which does {\em not} depend on $\mathbf{M}$. 
By Definition~\ref{definition_Xlambda(x)},
\begin{align}
\lefteqn{
\mathbf{X}_\lambda(x)
-
\mathbf{X}_\lambda(-x)
}
\nonumber \\
 &=
\left\{
\diag(0,1,1,0)
\cdot
\mathbf{W}(-x)^{-1}
\mathbf{W}_\lambda(-x)
+
\diag(1,0,0,1)
\cdot
\mathbf{W}(x)^{-1}
\mathbf{W}_\lambda(x)
\right\}
\nonumber \\
 &\quad
-
\left\{
\diag(0,1,1,0)
\cdot
\mathbf{W}(x)^{-1}
\mathbf{W}_\lambda(x)
+
\diag(1,0,0,1)
\cdot
\mathbf{W}(-x)^{-1}
\mathbf{W}_\lambda(-x)
\right\}
\nonumber \\
 &=
\mathcal{E}
\left\{
\mathbf{W}(x)^{-1}
\mathbf{W}_\lambda(x)
-
\mathbf{W}(-x)^{-1}
\mathbf{W}_\lambda(-x)
\right\},
\qquad
0 \neq \lambda \in \mathbb{C},
\
x \in \mathbb{R}.
\label{equation_s;fhgs;fovij}
\end{align}

Now we are ready to prove Theorem~\ref{theorem_eigencondition-X}.

\begin{proof}[Proof of Theorem~\ref{theorem_eigencondition-X}]
By Definition~\ref{definition_Xlambda(x)} and \eqref{equation_s;fhgs;fovij},
the condition \eqref{equation_afsdhbv;afovjh} is equivalent to
$
\left[
\mathbf{G}_\mathbf{M}
\left\{
\mathbf{X}_\lambda(l)
-
\mathbf{X}_\lambda(-l)
\right\}
+
\mathbf{X}_\lambda(l)
\right]
\mathbf{c}
 =
\mathbf{0}
$.
Thus the first assertion follows from Lemma~\ref{lemma_eigencondition-rephrased}. 
The second assertion follows from the first one,
since $\mathbf{G}_\mathbf{Q} = \mathbf{O}$ by \eqref{equation_GQ=O}.
\end{proof}

\section{Symmetries of $\protect\mathbf{X}_\lambda$ and $\protect\mathbf{Y}_\lambda$}
\label{section_symmetry}

As a consequence of Theorem~\ref{theorem_eigencondition-X}
and Corollary~\ref{corollary_eigencondition-det},
the matrix $\mathbf{X}_\lambda(l)$ is invertible
for every $0 \neq \lambda \in \mathbb{C}$ which is not in $\Spec\mathcal{K}_\mathbf{Q}$.
In fact, this is true for arbitrary $l > 0$
by Proposition~\ref{proposition_Q}.

\begin{definition}
\label{definition_Ylambda(x)}
Denote
$
\mathbf{Y}_\lambda(x)
 =
\mathbf{X}_\lambda(-x)
\mathbf{X}_\lambda(x)^{-1}
-
\mathbf{I}
$
for $0 \neq \lambda \in \mathbb{C}$
and $x > 0$
such that $\det\mathbf{X}_\lambda(x) \neq 0$.
\end{definition}

In view of 
Theorem~\ref{theorem_eigencondition-X}
and
Corollaries~\ref{corollary_eigencondition-det}, \ref{corollary_eigencondition-Y},
it is now apparent that
analysis on
the $4 \times 4$ matrices $\mathbf{X}_\lambda(x)$ and $\mathbf{Y}_\lambda(x)$
are important.
It turns out that they have various symmetries,
and some of them are explored in this section.
In particular, we will obtain the following result,
which is crucial in proving Theorem~\ref{theorem_existence}.

\begin{lemma}
\label{lemma_Ylambdainpibar}
$\mathbf{Y}_\lambda(x) \in \overline{\pi}(4)$
for every $0 \neq \lambda \in \mathbb{R}$ and $x > 0$
such that $\det\mathbf{X}_\lambda(x) \neq 0$.
\end{lemma}

The proof of Lemma~\ref{lemma_Ylambdainpibar} will be given at the end of
Section~\ref{section_1/k}.
To facilitate our analysis, we introduce 
the following change of variables
\begin{equation}
\label{equation_z}
z = \alpha x,
\end{equation}
which will be used extensively for the rest of the paper.
By \eqref{equation_Wxdecomposed}, \eqref{equation_Wx-1decomposed},
and \eqref{equation_z},
\begin{align}
\mathbf{W}(x)
 &=
\diag
\left( 1, \alpha, \alpha^2, \alpha^3 \right)
\cdot
\mathbf{W}_0
e^{\mathbf{\Omega} z}, 
\label{equation_Wxdecomposedz} \\
\mathbf{W}(x)^{-1}
 &=
\frac{1}{4}
e^{-\mathbf{\Omega} z}
\mathbf{W}_0^*
\cdot
\diag
\left( 1, \alpha, \alpha^2, \alpha^3 \right)^{-1}.
\label{equation_Wx-1decomposedz}
\end{align}

The following form of $\mathbf{X}_\lambda(x)$
will be useful.

\begin{lemma}
\label{lemma_Xlambda(x)2}
For $0 \neq \lambda \in \mathbb{C}$ and $x \in \mathbb{R}$,
\begin{align*}
\mathbf{X}_\lambda(x)
 &=
\frac{1}{4}
e^{-\mathcal{E} \mathbf{\Omega} z}
\left\{
\diag(0,1,1,0)
\cdot
\mathbf{W}_0^*
\cdot
\diag
\left( 1, \alpha, \alpha^2, \alpha^3 \right)^{-1}
\mathbf{W}_\lambda(-x)
\right. \\
 &\qquad\qquad\quad
\left.
+
\diag(1,0,0,1)
\cdot
\mathbf{W}_0^*
\cdot
\diag
\left( 1, \alpha, \alpha^2, \alpha^3 \right)^{-1}
\mathbf{W}_\lambda(x)
\right\}.
\end{align*}
\end{lemma}

\begin{proof}
By Definition~\ref{definition_Xlambda(x)}
and 
\eqref{equation_z}, \eqref{equation_Wx-1decomposedz},
we have
\begin{equation}
\label{equation_bhsofibjh}
\begin{split}
\mathbf{X}_\lambda(x)
 &=
\frac{1}{4}
\left\{
\diag(0,1,1,0)
\cdot
e^{\mathbf{\Omega} z}
\mathbf{W}_0^*
\cdot
\diag
\left( 1, \alpha, \alpha^2, \alpha^3 \right)^{-1}
\mathbf{W}_\lambda(-x)
\right. \\
 &\qquad
\left.
+
\diag(1,0,0,1)
\cdot
e^{-\mathbf{\Omega} z}
\mathbf{W}_0^*
\cdot
\diag
\left( 1, \alpha, \alpha^2, \alpha^3 \right)^{-1}
\mathbf{W}_\lambda(x)
\right\}.
\end{split}
\end{equation}
Note that
\begin{align}
\diag(0,1,1,0)
\cdot
e^{\mathbf{\Omega} z}
 &=
\diag
\left( 
0, 
e^{\omega_2 z}, 
e^{\omega_3 z}, 
0
\right)
\nonumber \\
 &=
\diag
\left( 
e^{-\omega_1 z}, 
e^{\omega_2 z}, 
e^{\omega_3 z}, 
e^{-\omega_4 z}
\right)
\cdot
\diag(0,1,1,0)
\nonumber \\
 &=
e^{-\mathcal{E} \mathbf{\Omega} z}
\cdot
\diag(0,1,1,0),
\label{equation_aflvha;fvn0110} \\
\diag(1,0,0,1)
\cdot
e^{-\mathbf{\Omega} z}
 &=
\diag
\left( 
e^{-\omega_1 z}, 
0, 
0, 
e^{-\omega_4 z}
\right)
\nonumber \\
 &=
\diag
\left( 
e^{-\omega_1 z}, 
e^{\omega_2 z}, 
e^{\omega_3 z}, 
e^{-\omega_4 z}
\right)
\cdot
\diag(1,0,0,1)
\nonumber \\
 &=
e^{-\mathcal{E} \mathbf{\Omega} z}
\cdot
\diag(1,0,0,1).
\label{equation_aflvha;fvn1001}
\end{align}
Now the lemma follows from
\eqref{equation_bhsofibjh},
\eqref{equation_aflvha;fvn0110},
\eqref{equation_aflvha;fvn1001}.
\end{proof}

\subsection{The case $\protect\lambda \neq 1/k$}
\label{section_neq1/k}

In this section,
we assume 
$\lambda$ is a complex number such that
$\lambda \neq 0$, $\lambda \neq 1/k$.
Let $\kappa = \chi(\lambda)$
as in Definition~\ref{definition_kappa}.
By Definition~\ref{definition_omega}, 
and \eqref{equation_Wlambdax}, \eqref{equation_z},
we have
\begin{align}
\mathbf{W}_\lambda(x)
 &=
\left(
\left( \omega_j \kappa \alpha \right)^{i-1}
e^{\omega_j \kappa \alpha x}
\right)_{1 \leq i,j \leq 4}
 =
\left(
\alpha^{i-1} \kappa^{i-1} \omega_j^{i-1}
e^{\omega_j \kappa z}
\right)_{1 \leq i,j \leq 4}
\nonumber \\
 &=
\diag
\left( 1, \alpha, \alpha^2, \alpha^3 \right)
\cdot
\diag
\left( 1, \kappa, \kappa^2, \kappa^3 \right)
\cdot
\mathbf{W}_0
e^{\mathbf{\Omega} \kappa z},
\label{equation_Wlambdaxdecomposedz}
\end{align}
hence
$
\diag
\left( 1, \alpha, \alpha^2, \alpha^3 \right)^{-1}
\mathbf{W}_\lambda(x)
 =
\diag
\left( 1, \kappa, \kappa^2, \kappa^3 \right)
\cdot
\mathbf{W}_0
e^{\mathbf{\Omega} \kappa z}
$.
Thus by \eqref{equation_z} and
Lemma~\ref{lemma_Xlambda(x)2}, 
we have
\begin{equation}
\label{equation_Xz,kappa}
\begin{split}
\mathbf{X}_\lambda(x)
 &=
\frac{1}{4}
e^{-\mathcal{E} \mathbf{\Omega} z}
\left\{
\diag(0,1,1,0)
\cdot
\mathbf{W}_0^*
\cdot
\diag
\left( 1, \kappa, \kappa^2, \kappa^3 \right)
\cdot
\mathbf{W}_0
e^{-\mathbf{\Omega} \kappa z}
\right. \\
 &\qquad\qquad
\left.
+
\diag(1,0,0,1)
\cdot
\mathbf{W}_0^*
\cdot
\diag
\left( 1, \kappa, \kappa^2, \kappa^3 \right)
\cdot
\mathbf{W}_0
e^{\mathbf{\Omega} \kappa z}
\right\},
\end{split}
\end{equation}
where
$
\kappa = \chi(\lambda)
$.
Note that \eqref{equation_Xz,kappa} is well-defined for every $z, \kappa \in \mathbb{C}$,
though we originally restricted the domains of $z, \kappa$.

\begin{definition}
\label{definition_X(z,kappa)}
For $z, \kappa \in \mathbb{C}$,
define
\begin{align*}
\mathbf{X}(z,\kappa)
 &=
\frac{1}{4}
e^{-\mathcal{E} \mathbf{\Omega} z}
\left\{
\diag(0,1,1,0)
\cdot
\mathbf{W}_0^*
\cdot
\diag
\left( 1, \kappa, \kappa^2, \kappa^3 \right)
\cdot
\mathbf{W}_0
e^{-\mathbf{\Omega} \kappa z}
\right.
\nonumber \\
 &\qquad\qquad\quad
\left.
+
\diag(1,0,0,1)
\cdot
\mathbf{W}_0^*
\cdot
\diag
\left( 1, \kappa, \kappa^2, \kappa^3 \right)
\cdot
\mathbf{W}_0
e^{\mathbf{\Omega} \kappa z}
\right\}.
\end{align*}

By \eqref{equation_z}, \eqref{equation_Xz,kappa},
we have
\begin{equation}
\label{equation_zkappaxlambda}
\mathbf{X}_\lambda(x)
 =
\mathbf{X}\left( \alpha x, \chi(\lambda) \right),
\qquad
x \in \mathbb{R},
\
\lambda \in \mathbb{C} \setminus \left\{ 0, 1/k \right\}.
\end{equation}
\end{definition}

\begin{lemma}
\label{lemma_X(z,kappa)}
\begin{itemize}
\item[{\rm (a)}]
$
\mathbf{X}(z,\mathbbm{i} \kappa)
 =
\mathbf{X}(z,\kappa)
\cdot
\mathbf{L}^{-1}
$
for every
$z, \kappa \in \mathbb{C}$.

\item[{\rm (b)}]
$
\mathbf{R}
\overline{
\mathbf{X}(z,\kappa)
}
\mathbf{R}
 =
\mathbf{X}\left( z, \overline{\kappa} \right)
$
for every $z \in \mathbb{R}$
and
$\kappa \in \mathbb{C}$.
\end{itemize}
\end{lemma}

\begin{proof}
By \eqref{equation_LOmegaL-1}, \eqref{equation_W0L-1},
we have
\begin{align*}
\lefteqn{
\diag
\left( 
1, 
\mathbbm{i} \kappa, 
\left( \mathbbm{i} \kappa \right)^2,
\left( \mathbbm{i} \kappa \right)^3
\right)
\cdot
\mathbf{W}_0
e^{\pm \mathbf{\Omega} \left( \mathbbm{i} \kappa \right) z}
} \\
 &=
\diag
\left( 1, \kappa, \kappa^2, \kappa^3 \right)
\cdot
\diag
\left( 1, \mathbbm{i}, \mathbbm{i}^2, \mathbbm{i}^3 \right)
\cdot
\mathbf{W}_0
e^{\pm \left( \mathbbm{i} \mathbf{\Omega} \right) \kappa z} \\
 &=
\diag
\left( 1, \kappa, \kappa^2, \kappa^3 \right)
\cdot
\mathbf{W}_0
\mathbf{L}^{-1}
\cdot
e^{
\pm 
\mathbf{L}
\mathbf{\Omega}
\mathbf{L}^{-1}
\kappa z
} \\
 &=
\diag
\left( 1, \kappa, \kappa^2, \kappa^3 \right)
\cdot
\mathbf{W}_0
\mathbf{L}^{-1}
\cdot
\mathbf{L}
e^{
\pm 
\mathbf{\Omega}
\kappa z
}
\mathbf{L}^{-1} \\
 &=
\diag
\left( 1, \kappa, \kappa^2, \kappa^3 \right)
\cdot
\mathbf{W}_0
e^{
\pm 
\mathbf{\Omega}
\kappa z
}
\cdot
\mathbf{L}^{-1},
\end{align*}
which implies (a) by Definition~\ref{definition_X(z,kappa)}.
Suppose $z \in \mathbb{R}$,
and let
\begin{equation}
\label{equation_afovjhadifv},
\begin{split}
\hat{\mathbf{X}}(z,\kappa)
 &=
\diag(0,1,1,0)
\cdot
\mathbf{W}_0^*
\cdot
\diag
\left( 1, \kappa, \kappa^2, \kappa^3 \right)
\cdot
\mathbf{W}_0
e^{-\mathbf{\Omega} \kappa z} \\
 &\quad
+
\diag(1,0,0,1)
\cdot
\mathbf{W}_0^*
\cdot
\diag
\left( 1, \kappa, \kappa^2, \kappa^3 \right)
\cdot
\mathbf{W}_0
e^{\mathbf{\Omega} \kappa z}
\end{split}
\end{equation}
so that by Definition~\ref{definition_X(z,kappa)},
\begin{equation}
\label{equation_va;ofhva}
\mathbf{X}(z,\kappa)
 =
\frac{1}{4}
e^{-\mathcal{E} \mathbf{\Omega} z}
\hat{\mathbf{X}}(z,\kappa).
\end{equation}
Note that
$
\diag(0,1,1,0),
\diag(0,1,1,0)
\in \overline{\pi}(4)
$.
Since $z \in \mathbb{R}$,
$
e^{-\mathcal{E} \mathbf{\Omega} z}
\in \overline{\pi}(4)
$,
and
$
\mathbf{R}
\overline{
e^{
\pm
\mathbf{\Omega} \kappa z}
}
\mathbf{R}
 =
\mathbf{R}
e^{
\pm
\overline{
\mathbf{\Omega}
}
\overline{\kappa}
z}
\mathbf{R}
 =
\mathbf{R}
e^{
\pm
\mathbf{R}
\mathbf{\Omega}
\mathbf{R}
\overline{\kappa}
z}
\mathbf{R}
 =
\mathbf{R}
\cdot
\mathbf{R}
e^{
\pm
\mathbf{\Omega}
\overline{\kappa}
z}
\mathbf{R}
\cdot
\mathbf{R}
 =
e^{\pm
\mathbf{\Omega}
\overline{\kappa}
z}
$
by \eqref{equation_R}.
Hence by \eqref{equation_afovjhadifv},
we have
\begin{align*}
\mathbf{R}
\overline{
\hat{\mathbf{X}}(z,\kappa)
}
\mathbf{R}
 &=
\mathbf{R}
\overline{
\left\{
\diag(0,1,1,0)
\cdot
\mathbf{W}_0^*
\cdot
\diag
\left( 1, \kappa, \kappa^2, \kappa^3 \right)
\cdot
\mathbf{W}_0
\cdot
e^{-\mathbf{\Omega} \kappa z}
\right\}
}
\mathbf{R} \\
 &\quad
+
\mathbf{R}
\overline{
\left\{
\diag(1,0,0,1)
\cdot
\mathbf{W}_0^*
\cdot
\diag
\left( 1, \kappa, \kappa^2, \kappa^3 \right)
\cdot
\mathbf{W}_0
\cdot
e^{\mathbf{\Omega} \kappa z}
\right\}
}
\mathbf{R} \\
 &=
\mathbf{R}
\overline{
\diag(0,1,1,0)
}
\mathbf{R}
\cdot
\mathbf{R}
\mathbf{W}_0^T
\cdot
\overline{
\diag
\left( 1, \kappa, \kappa^2, \kappa^3 \right)
}
\cdot
\overline{
\mathbf{W}_0
}
\mathbf{R}
\cdot
\mathbf{R}
\overline{
e^{-\mathbf{\Omega} \kappa z}
}
\mathbf{R} \\
 &\quad
+
\mathbf{R}
\overline{
\diag(1,0,0,1)
}
\mathbf{R}
\cdot
\mathbf{R}
\mathbf{W}_0^T
\cdot
\overline{
\diag
\left( 1, \kappa, \kappa^2, \kappa^3 \right)
}
\cdot
\overline{
\mathbf{W}_0
}
\mathbf{R}
\cdot
\mathbf{R}
\overline{
e^{\mathbf{\Omega} \kappa z}
}
\mathbf{R} \\
 &=
\diag(0,1,1,0)
\cdot
\left(
\mathbf{W}_0
\mathbf{R}
\right)^T
\cdot
\diag
\left( 1, \overline{\kappa}, \overline{\kappa}^2, \overline{\kappa}^3 \right)
\cdot
\overline{
\mathbf{W}_0
\mathbf{R}
}
\cdot
e^{-\mathbf{\Omega} \overline{\kappa} z} \\
 &\quad
+
\diag(1,0,0,1)
\cdot
\left(
\mathbf{W}_0
\mathbf{R}
\right)^T
\cdot
\diag
\left( 1, \overline{\kappa}, \overline{\kappa}^2, \overline{\kappa}^3 \right)
\cdot
\overline{
\mathbf{W}_0
\mathbf{R}
}
\cdot
e^{\mathbf{\Omega} \overline{\kappa} z} \\
 &=
\hat{\mathbf{X}}\left( z, \overline{\kappa} \right),
\end{align*}
since
$
\left(
\mathbf{W}_0
\mathbf{R}
\right)^T
 =
\overline{
\mathbf{W}_0
}^T
 =
\mathbf{W}_0^*
$
and
$
\overline{
\mathbf{W}_0
\mathbf{R}
}
 =
\overline{
\overline{
\mathbf{W}_0
}
}
 =
\mathbf{W}_0
$
by \eqref{equation_R}.
Thus by \eqref{equation_va;ofhva},
\begin{align*}
\mathbf{R}
\overline{
\mathbf{X}(z,\kappa)
}
\mathbf{R}
 &=
\mathbf{R}
\overline{
\left\{
\frac{1}{4}
e^{-\mathcal{E} \mathbf{\Omega} z}
\hat{\mathbf{X}}(z,\kappa)
\right\}
}
\mathbf{R}
 =
\frac{1}{4}
\mathbf{R}
\overline{
e^{-\mathcal{E} \mathbf{\Omega} z}
}
\mathbf{R}
\cdot
\mathbf{R}
\overline{
\hat{\mathbf{X}}(z,\kappa)
}
\mathbf{R} \\
 &=
\frac{1}{4}
e^{-\mathcal{E} \mathbf{\Omega} z}
\hat{\mathbf{X}}\left( z, \overline{\kappa} \right)
 =
\mathbf{X}\left( z, \overline{\kappa} \right),
\end{align*}
since
$
e^{-\mathcal{E} \mathbf{\Omega} z}
\in \overline{\pi}(4)$.
This shows (b), and the proof is complete.
\end{proof}

Interpreting Lemma~\ref{lemma_X(z,kappa)}
in terms of $\mathbf{X}_\lambda(x)$ and $\mathbf{Y}_\lambda(x)$,
we have the following.

\begin{lemma}
\label{lemma_XYlambdaneq1/k}
\begin{itemize}
\item[{\rm (a)}]
For
$
\lambda 
\in \mathbb{C} \setminus \left\{0, 1/k \right\}$ and $x \in \mathbb{R}$,
\[
\mathbf{R}
\overline{
\mathbf{X}_\lambda(x)
}
\mathbf{R}
 =
\left\{
\begin{array}{ll}
\mathbf{X}_\lambda(x), &
\text{if }
\lambda \in (-\infty,0) \cup \left( \frac{1}{k}, \infty \right) \\
\mathbf{X}_{\overline{\lambda}}(x)
\cdot
\mathbf{L}, &
\text{otherwise}
\end{array}
\right..
\]
In particular,
$
\mathbf{X}_\lambda(x)
\in \overline{\pi}(4)
$
for every 
$\lambda \in (-\infty,0) \cup \left( 1/k, \infty \right)$
and
$x \in \mathbb{R}$.

\item[{\rm (b)}]
$
\mathbf{R}
\overline{
\mathbf{Y}_\lambda(x)
}
\mathbf{R}
 =
\mathbf{Y}_{\overline{\lambda}}(x)
$
for every 
$
\lambda 
\in 
\mathbb{C} 
\setminus 
\left\{ 0, 1/k \right\}
$ 
and $x > 0$
such that $\det\mathbf{X}_\lambda(x) \neq 0$.
In particular,
$
\mathbf{Y}_\lambda(x)
\in \overline{\pi}(4)
$
for every
$
\lambda 
\in 
\mathbb{R} 
\setminus 
\left\{ 0, 1/k \right\}
$ 
and $x > 0$
such that $\det\mathbf{X}_\lambda(x) \neq 0$.
\end{itemize}
\end{lemma}

\begin{proof}
Suppose
$
\lambda 
\in \mathbb{C} \setminus \left\{0, 1/k \right\}$ and $x \in \mathbb{R}$.
Let
$
r e^{\mathbbm{i} \theta}
 =
1 - 1/(\lambda k)
$,
$r > 0$, 
$0 \leq \theta < 2\pi$.
Then by Definition~\ref{definition_kappa},
$
\chi(\lambda)
 =
\sqrt[4]{r}
e^{\mathbbm{i} \frac{\theta}{4}}
$.
Suppose $\theta = 0$,
which is equivalent to
$
\lambda \in (-\infty,0) \cup \left( 1/k, \infty \right)
$.
Then
$
\overline{\chi(\lambda)}
 =
\sqrt[4]{r}
 =
\chi(\lambda)
$,
hence by \eqref{equation_zkappaxlambda} and Lemma~\ref{lemma_X(z,kappa)} (b),
$
\mathbf{R}
\overline{
\mathbf{X}_\lambda(x)
}
\mathbf{R}
 =
\mathbf{R}
\overline{
\mathbf{X}\left( \alpha x, \chi(\lambda) \right)
}
\mathbf{R}
 =
\mathbf{X}\left( \alpha x, \overline{\chi(\lambda)} \right)
 =
\mathbf{X}\left( \alpha x, \chi(\lambda) \right)
 =
\mathbf{X}_\lambda(x)
$.
Suppose $0 < \theta < 2\pi$.
Then by Definition~\ref{definition_kappa},
$
\chi\left( \overline{\lambda} \right)
 =
\chi\left( r e^{-\mathbbm{i} \theta} \right)
 =
\sqrt[4]{r} e^{\mathbbm{i} \left( \frac{\pi}{2} - \frac{\theta}{4} \right)}
 =
\mathbbm{i}
\cdot
\sqrt[4]{r} e^{-\mathbbm{i} \frac{\theta}{4}}
 =
\mathbbm{i}
\cdot
\overline{\chi(\lambda)}
$,
since
$
0
 <
\Arg
\left(
\sqrt[4]{r} e^{\mathbbm{i} \left( \frac{\pi}{2} - \frac{\theta}{4} \right)}
\right)
 =
\pi/2 - \theta/4
 <
\pi/2
$.
Thus by \eqref{equation_zkappaxlambda} and Lemma~\ref{lemma_X(z,kappa)} (a),
$
\mathbf{X}_{\overline{\lambda}}(x)
 =
\mathbf{X}\left( \alpha x, \chi\left( \overline{\lambda} \right) \right)
 =
\mathbf{X}\left( \alpha x, \mathbbm{i} \cdot \overline{\chi(\lambda)} \right)
 =
\mathbf{X}\left( \alpha x, \overline{\chi(\lambda)} \right)
\cdot
\mathbf{L}^{-1}
$,
and hence by \eqref{equation_zkappaxlambda} and Lemma~\ref{lemma_X(z,kappa)} (b),
$
\mathbf{R}
\overline{
\mathbf{X}_\lambda(x)
}
\mathbf{R}
 =
\mathbf{R}
\overline{
\mathbf{X}\left( \alpha x, \chi(\lambda) \right)
}
\mathbf{R}
 =
\mathbf{X}\left( \alpha x, \overline{\chi(\lambda)} \right)
 =
\mathbf{X}_{\overline{\lambda}}(x)
\cdot
\mathbf{L}
$.
This shows (a).

Suppose
$
\lambda 
\in 
\mathbb{C} 
\setminus 
\left\{ 0, 1/k \right\}
$,
$x > 0$,
and
$\det\mathbf{X}_\lambda(x) \neq 0$.
Suppose first that
$\lambda \in (-\infty,0) \cup \left( 1/k, \infty \right)$.
Then by (a),
$
\mathbf{X}_\lambda(x),
\mathbf{X}_\lambda(-x)
\in \overline{\pi}(4)
$.
Thus by Lemma~\ref{lemma_pibar} and Definition~\ref{definition_Ylambda(x)},
$
\mathbf{Y}_\lambda(x)
 =
\mathbf{X}_\lambda(-x)
\mathbf{X}_\lambda(x)^{-1}
-
\mathbf{I}
\in \overline{\pi}(4)
$,
and hence
$
\mathbf{R}
\overline{
\mathbf{Y}_\lambda(x)
}
\mathbf{R}
 =
\mathbf{Y}_\lambda(x)
 =
\mathbf{Y}_{\overline{\lambda}}(x)
$,
since $\lambda$ is real.
Suppose 
$\lambda \not\in (-\infty,0) \cup \left( 1/k, \infty \right)$.
Then 
by Definition~\ref{definition_Ylambda(x)}
and (a),
\begin{align*}
\mathbf{R}
\overline{
\mathbf{Y}_\lambda(x)
}
\mathbf{R}
 &=
\mathbf{R}
\overline{
\left\{
\mathbf{X}_\lambda(-x)
\mathbf{X}_\lambda(x)^{-1}
-
\mathbf{I}
\right\}
}
\mathbf{R}
 =
\mathbf{R}
\overline{
\mathbf{X}_\lambda(-x)
}
\mathbf{R}
\cdot
\mathbf{R}
\overline{
\mathbf{X}_\lambda(x)^{-1}
}
\mathbf{R}
-
\mathbf{I} \\
 &=
\left\{
\mathbf{X}_{\overline{\lambda}}(-x)
\mathbf{L}
\right\}
\left\{
\mathbf{X}_{\overline{\lambda}}(x)
\mathbf{L}
\right\}^{-1}
-
\mathbf{I}
 =
\mathbf{X}_{\overline{\lambda}}(-x)
\mathbf{X}_{\overline{\lambda}}(x)^{-1}
-
\mathbf{I}
 =
\mathbf{Y}_{\overline{\lambda}}(x).
\end{align*}
Thus we showed (b),
and the proof is complete.
\end{proof}

\subsection{The case $\protect\lambda = 1/k$}
\label{section_1/k}

By \eqref{equation_Wlambdax}, \eqref{equation_z},
we have
\begin{align*}
\lefteqn{
\left\{
\diag
\left( 1, \alpha, \alpha^2, \alpha^3 \right)^{-1}
\cdot
\mathbf{W}_\frac{1}{k}(x)
\cdot
\diag
\left( 1, \alpha, \alpha^2, \alpha^3 \right)
\right\}_{i,j}
}
\nonumber \\
 &=
\alpha^{1-i}
\cdot
\left\{
\mathbf{W}_\frac{1}{k}(x)
\right\}_{i,j}
\cdot
\alpha^{j-1}
 =
\alpha^{j-i}
\cdot
H(j-i)
\frac{x^{j - i}}{(j - i)!}
 =
H(j-i)
\frac{(\alpha x)^{j - i}}{(j - i)!}
\nonumber \\
 &=
H(j-i)
\frac{z^{j - i}}{(j - i)!}
 =
\left\{
\mathbf{W}_\frac{1}{k}(z)
\right\}_{i,j},
\qquad
1 \leq i,j \leq 4,
\end{align*}
hence
$
\diag
\left( 1, \alpha, \alpha^2, \alpha^3 \right)^{-1}
\cdot
\mathbf{W}_\frac{1}{k}(x)
 =
\mathbf{W}_\frac{1}{k}(z)
\cdot
\diag
\left( 1, \alpha, \alpha^2, \alpha^3 \right)^{-1}
$.
Thus by 
\newline
Lemma~\ref{lemma_Xlambda(x)2},
\begin{align*}
\lefteqn{
\mathbf{X}_\frac{1}{k}(x)
 =
\frac{1}{4}
e^{-\mathcal{E} \mathbf{\Omega} z}
\left\{
\diag(0,1,1,0)
\cdot
\mathbf{W}_0^*
\mathbf{W}_\frac{1}{k}(-z)
\cdot
\diag
\left( 1, \alpha, \alpha^2, \alpha^3 \right)^{-1}
\right.
} \\
 &\qquad\qquad\qquad\qquad
\left.
+
\diag(1,0,0,1)
\cdot
\mathbf{W}_0^*
\mathbf{W}_\frac{1}{k}(z)
\cdot
\diag
\left( 1, \alpha, \alpha^2, \alpha^3 \right)^{-1}
\right\} \\
 &=
\frac{1}{4}
e^{-\mathcal{E} \mathbf{\Omega} z}
\left\{
\diag(0,1,1,0)
\cdot
\mathbf{W}_0^*
\mathbf{W}_\frac{1}{k}(-z)
+
\diag(1,0,0,1)
\cdot
\mathbf{W}_0^*
\mathbf{W}_\frac{1}{k}(z)
\right\}
\cdot \\
 &\qquad
\cdot 
\diag
\left( 1, \alpha, \alpha^2, \alpha^3 \right)^{-1}.
\end{align*}
For $z \in \mathbb{R}$,
denote
\begin{equation}
\label{equation_Pz}
\mathbf{P}(z)
 =
\diag(0,1,1,0)
\cdot
\mathbf{W}_0^*
\mathbf{W}_\frac{1}{k}(-z)
+
\diag(1,0,0,1)
\cdot
\mathbf{W}_0^*
\mathbf{W}_\frac{1}{k}(z),
\end{equation}
so that
\begin{equation}
\label{equation_anvrfuvh;aof}
\mathbf{X}_\frac{1}{k}(x)
 =
\frac{1}{4}
e^{-\mathcal{E} \mathbf{\Omega} z}
\mathbf{P}(z)
\cdot
\diag
\left( 1, \alpha, \alpha^2, \alpha^3 \right)^{-1}.
\end{equation}

See Appendix~\ref{appendix_Y1/kpibar} for the proof of Lemma~\ref{lemma_Y1/kpibar}.

\begin{lemma}
\label{lemma_Y1/kpibar}
$
\mathbf{Y}_\frac{1}{k}(x)
\in \overline{\pi}(4)
$
for every $x > 0$.
\end{lemma}

\begin{proof}[Proof of Lemma~\ref{lemma_Ylambdainpibar}]
The statement follows from Lemmas~\ref{lemma_XYlambdaneq1/k} (b)
and \ref{lemma_Y1/kpibar}
respectively for the case $\lambda \neq 1/k$
and
for the case $\lambda = 1/k$.
\end{proof}

\section{Proof of Theorem~\ref{theorem_existence}}
\label{section_existence}

Let
\begin{equation}
\label{equation_U2n}
\mathbf{U}_{2n}
 =
\frac{1}{\sqrt{2}}
\begin{pmatrix}
\mathbf{I}_n & 
\mathbf{R}_n \\
\mathbbm{i}
\mathbf{R}_n &
-
\mathbbm{i}
\mathbf{I}_n
\end{pmatrix}
\in
\gl(2n,\mathbb{C}),
\qquad
n \in \mathbb{N}.
\end{equation}
$\mathbf{U}_{2n}$ is unitary,
since
\begin{align*}
\mathbf{U}_{2n}^*
\mathbf{U}_{2n}
 &=
\frac{1}{\sqrt{2}}
\begin{pmatrix}
\mathbf{I}_n & 
-
\mathbbm{i}
\mathbf{R}_n \\
\mathbf{R}_n &
\mathbbm{i}
\mathbf{I}_n
\end{pmatrix}
\cdot
\frac{1}{\sqrt{2}}
\begin{pmatrix}
\mathbf{I}_n & 
\mathbf{R}_n \\
\mathbbm{i}
\mathbf{R}_n &
-
\mathbbm{i}
\mathbf{I}_n
\end{pmatrix}
 =
\begin{pmatrix}
\mathbf{I}_n & \mathbf{O}_n \\
\mathbf{O}_n & \mathbf{I}_n
\end{pmatrix}
 =
\mathbf{I}.
\end{align*}
Note also that
\begin{equation}
\label{equation_U2nR=U2nbar}
\mathbf{U}_{2n}
\mathbf{R}
 =
\frac{1}{\sqrt{2}}
\begin{pmatrix}
\mathbf{I}_n & 
\mathbf{R}_n \\
\mathbbm{i}
\mathbf{R}_n &
-
\mathbbm{i}
\mathbf{I}_n
\end{pmatrix}
\begin{pmatrix}
\mathbf{O}_n & \mathbf{R}_n \\
\mathbf{R}_n & \mathbf{O}_n
\end{pmatrix}
 =
\frac{1}{\sqrt{2}}
\begin{pmatrix}
\mathbf{I}_n & 
\mathbf{R}_n \\
-
\mathbbm{i}
\mathbf{R}_n &
\mathbbm{i}
\mathbf{I}_n
\end{pmatrix}
 =
\overline{
\mathbf{U}_{2n}
}.
\end{equation}

\begin{lemma}
\label{lemma_pibar2n=gl2nR}
For $n \in \mathbb{N}$,
The map 
$
\mathbf{A}
 \mapsto
\overline{\mathbf{U}_{2n}}
\mathbf{A}
\mathbf{U}_{2n}^T$
is an $\mathbb{R}$-algebra isomorphism
from $\overline{\pi}(2n)$
to $\gl(2n,\mathbb{R})$.
\end{lemma}

\begin{proof}
Since
$\mathbf{U}_{2n}$ is unitary,
and hence
$
\overline{\mathbf{U}_{2n}}^{-1}
 =
\mathbf{U}_{2n}^T$,
it is clear that the map
$
\mathbf{A}
 \mapsto
\overline{\mathbf{U}_{2n}}
\mathbf{A}
\mathbf{U}_{2n}^T$
is a $\mathbb{C}$-algebra isomorphism from
$\gl(2n,\mathbb{C})$ to $\gl(2n,\mathbb{C})$.
So it is sufficient to show that
$\mathbf{A} \in \overline{\pi}(2n)$
if and only if
$
\overline{\mathbf{U}_{2n}}
\mathbf{A}
\mathbf{U}_{2n}^T
\in \gl(2n,\mathbb{R})$.
Let
$
\mathbf{B}
 =
\overline{\mathbf{U}_{2n}}
\mathbf{A}
\mathbf{U}_{2n}^T
$.
Suppose $\mathbf{A} \in \overline{\pi}(2n)$.
Then by Definition~\ref{definition_pibar},
$
\overline{
\mathbf{A}
}
 =
\overline{
\mathbf{R}
\overline{
\mathbf{A}
}
\mathbf{R}
}
 =
\mathbf{R}
\mathbf{A}
\mathbf{R}
$,
hence
$
\overline{
\mathbf{B}
}
 =
\overline{
\overline{\mathbf{U}_{2n}}
\mathbf{A}
\mathbf{U}_{2n}^T
}
 =
\mathbf{U}_{2n}
\overline{
\mathbf{A}
}
\mathbf{U}_{2n}^*
 =
\mathbf{U}_{2n}
\cdot
\mathbf{R}
\mathbf{A}
\mathbf{R}
\cdot
\mathbf{U}_{2n}^*
 =
\left(
\mathbf{U}_{2n}
\mathbf{R}
\right)
\mathbf{A}
\left(
\mathbf{U}_{2n}
\mathbf{R}
\right)^*
 =
\overline{
\mathbf{U}_{2n}
}
\mathbf{A}
\mathbf{U}_{2n}^T
 =
\mathbf{B}
$
by \eqref{equation_U2nR=U2nbar}.
Thus
$
\mathbf{B}
\in \gl(2n,\mathbb{R})
$.
Conversely, suppose
$
\mathbf{B}
\in \gl(2n,\mathbb{R})
$.
Since
$
\mathbf{A}
 =
\left\{
\overline{\mathbf{U}_{2n}}
\right\}^{-1}
\mathbf{B}
\left\{
\mathbf{U}_{2n}^T
\right\}^{-1}
 =
\mathbf{U}_{2n}^T
\mathbf{B}
\overline{\mathbf{U}_{2n}}
$,
we have
$
\mathbf{R}
\overline{
\mathbf{A}
}
\mathbf{R}
 =
\mathbf{R}
\overline{
\left(
\mathbf{U}_{2n}^T
\mathbf{B}
\overline{
\mathbf{U}_{2n}
}
\right)
}
\mathbf{R}
 =
\left(
\mathbf{U}_{2n}
\mathbf{R}
\right)^*
\overline{
\mathbf{B}
}
\left(
\mathbf{U}_{2n}
\mathbf{R}
\right)
 =
\mathbf{U}_{2n}^T
\mathbf{B}
\overline{
\mathbf{U}_{2n}
}
 =
\mathbf{A}
$
by \eqref{equation_U2nR=U2nbar}.
Thus 
$\mathbf{A} \in \overline{\pi}(2n)$ by Definition~\ref{definition_pibar},
and the proof is complete.
\end{proof}

Note that $\mathbf{U}_4 = \mathbf{U}$,
where $\mathbf{U}$ is defined by \eqref{equation_U}.

\begin{lemma}
\label{lemma_existence}
For any $\mathbf{O} \neq \mathbf{G}_0 \in \overline{\pi}(4)$,
there exists $\mathbf{G} \in \overline{\pi}(4)$
such that
$
\det
\left(
\mathbf{G}
\mathbf{G}_0
-
\mathbf{I}
\right)
$
$
 =
0
$.
\end{lemma}

\begin{proof}
Let 
$
\hat{\mathbf{G}}_0
 =
\overline{\mathbf{U}}
\mathbf{G}_0
\mathbf{U}^T$,
which is not $\mathbf{O}$,
since $\mathbf{G}_0 \neq \mathbf{O}$.
By Lemma~\ref{lemma_pibar2n=gl2nR},
$\hat{\mathbf{G}}_0 \in \gl(4,\mathbb{R})$.
So there exists 
$
\mathbf{0}
 \neq
\mathbf{r}
\in \gl(4,1,\mathbb{R})$ such that
$
\hat{\mathbf{G}}_0
\mathbf{r}
 \neq 
\mathbf{0}
$.
It is clear that there exists
$\hat{\mathbf{G}} \in \gl(4,\mathbb{R})$
such that
$
\hat{\mathbf{G}}
\cdot
\hat{\mathbf{G}}_0
\mathbf{r}
 =
\mathbf{r}
$,
since
$
\hat{\mathbf{G}}_0
\mathbf{r}
 \neq 
\mathbf{0}
$.
Take
$
\mathbf{G}
 =
\mathbf{U}^T
\hat{\mathbf{G}}
\overline{\mathbf{U}}
$.
Then by Lemma~\ref{lemma_pibar2n=gl2nR},
$\mathbf{G} \in \overline{\pi}(4)$,
and
$
\mathbf{G}
\mathbf{G}_0
\cdot
\mathbf{U}^T
\mathbf{r}
 =
\mathbf{U}^T
\hat{\mathbf{G}}
\overline{\mathbf{U}}
\cdot
\mathbf{U}^T
\hat{\mathbf{G}}_0
\overline{\mathbf{U}}
\cdot
\mathbf{U}^T
\mathbf{r}
 =
\mathbf{U}^T
\cdot
\hat{\mathbf{G}}
\hat{\mathbf{G}}_0
\mathbf{r}
 =
\mathbf{U}^T
\mathbf{r}
$.
Since
$
\mathbf{U}^T
\mathbf{r}
 \neq
\mathbf{0}$,
it follows that $1$ is an eigenvalue of  
$
\mathbf{G}
\mathbf{G}_0
$,
which is equivalent to
$
\det
\left(
\mathbf{G}
\mathbf{G}_0
-
\mathbf{I}
\right)
 =
0
$.
\end{proof}

Lemma~\ref{lemma_Ylambda(x)neqO} below,
which is the last ingredient for the proof of Theorem~\ref{theorem_existence},
shows  that
$\mathbf{Y}_\lambda(x)$, when defined, never becomes the zero matrix for $x > 0$.
See Appendix~\ref{appendix_Ylambda(x)neqO} for its proof.

\begin{lemma}
\label{lemma_Ylambda(x)neqO}
$
\mathbf{Y}_\lambda(x)
 \neq
\mathbf{O}$
for every 
$0 \neq \lambda \in \mathbb{C}$ 
and 
$x > 0$
such that $\det\mathbf{X}_\lambda(x) \neq 0$.
\end{lemma}

\begin{proof}[Proof of Theorem~\ref{theorem_existence}]
Suppose
$
0
 \neq
\lambda 
\in \mathbb{R} \setminus \Spec\mathcal{K}_\mathbf{Q}$.
Then
$
\mathbf{O}
 \neq
\mathbf{Y}_\lambda(l) 
\in \overline{\pi}(4)
$
by Lemmas~\ref{lemma_Ylambdainpibar} and \ref{lemma_Ylambda(x)neqO}.
So by Lemma~\ref{lemma_existence}.
there exists $\mathbf{G} \in \overline{\pi}(4)$ such that
$
\det
\left\{
\mathbf{G}
\cdot
\mathbf{Y}_\lambda(l)
-
\mathbf{I}
\right\}
 =
0
$.
By Definitions~\ref{definition_wpR}, \ref{definition_GM}, and Lemma~\ref{lemma_GM},
there exists $\mathbf{M} \in \mathrm{wp}(4,8,\mathbb{R})$
such that
$
\mathbf{G}_\mathbf{M}
 =
\mathbf{G}
$,
since $\mathbf{G} \in \overline{\pi}(4)$.
Thus we have
$
\det
\left\{
\mathbf{G}_\mathbf{M}
\cdot
\mathbf{Y}_\lambda(l)
-
\mathbf{I}
\right\}
$
$
 =
0
$,
which
implies that
$\lambda \in \Spec\mathcal{K}_\mathbf{M}$
by Corollary~\ref{corollary_eigencondition-Y}.
\end{proof}

\section{Discussion}
\label{section_discussion}

The $4 \times 4$ matrices $\mathbf{X}_\lambda(x)$ and $\mathbf{Y}_\lambda(x)$ 
turn out to be rich in symmetries.
In fact, only part of their symmetries are exploited
to prove our results in this paper.
We have also tried to refrain, as possible as we can,
from resorting to more explicit forms of $\mathbf{X}_\lambda(x)$ and $\mathbf{Y}_\lambda(x)$,
despite of their explicit nature.
In view of what can be done more in these respects,
it is expected that we will have a clearer picture
of general well-posed boundary value problems
for finite beam deflection,
if we pursue closer investigations on $\mathbf{X}_\lambda(x)$ and $\mathbf{Y}_\lambda(x)$.
Especially, detailed results such as Proposition~\ref{proposition_Q},
which is for only one specific boundary condition $\mathbf{Q}$,
are expected to be obtained for the class of {\em all} well-posed boundary conditions.

\appendix

\renewcommand{\theequation}{\thesection.\arabic{equation}}
\renewcommand{\thelemma}{\thesection\arabic{lemma}}


\section{Proof of Lemma~\ref{lemma_KM}}
\label{appendix_Green}

By Definition~\ref{definition_GMKM},
we have
\begin{align}
\lefteqn{
\mathcal{K}_\mathbf{M}[w](x)
 =
-
\frac{\alpha}{4k}
\int_{-l}^x
\mathbf{y}(x)^T
\cdot
\mathbf{G}_\mathbf{M}^-
\cdot
\mathbf{y}(\xi)
w(\xi)
\, d\xi
}
\nonumber \\
 &\qquad\qquad\qquad
+
\frac{\alpha}{4k}
\int_x^l
\mathbf{y}(x)^T
\cdot
\mathbf{G}_\mathbf{M}^+
\cdot
\mathbf{y}(\xi)
w(\xi)
\, d\xi
\nonumber \\
 &=
\frac{\alpha}{4k}
\cdot
\mathbf{y}(x)^T
\left\{
-
\mathbf{G}_\mathbf{M}^-
\int_{-l}^x
\mathbf{y}(\xi)
w(\xi)
\, d\xi
+
\mathbf{G}_\mathbf{M}^+
\int_x^l
\mathbf{y}(\xi)
w(\xi)
\, d\xi
\right\},
\label{equation_KMux}
\end{align}
and by \eqref{equation_GM-+GM+},
\begin{align}
\lefteqn{
\frac{d}{dx}
\left\{
-
\mathbf{G}_\mathbf{M}^-
\int_{-l}^x
\mathbf{y}(\xi)
w(\xi)
\, d\xi
+
\mathbf{G}_\mathbf{M}^+
\int_x^l
\mathbf{y}(\xi)
w(\xi)
\, d\xi
\right\}
}
\nonumber \\
 &=
-
\mathbf{G}_\mathbf{M}^-
\cdot
\frac{d}{dx}
\int_{-l}^x
\mathbf{y}(\xi)
w(\xi)
\, d\xi
+
\mathbf{G}_\mathbf{M}^+
\cdot
\frac{d}{dx}
\int_x^l
\mathbf{y}(\xi)
w(\xi)
\, d\xi
\nonumber \\
 &=
-
\mathbf{G}_\mathbf{M}^-
\cdot
\mathbf{y}(x)
w(x)
-
\mathbf{G}_\mathbf{M}^+
\cdot
\mathbf{y}(x)
w(x)
 =
-
\mathbf{\Omega}
\mathbf{L}^2
\cdot
\mathbf{y}(x)
w(x).
\label{equation_ao;fvafoj}
\end{align}
Let
\begin{equation}
\label{equation_boldfx}
\mathbf{f}(x)
 =
-
\mathbf{G}_\mathbf{M}^-
\int_{-l}^x
\mathbf{y}(\xi)
w(\xi)
\, d\xi
+
\mathbf{G}_\mathbf{M}^+
\int_x^l
\mathbf{y}(\xi)
w(\xi)
\, d\xi,
\end{equation}
so that
\begin{equation}
\label{equation_ahvofv'afvj}
\mathcal{K}_\mathbf{M}[w](x)
 =
\frac{\alpha}{4k}
\cdot
\mathbf{y}(x)^T
\cdot
\mathbf{f}(x),
\end{equation}
and
$
\mathbf{f}^\prime(x)
 =
-
\mathbf{\Omega}
\mathbf{L}^2
\cdot
\mathbf{y}(x)
w(x)
$
by \eqref{equation_KMux}, \eqref{equation_ao;fvafoj}.
By \eqref{equation_L2},
\begin{align}
\lefteqn{
\mathbf{y}(x)^T
\cdot
\mathbf{\Omega}^n
\cdot
\mathbf{f}^\prime(x)
 =
-
\mathbf{y}(x)^T
\cdot
\mathbf{\Omega}^n
\cdot
\mathbf{\Omega}
\mathbf{L}^2
\cdot
\mathbf{y}(x)
w(x)
}
\nonumber \\
 &=
-
w(x)
\left\{
\mathbf{y}(x)^T
\cdot
\mathbf{\Omega}^{n+1}
\cdot
\mathbf{y}(-x)
\right\}
\nonumber \\
 &=
-w(x)
\begin{pmatrix}
e^{\omega_1 \alpha x} &
e^{\omega_2 \alpha x} &
e^{\omega_3 \alpha x} &
e^{\omega_4 \alpha x}
\end{pmatrix}
\cdot
\nonumber \\
 &\qquad
\cdot
\diag
\left(
\omega_1^{n+1},
\omega_2^{n+1},
\omega_3^{n+1},
\omega_4^{n+1}
\right)
\cdot
\begin{pmatrix}
e^{-\omega_1 \alpha x} \\
e^{-\omega_2 \alpha x} \\
e^{-\omega_3 \alpha x} \\
e^{-\omega_4 \alpha x}
\end{pmatrix}
\nonumber \\
 &=
-
w(x)
\sum_{j=1}^4
e^{\omega_j \alpha x}
\omega_j^{n+1}
e^{-\omega_j \alpha x}
 =
-
w(x)
\sum_{j=1}^4
\omega_j^{n+1},
\quad
n = 0,1,2,\ldots
\label{equation_asfovhafvhj}
\end{align}
By \eqref{equation_omega},
$
\sum_{j=1}^4
\omega_j^4
 =
\sum_{j=1}^4
(-1)
 =
-4
$,
and
$
\sum_{j=1}^4
\omega_j^2
 =
\sum_{j=1}^4
\left(
\mathbbm{i}^{j-1}
\omega_1
\right)^2
 =
\omega_1^2
\cdot
$
$
\cdot
\sum_{j=1}^4
(-1)^{j-1}
 =
0
$.
By \eqref{equation_omegaL2},
$
\sum_{j=1}^4
\omega_j
 =
0
$,
hence by \eqref{equation_omega},
$
\sum_{j=1}^4
\omega_j^3
 =
\sum_{j=1}^4
\left(
-\overline{\omega_j}
\right)
$
$
 =
-
\overline{
\sum_{j=1}^4
\omega_j
}
 =
0
$.
So by \eqref{equation_asfovhafvhj},
we have
\begin{equation}
\label{equation_av;oafhvj;afjhvb}
\mathbf{y}(x)^T
\cdot
\mathbf{\Omega}^n
\cdot
\mathbf{f}^\prime(x)
 =
0,
\quad
n = 0,1,2,
\qquad
\mathbf{y}(x)^T
\cdot
\mathbf{\Omega}^3
\cdot
\mathbf{f}^\prime(x)
 =
4
\cdot
w(x).
\end{equation}
By \eqref{equation_y'x},
$
\mathbf{y}^\prime(x)^T
 =
\left\{
\alpha
\mathbf{\Omega}
\cdot
\mathbf{y}(x)
\right\}^T
 =
\alpha
\cdot
\mathbf{y}(x)^T
\cdot
\mathbf{\Omega}
$.
So by \eqref{equation_ahvofv'afvj}, \eqref{equation_av;oafhvj;afjhvb},
we have
\begin{align}
\mathcal{K}_\mathbf{M}[w]^\prime(x)
 &=
\frac{\alpha}{4k}
\left\{
\mathbf{y}^\prime(x)^T
\cdot
\mathbf{f}(x)
+
\mathbf{y}(x)^T
\cdot
\mathbf{f}^\prime(x)
\right\}
\nonumber \\
 &=
\frac{\alpha^2}{4k}
\cdot
\mathbf{y}(x)^T
\cdot
\mathbf{\Omega}
\cdot
\mathbf{f}(x),
\label{equation_sojvoiujvb1} \\
\mathcal{K}_\mathbf{M}[w]^{\prime\prime}(x)
 &=
\frac{\alpha^2}{4k}
\left\{
\mathbf{y}^\prime(x)^T
\cdot
\mathbf{\Omega}
\cdot
\mathbf{f}(x)
+
\mathbf{y}(x)^T
\cdot
\mathbf{\Omega}
\cdot
\mathbf{f}^\prime(x)
\right\}
\nonumber \\
 &=
\frac{\alpha^3}{4k}
\cdot
\mathbf{y}(x)^T
\cdot
\mathbf{\Omega}^2
\cdot
\mathbf{f}(x),
\label{equation_sojvoiujvb2} \\
\mathcal{K}_\mathbf{M}[w]^{\prime\prime\prime}(x)
 &=
\frac{\alpha^3}{4k}
\left\{
\mathbf{y}^\prime(x)^T
\cdot
\mathbf{\Omega}^2
\cdot
\mathbf{f}(x)
+
\mathbf{y}(x)^T
\cdot
\mathbf{\Omega}^2
\cdot
\mathbf{f}^\prime(x)
\right\}
\nonumber \\
 &=
\frac{\alpha^4}{4k}
\cdot
\mathbf{y}(x)^T
\cdot
\mathbf{\Omega}^3
\cdot
\mathbf{f}(x),
\label{equation_sojvoiujvb3}
\end{align}
hence by 
\eqref{equation_Omega}, 
\eqref{equation_y'x},
\eqref{equation_ahvofv'afvj}, 
\eqref{equation_av;oafhvj;afjhvb},
\begin{align*}
\lefteqn{
\mathcal{K}_\mathbf{M}[w]^{(4)}(x)
 =
\frac{\alpha^4}{4k}
\left\{
\mathbf{y}^\prime(x)^T
\cdot
\mathbf{\Omega}^3
\cdot
\mathbf{f}(x)
+
\mathbf{y}(x)^T
\cdot
\mathbf{\Omega}^3
\cdot
\mathbf{f}^\prime(x)
\right\}
} \\
 &=
\frac{\alpha^4}{4k}
\left\{
\alpha
\cdot
\mathbf{y}(x)^T
\cdot
\mathbf{\Omega}^4
\cdot
\mathbf{f}(x)
+
4
\cdot
w(x)
\right\}
 =
\frac{\alpha^4}{4k}
\left\{
-
\alpha
\cdot
\mathbf{y}(x)^T
\cdot
\mathbf{f}(x)
+
4
\cdot
w(x)
\right\} \\
 &=
\frac{\alpha^4}{4k}
\left\{
-
\alpha
\cdot
\frac{4k}{\alpha}
\mathcal{K}_\mathbf{M}[w](x)
+
4
\cdot
w(x)
\right\}
 =
-
\alpha^4
\cdot
\mathcal{K}_\mathbf{M}[w](x)
+
\frac{\alpha^4}{k}
\cdot
w(x).
\end{align*}
This shows that $\mathcal{K}_\mathbf{M}[w](x)$
satisfies
$\mathrm{DE}(w)$.

By 
\eqref{equation_ahvofv'afvj},
\eqref{equation_sojvoiujvb1},
\eqref{equation_sojvoiujvb2},
\eqref{equation_sojvoiujvb3},
and
\eqref{equation_y'x},
we have
\begin{align}
\mathcal{K}_\mathbf{M}[w]^{(n)}(x)
 &=
\frac{\alpha^{n+1}}{4k}
\cdot
\mathbf{y}(x)^T
\cdot
\mathbf{\Omega}^n
\cdot
\mathbf{f}(x)
 =
\frac{\alpha}{4k}
\left\{
\alpha^n
\mathbf{\Omega}^n
\cdot
\mathbf{y}(x)
\right\}^T
\cdot
\mathbf{f}(x)
\nonumber \\
 &=
\frac{\alpha}{4k}
\cdot
\mathbf{y}^{(n)}(x)^T
\cdot
\mathbf{f}(x),
\qquad
n = 0,1,2,3.
\label{equation_afoba;foivjaj}
\end{align}
By \eqref{equation_boldfx},
$
\mathbf{f}(\pm l)
 =
\mp
\mathbf{G}_\mathbf{M}^\mp
\int_{-l}^l
\mathbf{y}(\xi)
w(\xi)
\, d\xi
$,
hence by \eqref{equation_afoba;foivjaj},
\[
\mathcal{K}_\mathbf{M}[w]^{(n)}(\pm l)
 =
\mp
\frac{\alpha}{4k}
\cdot
\mathbf{y}^{(n)}(\pm l)^T
\cdot
\mathbf{G}_\mathbf{M}^\mp
\int_{-l}^l
\mathbf{y}(\xi)
w(\xi)
\, d\xi,
\qquad
n = 0,1,2,3.
\]
So by Definitions~\ref{definition_W(x)} and \ref{definition_calBpm},
\begin{equation}
\label{equation_lih;hfvav}
\mathcal{B}^\pm
\left[
\mathcal{K}_\mathbf{M}[w]
\right]
 =
\mp
\frac{\alpha}{4k}
\cdot
\mathbf{W}(\pm l)
\cdot
\mathbf{G}_\mathbf{M}^\mp
\int_{-l}^l
\mathbf{y}(\xi)
w(\xi)
\, d\xi.
\end{equation}
Let
$\mathbf{M}^-, \mathbf{M}^+ \in \gl(4,\mathbb{C})$
be the $4 \times 4$ minors of 
$\mathbf{M}$ such that
$
\mathbf{M}
 = 
\begin{pmatrix}
\mathbf{M}^-  & \mathbf{M}^+ 
\end{pmatrix}
$.
By Definitions~\ref{definition_Mtilde}, \ref{definition_GMpm}, and \eqref{equation_lih;hfvav},
\begin{align*}
\mathbf{M}^\pm
\cdot
\mathcal{B}^\pm
\left[
\mathcal{K}_\mathbf{M}[w]
\right]
 &=
\mp
\frac{\alpha}{4k}
\cdot
\mathbf{M}^\pm
\mathbf{W}(\pm l)
\cdot
\mathbf{G}_\mathbf{M}^\mp
\int_{-l}^l
\mathbf{y}(\xi)
w(\xi)
\, d\xi \\
 &=
\mp
\frac{\alpha}{4k}
\cdot
\widetilde{\mathbf{M}}^\pm
\cdot
\widetilde{\mathbf{M}}^{-1}
\widetilde{\mathbf{M}}^\mp
\mathbf{\Omega}
\mathbf{L}^2
\int_{-l}^l
\mathbf{y}(\xi)
w(\xi)
\, d\xi,
\end{align*}
hence by \eqref{equation_MBpm},
\begin{align*}
\mathbf{M}
\cdot
\mathcal{B}
\left[
\mathcal{K}_\mathbf{M}[w]
\right]
 &=
\mathbf{M}^-
\cdot
\mathcal{B}^-
\left[
\mathcal{K}_\mathbf{M}[w]
\right]
+
\mathbf{M}^+
\cdot
\mathcal{B}^+
\left[
\mathcal{K}_\mathbf{M}[w]
\right] \\
 &=
\frac{\alpha}{4k}
\widetilde{\mathbf{M}}^-
\widetilde{\mathbf{M}}^{-1}
\widetilde{\mathbf{M}}^+
\mathbf{\Omega}
\mathbf{L}^2
\int_{-l}^l
\mathbf{y}(\xi)
w(\xi)
\, d\xi \\
 &\quad
-
\frac{\alpha}{4k}
\widetilde{\mathbf{M}}^+
\widetilde{\mathbf{M}}^{-1}
\widetilde{\mathbf{M}}^-
\mathbf{\Omega}
\mathbf{L}^2
\int_{-l}^l
\mathbf{y}(\xi)
w(\xi)
\, d\xi \\
 &=
\frac{\alpha}{4k}
\left(
\widetilde{\mathbf{M}}^-
\widetilde{\mathbf{M}}^{-1}
\widetilde{\mathbf{M}}^+
-
\widetilde{\mathbf{M}}^+
\widetilde{\mathbf{M}}^{-1}
\widetilde{\mathbf{M}}^-
\right)
\mathbf{\Omega}
\mathbf{L}^2
\int_{-l}^l
\mathbf{y}(\xi)
w(\xi)
\, d\xi.
\end{align*}
Thus we have
$
\mathbf{M}
\cdot
\mathcal{B}
\left[
\mathcal{K}_\mathbf{M}[w]
\right]
 =
\mathbf{0}
$,
since
$\widetilde{\mathbf{M}} = \widetilde{\mathbf{M}}^- + \widetilde{\mathbf{M}}^+$ 
by Definition~\ref{definition_Mtilde}
and hence
\begin{align*}
\lefteqn{
\widetilde{\mathbf{M}}^-
\widetilde{\mathbf{M}}^{-1}
\widetilde{\mathbf{M}}^+
-
\widetilde{\mathbf{M}}^+
\widetilde{\mathbf{M}}^{-1}
\widetilde{\mathbf{M}}^-
} \\
 &=
\left(
\widetilde{\mathbf{M}} - \widetilde{\mathbf{M}}^+
\right)
\widetilde{\mathbf{M}}^{-1}
\widetilde{\mathbf{M}}^+
-
\widetilde{\mathbf{M}}^+
\widetilde{\mathbf{M}}^{-1}
\left(
\widetilde{\mathbf{M}} - \widetilde{\mathbf{M}}^+
\right) \\
 &=
\widetilde{\mathbf{M}}^+
-
\widetilde{\mathbf{M}}^+
\widetilde{\mathbf{M}}^{-1}
\widetilde{\mathbf{M}}^+
-
\widetilde{\mathbf{M}}^+
+
\widetilde{\mathbf{M}}^+
\widetilde{\mathbf{M}}^{-1}
\widetilde{\mathbf{M}}^+
 =
\mathbf{O}.
\end{align*}
This shows that $\mathcal{K}_\mathbf{M}[w]$
satisfies
$\mathrm{BC}(\mathbf{M})$,
and the proof is complete.

\section{Proof of Lemma~\ref{lemma_Y1/kpibar}}
\label{appendix_Y1/kpibar}


Denote
\begin{equation}
\label{equation_pniz}
p_{n,i}(z)
 =
\sum_{r=0}^n
\frac{\omega_i^{n-r}}{r!}
z^r,
\qquad
n = 0,1,2,3,
\
i \in \mathbb{Z},
\end{equation}
where it is understood that $0^0 = 1$.
In particular, denote
\begin{equation}
\label{equation_pnz}
p_n(z)
 =
p_{n,1}(z),
\qquad
n = 0,1,2,3.
\end{equation}
By Definitions~\ref{definition_omega}, \ref{definition_R}, and \eqref{equation_R},
\[
\mathbf{W}_0^*
 =
\overline{
\mathbf{W}_0
}^T
 =
\left(
\mathbf{W}_0
\mathbf{R}
\right)^T
 =
\mathbf{R}
\mathbf{W}_0^T
 =
\mathbf{R}
\cdot
\left(
\omega_i^{j-1}
\right)_{1 \leq i,j \leq 4}
 =
\left(
\omega_{5-i}^{j-1}
\right)_{1 \leq i,j \leq 4},
\]
hence by \eqref{equation_Wlambdax}, \eqref{equation_pniz},
we have
\begin{align*}
\lefteqn{
\diag(0,1,1,0)
\cdot
\mathbf{W}_0^*
\mathbf{W}_\frac{1}{k}(-z)
} \\
 &=
\begin{pmatrix}
0 & 0 & 0 & 0 \\
1 & \omega_3 & \omega_3^2 & \omega_3^3 \\
1 & \omega_2 & \omega_2^2 & \omega_2^3 \\
0 & 0 & 0 & 0
\end{pmatrix}
\begin{pmatrix}
(-z)^0 & (-z)^1 & \frac{1}{2} (-z)^2 & \frac{1}{6} (-z)^3 \\
0 & (-z)^0 & (-z)^1 & \frac{1}{2} (-z)^2 \\
0 & 0 & (-z)^0 & (-z)^1 \\
0 & 0 & 0 & (-z)^0
\end{pmatrix} \\
 &=
\begin{pmatrix}
0 & 0 & 0 & 0 \\
p_{0,3}(-z) & p_{1,3}(-z) & p_{2,3}(-z) & p_{3,3}(-z) \\
p_{0,2}(-z) & p_{1,2}(-z) & p_{2,2}(-z) & p_{3,2}(-z) \\
0 & 0 & 0 & 0
\end{pmatrix},
\end{align*}
\begin{align*}
\diag(1,0,0,1)
\cdot
\mathbf{W}_0^*
\mathbf{W}_\frac{1}{k}(z)
 &=
\begin{pmatrix}
1 & \omega_4 & \omega_4^2 & \omega_4^3 \\
0 & 0 & 0 & 0 \\
0 & 0 & 0 & 0 \\
1 & \omega_1 & \omega_1^2 & \omega_1^3
\end{pmatrix}
\begin{pmatrix}
z^0 & z^1 & \frac{1}{2} z^2 & \frac{1}{6} z^3 \\
0 & z^0 & z^1 & \frac{1}{2} z^2 \\
0 & 0 & z^0 & z^1 \\
0 & 0 & 0 & z^0
\end{pmatrix} \\
 &=
\begin{pmatrix}
p_{0,4}(z) & p_{1,4}(z) & p_{2,4}(z) & p_{3,4}(z) \\
0 & 0 & 0 & 0 \\
0 & 0 & 0 & 0 \\
p_{0,1}(z) & p_{1,1}(z) & p_{2,1}(z) & p_{3,1}(z)
\end{pmatrix}.
\end{align*}
Thus by \eqref{equation_Pz},
we have
\begin{equation}
\label{equation_adofhbad;fov}
\mathbf{P}(z)
 =
\begin{pmatrix}
p_{0,4}(z) & p_{1,4}(z) & p_{2,4}(z) & p_{3,4}(z) \\
p_{0,3}(-z) & p_{1,3}(-z) & p_{2,3}(-z) & p_{3,3}(-z) \\
p_{0,2}(-z) & p_{1,2}(-z) & p_{2,2}(-z) & p_{3,2}(-z) \\
p_{0,1}(z) & p_{1,1}(z) & p_{2,1}(z) & p_{3,1}(z)
\end{pmatrix}.
\end{equation}
Note from 
\eqref{equation_omegaR}, \eqref{equation_omegaL2},
\eqref{equation_pniz}, \eqref{equation_pnz} that,
for $n = 0,1,2,3$,
\begin{align*}
p_{n,4}(z)
 &=
\sum_{r=0}^n
\frac{\omega_4^{n-r}}{r!}
z^r
 =
\sum_{r=0}^n
\frac{\overline{\omega_1}^{n-r}}{r!}
z^r
 =
\overline{
\left(
\sum_{r=0}^n
\frac{\omega_1^{n-r}}{r!}
z^r
\right)
}
 =
\overline{
p_{n,1}(z)
}
 =
\overline{
p_n(z)
}, \\
p_{n,3}(-z)
 &=
\sum_{r=0}^n
\frac{\omega_3^{n-r}}{r!}
(-z)^r
 =
\sum_{r=0}^n
\frac{\left( -\omega_1 \right)^{n-r}}{r!}
(-z)^r
 =
(-1)^n
\sum_{r=0}^n
\frac{\omega_1^{n-r}}{r!}
z^r \\
 &=
(-1)^n
p_{n,1}(z)
 =
(-1)^n
p_n(z), \\
p_{n,2}(-z)
 &=
\sum_{r=0}^n
\frac{\omega_2^{n-r}}{r!}
(-z)^r
 =
\sum_{r=0}^n
\frac{\overline{\omega_3}^{n-r}}{r!}
(-z)^r
 =
\overline{
\left\{
\sum_{r=0}^n
\frac{\omega_3^{n-r}}{r!}
(-z)^r
\right\}
} \\
 &=
\overline{
p_{n,3}(-z)
}
 =
\overline{
(-1)^n
p_n(z)
}
 =
(-1)^n
\overline{
p_n(z)
}.
\end{align*}
Thus by \eqref{equation_adofhbad;fov},
we have
\begin{equation}
\label{equation_Pzexplicit}
\mathbf{P}(z)
 =
\begin{pmatrix}
\overline{p_0(z)} & \overline{p_1(z)} & \overline{p_2(z)} & \overline{p_3(z)} \\
p_0(z) & -p_1(z) & p_2(z) & -p_3(z) \\
\overline{p_0(z)} & -\overline{p_1(z)} & \overline{p_2(z)} & -\overline{p_3(z)} \\
p_0(z) & p_1(z) & p_2(z) & p_3(z)
\end{pmatrix}.
\end{equation}

Denote
\begin{equation}
\label{equation_Ppmz}
\mathbf{P}^+(z)
 =
\begin{pmatrix}
\overline{p_0(z)} & \overline{p_2(z)} \\
p_0(z) & p_2(z)
\end{pmatrix},
\qquad
\mathbf{P}^-(z)
 =
\begin{pmatrix}
-\overline{p_1(z)} & -\overline{p_3(z)} \\
p_1(z) & p_3(z)
\end{pmatrix},
\end{equation}
\begin{equation}
\label{equation_VVhat}
\mathbf{V}
 =
\frac{1}{\sqrt{2}}
\begin{pmatrix}
\mathbf{I} & \mathbf{I} \\
-\mathbf{I} & \mathbf{I}
\end{pmatrix}
 =
\frac{1}{\sqrt{2}}
\begin{pmatrix}
1 & 0 & 1 & 0 \\
0 & 1 & 0 & 1 \\
-1 & 0 & 1 & 0 \\
0 & -1 & 0 & 1
\end{pmatrix},
\
\hat{\mathbf{V}}
 =
\begin{pmatrix}
1 & 0 & 0 & 0 \\
0 & 0 & 1 & 0 \\
0 & 1 & 0 & 0 \\
0 & 0 & 0 & 1
\end{pmatrix}.
\end{equation}
Note that 
$\mathbf{V}, \hat{\mathbf{V}} \in O(4)$
and
\begin{equation}
\label{equation_detVVhat}
\det\mathbf{V} = 1,
\qquad
\det\hat{\mathbf{V}} = -1.
\end{equation}

\begin{lemma}
\label{lemma_VPVhat}
$
\mathbf{V}
\cdot
\mathbf{P}(z)
\cdot
\hat{\mathbf{V}}
 =
\sqrt{2}
\begin{pmatrix}
\mathbf{P}^+(z) & \mathbf{O} \\
\mathbf{O} & \mathbf{P}^-(z)
\end{pmatrix}
$.
\end{lemma}

\begin{proof}
By \eqref{equation_Pzexplicit}, \eqref{equation_VVhat},
\begin{align*}
\mathbf{V}
\cdot
\mathbf{P}(z)
 &=
\frac{1}{\sqrt{2}}
\begin{pmatrix}
1 & 0 & 1 & 0 \\
0 & 1 & 0 & 1 \\
-1 & 0 & 1 & 0 \\
0 & -1 & 0 & 1
\end{pmatrix}
\begin{pmatrix}
\overline{p_0(z)} & \overline{p_1(z)} & \overline{p_2(z)} & \overline{p_3(z)} \\
p_0(z) & -p_1(z) & p_2(z) & -p_3(z) \\
\overline{p_0(z)} & -\overline{p_1(z)} & \overline{p_2(z)} & -\overline{p_3(z)} \\
p_0(z) & p_1(z) & p_2(z) & p_3(z)
\end{pmatrix} \\
 &=
\sqrt{2}
\begin{pmatrix}
\overline{p_0(z)} & 0 & \overline{p_2(z)} & 0 \\
p_0(z) & 0 & p_2(z) & 0 \\
0 & -\overline{p_1(z)} & 0 & -\overline{p_3(z)} \\
0 & p_1(z) & 0 & p_3(z)
\end{pmatrix},
\end{align*}
hence the lemma follows,
since
multiplying $\hat{\mathbf{V}}$ on the right
amounts to interchanging
the second and the third columns.
\end{proof}

By Lemma~\ref{lemma_VPVhat},
we have
\begin{equation}
\label{equation_va;ofvijafov}
\mathbf{P}(z)
 =
\mathbf{V}^{-1}
\cdot
\sqrt{2}
\begin{pmatrix}
\mathbf{P}^+(z) & \mathbf{O} \\
\mathbf{O} & \mathbf{P}^-(z)
\end{pmatrix}
\cdot
\hat{\mathbf{V}}^{-1}
 =
\sqrt{2}
\cdot
\mathbf{V}^T
\begin{pmatrix}
\mathbf{P}^+(z) & \mathbf{O} \\
\mathbf{O} & \mathbf{P}^-(z)
\end{pmatrix}
\hat{\mathbf{V}}^T,
\end{equation}
since $\mathbf{V}$, $\hat{\mathbf{V}}$ are orthogonal.
So by \eqref{equation_detVVhat}
\begin{align*}
\det
\mathbf{P}(z)
 &=
\sqrt{2}^4
\cdot
\det
\mathbf{V}
\cdot
\det
\begin{pmatrix}
\mathbf{P}^+(z) & \mathbf{O} \\
\mathbf{O} & \mathbf{P}^-(z)
\end{pmatrix}
\cdot
\det
\hat{\mathbf{V}} \\
 &=
-
4
\cdot
\det
\begin{pmatrix}
\mathbf{P}^+(z) & \mathbf{O} \\
\mathbf{O} & \mathbf{P}^-(z)
\end{pmatrix}
 =
-
4
\cdot
\det
\mathbf{P}^+(z)
\cdot
\det
\mathbf{P}^-(z),
\end{align*}
hence by \eqref{equation_anvrfuvh;aof},
\begin{align}
\det
\mathbf{X}_\frac{1}{k}(x)
 &=
\left(
\frac{1}{4}
\right)^4
\cdot
\det
e^{-\mathcal{E} \mathbf{\Omega} z}
\cdot
\det
\mathbf{P}(z)
\cdot
\det
\diag
\left( 1, \alpha, \alpha^2, \alpha^3 \right)^{-1}
\nonumber \\
 &=
-
\frac{e^{-2\sqrt{2} z}}{4^3 \alpha^6}
\cdot
\det
\mathbf{P}^+(z)
\cdot
\det
\mathbf{P}^-(z),
\label{equation_avofjvafpjk}
\end{align}
since
$
\det
\diag
\left( 1, \alpha, \alpha^2, \alpha^3 \right)^{-1}
 =
1
\cdot
\alpha^{-1}
\cdot
\alpha^{-2}
\cdot
\alpha^{-3}
 =
\alpha^{-6}
$,
and
$
\det
e^{-\mathcal{E} \mathbf{\Omega} z}
 =
e^{-\omega_1 z}
\cdot
e^{\omega_2 z}
\cdot
e^{\omega_3 z}
\cdot
e^{-\omega_4 z}
 =
e^{-\left\{ \left( \omega_1 - \omega_3 \right) + \left( \omega_4 - \omega_2 \right) \right\} z}
 =
e^{-2 \left( \omega_1 + \omega_4 \right) z} 
 =
e^{-2 \cdot 2 \Real{\omega_1} \cdot z}
 =
e^{-2\sqrt{2} z}
$
by
\eqref{equation_omegaR}, \eqref{equation_omegaL2}, \eqref{equation_omegaReIm}.
Since 
$
\det
\mathbf{X}_\frac{1}{k}(x)
 \neq
0
$
for every $x > 0$
by Proposition~\ref{proposition_Q}, Corollary~\ref{corollary_eigencondition-det},
and Lemma~\ref{lemma_GQ=Glalphak},
it follows from \eqref{equation_z}, \eqref{equation_avofjvafpjk} that
$\det\mathbf{P}^+(z) \neq 0$ and $\det\mathbf{P}^-(z) \neq 0$
for every $z > 0$.
From \eqref{equation_Ppmz},
we have
\begin{align}
\det
\mathbf{P}^+(z)
 &=
\overline{p_0(z)}
p_2(z)
-
p_0(z)
\overline{p_2(z)}
 =
2
\mathbbm{i}
\Imag
\left\{
\overline{p_0(z)}
p_2(z)
\right\},
\label{equation_vfhvafvjvasf+} \\
\det
\mathbf{P}^-(z)
 &=
p_1(z)
\overline{p_3(z)}
-
\overline{p_1(z)}
p_3(z)
 =
2
\mathbbm{i}
\Imag
\left\{
p_1(z)
\overline{p_3(z)}
\right\}.
\label{equation_vfhvafvjvasf-}
\end{align}

Note from Definition~\ref{definition_pibar} that 
$
\begin{pmatrix}
a_{11} & a_{12} \\
a_{21} & a_{22}
\end{pmatrix}
\in \overline{\pi}(2)
$,
if and only if
$a_{11} = \overline{a_{22}}$
and
$a_{12} = \overline{a_{21}}$.

\begin{lemma}
\label{lemma_Ppibar}
$
\mathbf{P}^+(-z)
\mathbf{P}^+(z)^{-1},
\mathbf{P}^-(-z)
\mathbf{P}^-(z)^{-1}
\in
\overline{\pi}(2)
$
for every $z > 0$.
\end{lemma}

\begin{proof}
From \eqref{equation_Ppmz},
we have
\begin{align*}
\lefteqn{
\mathbbm{i}
\mathbf{P}^+(-z)
\cdot
\adj{\mathbf{P}^+(z)}
 =
\mathbbm{i}
\begin{pmatrix}
\overline{p_0(-z)} & \overline{p_2(-z)} \\
p_0(-z) & p_2(-z)
\end{pmatrix}
\begin{pmatrix}
p_2(z) & -\overline{p_2(z)} \\
-p_0(z) & \overline{p_0(z)}
\end{pmatrix}
} \\
 &=
\begin{pmatrix}
\mathbbm{i}
\left\{
\overline{p_0(-z)}
p_2(z)
-
\overline{p_2(-z)}
p_0(z)
\right\}
 &
\mathbbm{i}
\left\{
-
\overline{p_0(-z)}
\overline{p_2(z)}
+
\overline{p_2(-z)}
\overline{p_0(z)}
\right\}
 \\
\mathbbm{i}
\left\{
p_0(-z)
p_2(z)
-
p_2(-z)
p_0(z)
\right\}
 &
\mathbbm{i}
\left\{
-
p_0(-z)
\overline{p_2(z)}
+
p_2(-z)
\overline{p_0(z)}
\right\}
\end{pmatrix} \\
 &=
\begin{pmatrix}
\overline{
\mathbbm{i}
\left\{
-
p_0(-z)
\overline{p_2(z)}
+
p_2(-z)
\overline{p_0(z)}
\right\}
}
 &
\overline{
\mathbbm{i}
\left\{
p_0(-z)
p_2(z)
-
p_2(-z)
p_0(z)
\right\}
}
 \\
\mathbbm{i}
\left\{
p_0(-z)
p_2(z)
-
p_2(-z)
p_0(z)
\right\}
 &
\mathbbm{i}
\left\{
-
p_0(-z)
\overline{p_2(z)}
+
p_2(-z)
\overline{p_0(z)}
\right\}
\end{pmatrix}, \\
\lefteqn{
\mathbbm{i}
\mathbf{P}^-(-z)
\cdot
\adj{\mathbf{P}^-(z)}
 =
\mathbbm{i}
\begin{pmatrix}
-\overline{p_1(-z)} & -\overline{p_3(-z)} \\
p_1(-z) & p_3(-z)
\end{pmatrix}
\begin{pmatrix}
p_3(z) & \overline{p_3(z)} \\
-p_1(z) & -\overline{p_1(z)}
\end{pmatrix}
} \\
 &=
\begin{pmatrix}
\mathbbm{i}
\left\{
-
\overline{p_1(-z)}
p_3(z)
+
\overline{p_3(-z)}
p_1(z)
\right\}
 &
\mathbbm{i}
\left\{
-
\overline{p_1(-z)}
\overline{p_3(z)}
+
\overline{p_3(-z)}
\overline{p_1(z)}
\right\}
 \\
\mathbbm{i}
\left\{
p_1(-z)
p_3(z)
-
p_3(-z)
p_1(z)
\right\}
 &
\mathbbm{i}
\left\{
p_1(-z)
\overline{p_3(z)}
-
p_3(-z)
\overline{p_1(z)}
\right\}
\end{pmatrix} \\
 &=
\begin{pmatrix}
\overline{
\mathbbm{i}
\left\{
p_1(-z)
\overline{p_3(z)}
-
p_3(-z)
\overline{p_1(z)}
\right\}
}
 &
\overline{
\mathbbm{i}
\left\{
p_1(-z)
p_3(z)
-
p_3(-z)
p_1(z)
\right\}
}
 \\
\mathbbm{i}
\left\{
p_1(-z)
p_3(z)
-
p_3(-z)
p_1(z)
\right\}
 &
\mathbbm{i}
\left\{
p_1(-z)
\overline{p_3(z)}
-
p_3(-z)
\overline{p_1(z)}
\right\}
\end{pmatrix},
\end{align*}
hence we have
$
\mathbbm{i}
\mathbf{P}^+(-z)
\cdot
\adj{\mathbf{P}^+(z)},
\mathbbm{i}
\mathbf{P}^-(-z)
\cdot
\adj{\mathbf{P}^-(z)}
\in \overline{\pi}(2)
$
by Definition~\ref{definition_pibar}.
Thus by \eqref{equation_vfhvafvjvasf+}, \eqref{equation_vfhvafvjvasf-},
\begin{align*}
\mathbf{P}^+(-z)
\mathbf{P}^+(z)^{-1}
 &=
\mathbf{P}^+(-z)
\cdot
\frac
{
\adj
\mathbf{P}^+(z)
}
{
\det
\mathbf{P}^+(z)
}
 =
-
\frac
{
\mathbbm{i}
\mathbf{P}^+(-z)
\cdot
\adj
\mathbf{P}^+(z)
}
{
2
\Imag
\left\{
\overline{p_0(z)}
p_2(z)
\right\}
}, \\
\mathbf{P}^-(-z)
\mathbf{P}^-(z)^{-1}
 &=
\mathbf{P}^-(-z)
\cdot
\frac
{
\adj
\mathbf{P}^-(z)
}
{
\det
\mathbf{P}^-(z)
}
 =
-
\frac
{
\mathbbm{i}
\mathbf{P}^-(-z)
\cdot
\adj
\mathbf{P}^-(z)
}
{
2
\Imag
\left\{
p_1(z)
\overline{p_3(z)}
\right\}
},
\end{align*}
both of which are in $\overline{\pi}(2)$ by Lemma~\ref{lemma_pibar}.
\end{proof}

Note from \eqref{equation_va;ofvijafov} that
\begin{align}
\lefteqn{
\mathbf{P}(-z)
\mathbf{P}(z)^{-1}
}
\nonumber \\
 &=
\left\{
\sqrt{2}
\cdot
\mathbf{V}^T
\begin{pmatrix}
\mathbf{P}^+(-z) & \mathbf{O} \\
\mathbf{O} & \mathbf{P}^-(-z)
\end{pmatrix}
\hat{\mathbf{V}}^{-1}
\right\}
\cdot
\nonumber \\
 &\qquad
\cdot
\left\{
\sqrt{2}
\cdot
\mathbf{V}^T
\begin{pmatrix}
\mathbf{P}^+(z) & \mathbf{O} \\
\mathbf{O} & \mathbf{P}^-(z)
\end{pmatrix}
\hat{\mathbf{V}}^{-1}
\right\}^{-1}
\nonumber \\
 &=
\sqrt{2}
\cdot
\mathbf{V}^T
\begin{pmatrix}
\mathbf{P}^+(-z) & \mathbf{O} \\
\mathbf{O} & \mathbf{P}^-(-z)
\end{pmatrix}
\hat{\mathbf{V}}^{-1}
\cdot
\frac{1}{\sqrt{2}}
\cdot
\hat{\mathbf{V}}
\begin{pmatrix}
\mathbf{P}^+(z) & \mathbf{O} \\
\mathbf{O} & \mathbf{P}^-(z)
\end{pmatrix}^{-1}
\mathbf{V}
\nonumber \\
 &=
\mathbf{V}^T
\begin{pmatrix}
\mathbf{P}^+(-z) & \mathbf{O} \\
\mathbf{O} & \mathbf{P}^-(-z)
\end{pmatrix}
\begin{pmatrix}
\mathbf{P}^+(z)^{-1} & \mathbf{O} \\
\mathbf{O} & \mathbf{P}^-(z)^{-1}
\end{pmatrix}
\mathbf{V}
\nonumber \\
 &=
\mathbf{V}^T
\begin{pmatrix}
\mathbf{P}^+(-z) \mathbf{P}^+(z)^{-1} & \mathbf{O} \\
\mathbf{O} & \mathbf{P}^-(-z) \mathbf{P}^-(z)^{-1}
\end{pmatrix}
\mathbf{V},
\qquad
z > 0.
\label{equation_a;fvjaofvj}
\end{align}

\begin{proof}[Proof of Lemma~\ref{lemma_Y1/kpibar}]
By Definition~\ref{definition_Ylambda(x)} and Lemma~\ref{lemma_pibar},
it is sufficient to show that
$
\mathbf{X}_\frac{1}{k}(-x)
\mathbf{X}_\frac{1}{k}(x)^{-1}
\in \overline{\pi}(4)
$
for $x > 0$.
By \eqref{equation_z}, \eqref{equation_anvrfuvh;aof},
\begin{align*}
\lefteqn{
\mathbf{X}_\frac{1}{k}(-x)
\mathbf{X}_\frac{1}{k}(x)^{-1}
} \\
 &=
\left\{
\frac{1}{4}
e^{-\mathcal{E} \mathbf{\Omega} (-z)}
\mathbf{P}(-z)
\cdot
\diag
\left( 1, \alpha, \alpha^2, \alpha^3 \right)^{-1}
\right\}
\cdot \\
 &\qquad
\cdot
\left\{
\frac{1}{4}
e^{-\mathcal{E} \mathbf{\Omega} z}
\mathbf{P}(z)
\cdot
\diag
\left( 1, \alpha, \alpha^2, \alpha^3 \right)^{-1}
\right\}^{-1} \\
 &=
\frac{1}{4}
e^{\mathcal{E} \mathbf{\Omega} z}
\mathbf{P}(-z)
\cdot
\diag
\left( 1, \alpha, \alpha^2, \alpha^3 \right)^{-1}
\cdot
4
\diag
\left( 1, \alpha, \alpha^2, \alpha^3 \right)
\cdot
\mathbf{P}(z)^{-1}
e^{\mathcal{E} \mathbf{\Omega} z} \\
 &=
e^{\mathcal{E} \mathbf{\Omega} z}
\mathbf{P}(-z)
\mathbf{P}(z)^{-1}
e^{\mathcal{E} \mathbf{\Omega} z},
\end{align*}
hence it is sufficient to show that
$
\mathbf{P}(-z)
\mathbf{P}(z)^{-1}
\in \overline{\pi}(4)
$
for $z > 0$,
since
$
e^{\mathcal{E} \mathbf{\Omega} z}
\in \overline{\pi}(4)
$
for $z \in \mathbb{R}$.
By \eqref{equation_VVhat}, \eqref{equation_a;fvjaofvj},
\begin{align}
\lefteqn{
\mathbf{R}
\overline{
\left\{
\mathbf{P}(-z)
\mathbf{P}(z)^{-1}
\right\}
}
\mathbf{R}
}
\nonumber \\
 &=
\mathbf{R}
\overline{
\left\{
\mathbf{V}^T
\begin{pmatrix}
\mathbf{P}^+(-z) \mathbf{P}^+(z)^{-1} & \mathbf{O} \\
\mathbf{O} & \mathbf{P}^-(-z) \mathbf{P}^-(z)^{-1}
\end{pmatrix}
\mathbf{V}
\right\}
}
\mathbf{R}
\nonumber \\
 &=
\mathbf{R}
\mathbf{V}^T
\mathbf{R}
\cdot
\mathbf{R}
\overline{
\begin{pmatrix}
\mathbf{P}^+(-z) \mathbf{P}^+(z)^{-1} & \mathbf{O} \\
\mathbf{O} & \mathbf{P}^-(-z) \mathbf{P}^-(z)^{-1}
\end{pmatrix}
}
\mathbf{R}
\cdot
\mathbf{R}
\mathbf{V}
\mathbf{R}.
\label{equation_a;vhafhvao}
\end{align}
By Lemma~\ref{lemma_Ppibar},
\begin{align}
\lefteqn{
\mathbf{R}
\overline{
\begin{pmatrix}
\mathbf{P}^+(-z) \mathbf{P}^+(z)^{-1} & \mathbf{O} \\
\mathbf{O} & \mathbf{P}^-(-z) \mathbf{P}^-(z)^{-1}
\end{pmatrix}
}
\mathbf{R}
}
\nonumber \\
 &=
\begin{pmatrix}
\mathbf{O} & \mathbf{R} \\
\mathbf{R} & \mathbf{O}
\end{pmatrix}
\begin{pmatrix}
\overline{
\mathbf{P}^+(-z) \mathbf{P}^+(z)^{-1}
} &
\mathbf{O} \\
\mathbf{O} & 
\overline{
\mathbf{P}^-(-z) \mathbf{P}^-(z)^{-1}
}
\end{pmatrix}
\begin{pmatrix}
\mathbf{O} & \mathbf{R} \\
\mathbf{R} & \mathbf{O}
\end{pmatrix}
\nonumber \\
 &=
\begin{pmatrix}
\mathbf{O} &
\mathbf{R}
\overline{
\mathbf{P}^-(-z) \mathbf{P}^-(z)^{-1}
} \\
\mathbf{R}
\overline{
\mathbf{P}^+(-z) \mathbf{P}^+(z)^{-1}
} &
\mathbf{O}
\end{pmatrix}
\begin{pmatrix}
\mathbf{O} & \mathbf{R} \\
\mathbf{R} & \mathbf{O}
\end{pmatrix}
\nonumber \\
 &=
\begin{pmatrix}
\mathbf{R}
\overline{
\mathbf{P}^-(-z) \mathbf{P}^-(z)^{-1}
}
\mathbf{R} & 
\mathbf{O} \\
\mathbf{O} & 
\mathbf{R}
\overline{
\mathbf{P}^+(-z) \mathbf{P}^+(z)^{-1}
}
\mathbf{R}
\end{pmatrix}
\nonumber \\
 &=
\begin{pmatrix}
\mathbf{P}^-(-z) \mathbf{P}^-(z)^{-1} & 
\mathbf{O} \\
\mathbf{O} & 
\mathbf{P}^+(-z) \mathbf{P}^+(z)^{-1}
\end{pmatrix},
\qquad
z > 0.
\label{equation_qvhaofvaijva}
\end{align}
By \eqref{equation_VVhat},
\begin{align}
\mathbf{R}
\mathbf{V}
\mathbf{R}
 &=
\begin{pmatrix}
\mathbf{O} & \mathbf{R} \\
\mathbf{R} & \mathbf{O}
\end{pmatrix}
\cdot
\frac{1}{\sqrt{2}}
\begin{pmatrix}
\mathbf{I} & \mathbf{I} \\
-\mathbf{I} & \mathbf{I}
\end{pmatrix}
\cdot
\begin{pmatrix}
\mathbf{O} & \mathbf{R} \\
\mathbf{R} & \mathbf{O}
\end{pmatrix}
 =
\frac{1}{\sqrt{2}}
\begin{pmatrix}
-\mathbf{R} & \mathbf{R} \\
\mathbf{R} & \mathbf{R}
\end{pmatrix}
\begin{pmatrix}
\mathbf{O} & \mathbf{R} \\
\mathbf{R} & \mathbf{O}
\end{pmatrix}
\nonumber \\
 &=
\frac{1}{\sqrt{2}}
\begin{pmatrix}
\mathbf{I} & -\mathbf{I} \\
\mathbf{I} & \mathbf{I}
\end{pmatrix}
 =
\mathbf{V}^T,
\label{equation_asij[asjvav} \\
\mathbf{V}^2
 &=
\frac{1}{\sqrt{2}}
\begin{pmatrix}
\mathbf{I} & \mathbf{I} \\
-\mathbf{I} & \mathbf{I}
\end{pmatrix}
\cdot
\frac{1}{\sqrt{2}}
\begin{pmatrix}
\mathbf{I} & \mathbf{I} \\
-\mathbf{I} & \mathbf{I}
\end{pmatrix}
 =
\begin{pmatrix}
\mathbf{O} & \mathbf{I} \\
-\mathbf{I} & \mathbf{O}
\end{pmatrix}.
\label{equation_sofghbodgj}
\end{align}
So by 
\eqref{equation_a;fvjaofvj}, 
\eqref{equation_a;vhafhvao}, 
\eqref{equation_qvhaofvaijva}, 
\eqref{equation_asij[asjvav},
\eqref{equation_sofghbodgj},
we have
\begin{align*}
\lefteqn{
\mathbf{R}
\overline{
\left\{
\mathbf{P}(-z)
\mathbf{P}(z)^{-1}
\right\}
}
\mathbf{R}
} \\
 &=
\left(
\mathbf{R}
\mathbf{V}
\mathbf{R}
\right)^T
\cdot
\begin{pmatrix}
\mathbf{P}^-(-z) \mathbf{P}^-(z)^{-1} & 
\mathbf{O} \\
\mathbf{O} & 
\mathbf{P}^+(-z) \mathbf{P}^+(z)^{-1}
\end{pmatrix}
\cdot
\mathbf{R}
\mathbf{V}
\mathbf{R} \\
 &=
\mathbf{V}
\begin{pmatrix}
\mathbf{P}^-(-z) \mathbf{P}^-(z)^{-1} & \mathbf{O} \\
\mathbf{O} & \mathbf{P}^+(-z) \mathbf{P}^+(z)^{-1}
\end{pmatrix}
\mathbf{V}^T \\
 &=
\mathbf{V}^T
\mathbf{V}^2
\begin{pmatrix}
\mathbf{P}^-(-z) \mathbf{P}^-(z)^{-1} & \mathbf{O} \\
\mathbf{O} & \mathbf{P}^+(-z) \mathbf{P}^+(z)^{-1}
\end{pmatrix}
\left( \mathbf{V}^2 \right)^T
\mathbf{V} \\
 &=
\mathbf{V}^T
\begin{pmatrix}
\mathbf{O} & \mathbf{I} \\
-\mathbf{I} & \mathbf{O}
\end{pmatrix}
\begin{pmatrix}
\mathbf{P}^-(-z) \mathbf{P}^-(z)^{-1} & \mathbf{O} \\
\mathbf{O} & \mathbf{P}^+(-z) \mathbf{P}^+(z)^{-1}
\end{pmatrix}
\begin{pmatrix}
\mathbf{O} & -\mathbf{I} \\
\mathbf{I} & \mathbf{O}
\end{pmatrix}
\mathbf{V} \\
 &=
\mathbf{V}^T
\begin{pmatrix}
\mathbf{O} & \mathbf{P}^+(-z) \mathbf{P}^+(z)^{-1} \\
-\mathbf{P}^-(-z) \mathbf{P}^-(z)^{-1} & \mathbf{O}
\end{pmatrix}
\begin{pmatrix}
\mathbf{O} & -\mathbf{I} \\
\mathbf{I} & \mathbf{O}
\end{pmatrix}
\mathbf{V} \\
 &=
\mathbf{V}^T
\begin{pmatrix}
\mathbf{P}^+(-z) \mathbf{P}^+(z)^{-1} & \mathbf{O} \\
\mathbf{O} & \mathbf{P}^-(-z) \mathbf{P}^-(z)^{-1}
\end{pmatrix}
\mathbf{V}
 =
\mathbf{P}(-z)
\mathbf{P}(z)^{-1},
\quad
z > 0.
\end{align*}
Thus
$
\mathbf{P}(-z)
\mathbf{P}(z)^{-1}
\in \overline{\pi}(4)
$
for $z > 0$,
and the proof is complete.
\end{proof}


\section{Proof of Lemma~\ref{lemma_Ylambda(x)neqO}}
\label{appendix_Ylambda(x)neqO}

\begin{lemma}
\label{lemma_prep}
Suppose 
$t_0 \in \mathbb{R}$
satisfies
$
\sum_{r=1}^4
\omega_r^n
e^{\omega_r t_0}
 =
0
$
for $n=1,2,3$.
Then $t_0 = 0$.
\end{lemma}

\begin{proof}
Let
$
a = \sum_{r=1}^4 e^{\omega_r t_0}
$.
Then the condition for $t_0$
is equivalent to

\[
\begin{pmatrix}
a \\ 0 \\ 0 \\ 0
\end{pmatrix}
 =
\left(
\omega_j^{i-1}
\right)_{1 \leq i,j \leq 4}
\cdot
\begin{pmatrix}
e^{\omega_1 t_0} \\
e^{\omega_2 t_0} \\
e^{\omega_3 t_0} \\
e^{\omega_4 t_0}
\end{pmatrix}
 =
\mathbf{W}_0
\begin{pmatrix}
e^{\omega_1 t_0} \\
e^{\omega_2 t_0} \\
e^{\omega_3 t_0} \\
e^{\omega_4 t_0}
\end{pmatrix},
\]
which, by Lemma~\ref{lemma_W0} and \eqref{equation_W0*}, is equivalent again to
\[
\begin{pmatrix}
e^{\omega_1 t_0} \\
e^{\omega_2 t_0} \\
e^{\omega_3 t_0} \\
e^{\omega_4 t_0}
\end{pmatrix}
 =
\frac{1}{4}
\mathbf{W}_0^*
\begin{pmatrix}
a \\ 0 \\ 0 \\ 0
\end{pmatrix}
 =
\frac{a}{4}
\left(
\omega_i^{1-j}
\right)_{1 \leq i,j \leq 4}
\cdot
\begin{pmatrix}
1 \\ 0 \\ 0 \\ 0
\end{pmatrix}
 =
\frac{a}{4}
\begin{pmatrix}
1 \\ 1 \\ 1 \\ 1
\end{pmatrix}.
\]
It follows that
$
e^{\omega_i t_0} = e^{\omega_j t_0}
$
for every $i,j \in \mathbb{Z}$.
In particular,
$
e^{\omega_1 t_0} = e^{\omega_2 t_0}
$,
hence
$
1
 =
{e^{\omega_1 t_0}}
/
{e^{\omega_2 t_0}}
 =
e^{\left( \omega_1 - \omega_2 \right) t_0}
 =
e^{\sqrt{2} t_0}
$
by Definition~\ref{definition_omega},
which implies that $t_0 = 0$.
\end{proof}

\begin{proof}[Proof of Lemma~\ref{lemma_Ylambda(x)neqO}]
Suppose
on the contrary
that
$
\mathbf{Y}_{\lambda_0}\left( x_0 \right)
 =
\mathbf{O}
$
for some
$0 \neq \lambda_0 \in \mathbb{C}$ 
and 
$x_0 > 0$
such that $\det\mathbf{X}_{\lambda_0}(x_0) \neq 0$.
Then by Definition~\ref{definition_Ylambda(x)},
we have
$
\mathbf{X}_{\lambda_0}\left( -x_0 \right)
\cdot
\mathbf{X}_{\lambda_0}\left( x_0 \right)^{-1}
 -
\mathbf{I}
 =
\mathbf{O}
$,
hence
$
\mathbf{X}_{\lambda_0}\left( -x_0 \right)
-
\mathbf{X}_{\lambda_0}\left( x_0 \right)
 =
\mathbf{O}
$.
So by \eqref{equation_s;fhgs;fovij},
we have
$
\mathbf{W}\left( x_0 \right)^{-1}
\mathbf{W}_{\lambda_0}\left( x_0 \right)
 =
\mathbf{W}\left( -x_0 \right)^{-1}
\mathbf{W}_{\lambda_0}\left( -x_0 \right)
$,
hence
\begin{equation}
\label{equation_aofvjafv}
\mathbf{W}_{\lambda_0}\left( -x_0 \right)
\mathbf{W}_{\lambda_0}\left( x_0 \right)^{-1}
 =
\mathbf{W}\left( -x_0 \right)
\mathbf{W}\left( x_0 \right)^{-1}.
\end{equation}
Let $z_0 = \alpha x_0 > 0$.
By \eqref{equation_Wxdecomposedz}, \eqref{equation_Wx-1decomposedz},
\begin{align}
\lefteqn{
\mathbf{W}\left( -x_0 \right)
\mathbf{W}\left( x_0 \right)^{-1}
}
\nonumber \\
 &=
\left\{
\diag
\left( 1, \alpha, \alpha^2, \alpha^3 \right)
\cdot
\mathbf{W}_0
e^{\mathbf{\Omega} \left( -z_0 \right)}
\right\}
\left\{
\frac{1}{4}
e^{-\mathbf{\Omega} z_0}
\mathbf{W}_0^*
\cdot
\diag
\left( 1, \alpha, \alpha^2, \alpha^3 \right)^{-1}
\right\}
\nonumber \\
 &=
\frac{1}{4}
\diag
\left( 1, \alpha, \alpha^2, \alpha^3 \right)
\cdot
\mathbf{W}_0
e^{-2 \mathbf{\Omega} z_0}
\mathbf{W}_0^*
\cdot
\diag
\left( 1, \alpha, \alpha^2, \alpha^3 \right)^{-1}.
\label{equation_sogbjs'fgij}
\end{align}
By \eqref{equation_W0*},
\begin{align}
\lefteqn{
\mathbf{W}_0
e^{-2 \mathbf{\Omega} z_0}
\mathbf{W}_0^*
} \\
 &=
\left(
\omega_j^{i-1}
\right)_{1 \leq i,j \leq 4}
\cdot
\diag
\left(
e^{-2 \omega_1 z_0},
e^{-2 \omega_2 z_0},
e^{-2 \omega_3 z_0},
e^{-2 \omega_4 z_0}
\right)
\cdot
\left(
\omega_i^{1-j}
\right)_{1 \leq i,j \leq 4}
\nonumber \\
 &=
\left(
\omega_j^{i-1}
e^{-2 \omega_j z_0}
\right)_{1 \leq i,j \leq 4}
\cdot
\left(
\omega_i^{1-j}
\right)_{1 \leq i,j \leq 4}
 =
\left(
\sum_{r=1}^4
\omega_r^{i-1}
e^{-2 \omega_r z_0}
\omega_r^{1-j}
\right)_{1 \leq i,j \leq 4}
\nonumber \\
 &=
\left(
\sum_{r=1}^4
\omega_r^{i-j}
e^{-2 \omega_r z_0}
\right)_{1 \leq i,j \leq 4},
\label{equation_af'gba'fpbaj}
\end{align}
hence by \eqref{equation_sogbjs'fgij},
\begin{align}
\left\{
\mathbf{W}\left( -x_0 \right)
\mathbf{W}\left( x_0 \right)^{-1}
\right\}_{i,j}
 &=
\frac{\alpha^{i-1} \cdot \alpha^{1-j}}{4}
\cdot
\left(
\mathbf{W}_0
e^{-2 \mathbf{\Omega} z_0}
\mathbf{W}_0^*
\right)_{i,j}
\nonumber \\
 &=
\frac{\alpha^{i-j}}{4}
\cdot
\sum_{r=1}^4
\omega_r^{i-j}
e^{-2 \omega_r z_0},
\qquad
1 \leq i,j \leq 4.
\label{equation_asndfvajdfv}
\end{align}

Suppose $\lambda_0 = 1/k$.
Note from \eqref{equation_Wlambdax} that
$\mathbf{W}_\frac{1}{k}(x)$ is upper diagonal.
So
$
\mathbf{W}_\frac{1}{k}(-x)
\cdot
$
$
\cdot
\mathbf{W}_\frac{1}{k}(x)^{-1}
$
is upper diagonal
as well.
Hence by \eqref{equation_aofvjafv}, \eqref{equation_asndfvajdfv},
we have
$
\sum_{r=1}^4
\omega_r^n
e^{-2 \omega_r z_0}
$
$
 =
\sum_{r=1}^4
\omega_r^n
e^{\omega_r \left( -2 \alpha x_0 \right)}
 =
0
$
for $n = 1,2,3$.
This implies that $x_0 = 0$
by Lemma~\ref{lemma_prep},
which contradicts the assumption that $x_0 > 0$.
Thus we conclude that $\lambda_0 \neq 1/k$.
Let $\kappa_0 = \chi\left( \lambda_0 \right)$,
where $\chi$ is as in Definition~\ref{definition_kappa}.
Note that $\kappa_0^4 \neq 1$.
$\kappa_0 \neq 0$,
since $\lambda_0 \neq 1/k$.
By \eqref{equation_Wlambdaxdecomposedz} and Lemma~\ref{lemma_W0},
\begin{align*}
\lefteqn{
\mathbf{W}_{\lambda_0}\left( -x_0 \right)
\mathbf{W}_{\lambda_0}\left( x_0 \right)^{-1}
} \\
 &=
\left\{
\diag
\left( 1, \alpha, \alpha^2, \alpha^3 \right)
\cdot
\diag
\left( 1, \kappa_0, \kappa_0^2, \kappa_0^3 \right)
\mathbf{W}_0
e^{\mathbf{\Omega} \kappa_0 \left( -z_0 \right)}
\right\}
\cdot \\
 &\quad
\cdot
\left\{
\frac{1}{4}
e^{-\mathbf{\Omega} \kappa_0 z_0}
\mathbf{W}_0^*
\cdot
\diag
\left( 1, \kappa_0, \kappa_0^2, \kappa_0^3 \right)^{-1}
\cdot
\diag
\left( 1, \alpha, \alpha^2, \alpha^3 \right)^{-1}
\right\} \\
 &=
\frac{1}{4}
\diag
\left( 1, \alpha, \alpha^2, \alpha^3 \right)
\cdot
\diag
\left( 1, \kappa_0, \kappa_0^2, \kappa_0^3 \right)
\cdot
\mathbf{W}_0
e^{-2 \mathbf{\Omega} \kappa_0 z_0}
\mathbf{W}_0^*
\cdot \\
 &\quad
\cdot
\diag
\left( 1, \kappa_0, \kappa_0^2, \kappa_0^3 \right)^{-1}
\cdot
\diag
\left( 1, \alpha, \alpha^2, \alpha^3 \right)^{-1},
\end{align*}
hence by \eqref{equation_aofvjafv}, \eqref{equation_sogbjs'fgij},
we have
\begin{align}
\lefteqn{
\mathbf{W}_0
e^{-2 \mathbf{\Omega} z_0}
\mathbf{W}_0^*
}
\nonumber \\
 &=
\diag
\left( 1, \kappa_0, \kappa_0^2, \kappa_0^3 \right)
\cdot
\mathbf{W}_0
e^{-2 \mathbf{\Omega} \kappa_0 z_0}
\mathbf{W}_0^*
\cdot
\diag
\left( 1, \kappa_0, \kappa_0^2, \kappa_0^3 \right)^{-1}.
\label{equation_vafvha;ofvjh}
\end{align}
Similarly to \eqref{equation_af'gba'fpbaj},
we have
\[
\mathbf{W}_0
e^{-2 \mathbf{\Omega} \kappa_0 z_0}
\mathbf{W}_0^*
 =
\left(
\sum_{r=1}^4
\omega_r^{i-j}
e^{-2 \omega_r \kappa_0 z_0}
\right)_{1 \leq i,j \leq 4},
\]
hence by \eqref{equation_af'gba'fpbaj}, \eqref{equation_vafvha;ofvjh},
$
\kappa_0^{i-j}
\sum_{r=1}^4
\omega_r^{i-j}
e^{-2 \omega_r \kappa_0 z_0}
 =
\sum_{r=1}^4
\omega_r^{i-j}
e^{-2 \omega_r z_0}
$
for
$
1 \leq i,j \leq 4
$,
or equivalently,
$
\kappa_0^n
\sum_{r=1}^4
\omega_r^n
e^{-2 \omega_r \kappa_0 z_0}
 =
\sum_{r=1}^4
\omega_r^n
e^{-2 \omega_r z_0}
$
for
$
-3 \leq n \leq 3
$.
So by \eqref{equation_omega},
we have
\begin{align*}
\sum_{r=1}^4
\omega_r^n
e^{-2 \omega_r z_0}
 &=
\kappa_0^n
\sum_{r=1}^4
\omega_r^n
e^{-2 \omega_r \kappa_0 z_0}
 =
\kappa_0^n
\sum_{r=1}^4
\left( -\omega_r^{n-4} \right)
e^{-2 \omega_r \kappa_0 z_0} \\
 &=
-
\kappa_0^4
\cdot
\kappa_0^{n-4}
\sum_{r=1}^4
\omega_r^{n-4}
e^{-2 \omega_r \kappa_0 z_0}
 =
-
\kappa_0^4
\cdot
\sum_{r=1}^4
\omega_r^{n-4}
e^{-2 \omega_r z_0} \\
 &=
-
\kappa_0^4
\cdot
\sum_{r=1}^4
\left( -\omega_r^n \right)
e^{-2 \omega_r z_0} 
 =
\kappa_0^4
\cdot
\sum_{r=1}^4
\omega_r^n
e^{-2 \omega_r z_0},
\quad
n = 1,2,3.
\end{align*}
Since
$
\kappa_0^4
 \neq
1
$,
it follows that
$
\sum_{r=1}^4
\omega_r^n
e^{-2 \omega_r z_0}
 =
\sum_{r=1}^4
\omega_r^n
e^{\omega_r \left( -2 \alpha x_0 \right)}
 =
0
$
for $n = 1,2,3$,
which implies
$
x_0
 =
0$
by Lemma~\ref{lemma_prep}.
This again contradicts the assumption that $x_0 > 0$.
Thus we conclude
$
\mathbf{Y}_\lambda(x)
 \neq
\mathbf{O}
$
for every $0 \neq \lambda \in \mathbb{C}$
and $x > 0$
such that $\det\mathbf{X}_\lambda(x) \neq 0$.
\end{proof}


\begin{thebibliography}{99}


\bibitem{Beaufait1980}
F. W. Beaufait and P. W. Hoadley, 
{\it Analysis of elastic beams on nonlinear foundations},
Comput. Struct. {\bf 12} (1980), no. 5, 669--676.

\bibitem{ChoiBKMS2015}
S. W. Choi,
{\it On positiveness and contractiveness of the integral operator
arising from the beam deflection problem on elastic foundation},
Bull. Korean Math. Soc. {\bf 52} (2015), no. 4, 1225--1240.

\bibitem{ChoiII}
S. W. Choi,
{\it Spectral analysis of the integral operator 
arising from the beam deflection problem on elastic foundation II:
eigenvalues},
Bound. Value Probl. 
{\bf 2015}, 6 
(2015)

\bibitem{Choi2020nonlinear}
S. W. Choi,
{\it Existence and uniqueness of 
finite beam deflection on nonlinear non-uniform elastic foundation
with arbitrary well-posed boundary condition},
Bound. Value Probl.
{\bf 2020}, 113
(2020)

\bibitem{ChoiJang}
S. W. Choi, T. S. Jang,
{\it Existence and uniqueness of nonlinear deflections 
of an infinite beam resting on a non-uniform non-linear elastic foundation},
Bound. Value Probl. 
{\bf 2012}, 5 
(2012)

\bibitem{Galewski2011}
M. Galewski,
{\it On the nonlinear elastic simply supported beam equation},
An. \c{S}tiin\c{t}. Univ. `Ovidius' Constan\c{t}a Ser. Mat. {\bf 19} (2011), no. 1, 109--119.

\bibitem{Hetenyi1946}
M. Hetenyi,
{\it Beams on Elastic Foundation},
The University of Michigan Press, Ann Arbor, Mich.,
1946.

\bibitem{Kuoetal1994}
Y. H. Kuo and S. Y. Lee, 
{\it Deflection of nonuniform beams resting on a nonlinear elastic foundation},
Comput. Struct. {\bf 51} (1994), no. 5, 513--519.

\bibitem{Maetal2009}
X. Ma, J. W. Butterworth, and G. C. Clifton, 
{\it Static analysis of an infinite beam resting on a tensionless Pasternak foundation},
Eur. J. Mech., A, Solids {\bf 28} (2009), no. 4, 697--703.

\bibitem{Mirandaetal1966}
C. Miranda and K. Nair, 
{\it Finite beams on elastic foundation},
ASCE. J. Struct. Div. {\bf 92} (1966), 131--142.

\bibitem{Stakgold}
I. Stakgold and M. Holst,
{\it Green's Functions and Boundary Value Problems},
3rd ed.,
Pure and Applied Mathematics,
John Wiley \& Sons, Inc., Hoboken, NJ,
2011.

\bibitem{Timoshenko1953}
S. Timoshenko,
{\it History of strength of materials. 
{W}ith a brief account of the history of theory of elasticity and theory of structures},
McGraw-Hill Book Company, Inc., New York-Toronto-London,
1953.

\bibitem{Ting1982}
B. Y. Ting, 
{\it Finite beams on elastic foundation with restraints},
ASCE. J. Struct. Div. {\bf 108} (1982), 611--621.





\end{thebibliography}
\end{document}